\documentclass[a4paper,reqno,10pt]{amsart}
    \DeclareMathSizes{12}{12}{7}{6}
\usepackage{geometry}
\usepackage{amssymb,amsmath,mathrsfs,bm}
\usepackage{paralist}
\usepackage[usenames]{color}
\usepackage[all]{xy}
\usepackage{url}
\usepackage{braket}
\usepackage{graphicx}
\usepackage{transparent}
\usepackage{xcolor}
\colorlet{blue}{black}
\colorlet{red}{black}
\colorlet{magenta}{black}
\DeclareRobustCommand{\darkblueedit}[1]{#1}
\usepackage[shortlabels]{enumitem}
\usepackage{mathtools}
\usepackage{tikz}
\usepackage{tikz-cd}

\usepackage{wasysym} %
\usepackage{stmaryrd} %
\makeatletter

\@addtoreset{equation}{section}
\makeatother
\usepackage{amsthm}
\usepackage{aliascnt}
\newtheorem{theorem}{Theorem}[subsection]
\newaliascnt{lemma}{theorem}
\newtheorem{lemma}[lemma]{Lemma}
\aliascntresetthe{lemma}

\newaliascnt{proposition}{theorem}
\newtheorem{proposition}[proposition]{Proposition}
\aliascntresetthe{proposition}

\newaliascnt{corollary}{theorem}
\newtheorem{corollary}[corollary]{Corollary}
\aliascntresetthe{corollary}

\theoremstyle{definition}
\newaliascnt{definition}{theorem}
\newtheorem{definition}[definition]{Definition}
\aliascntresetthe{definition}

\theoremstyle{definition}
\newaliascnt{remark}{theorem}
\newtheorem{remark}[remark]{Remark}
\aliascntresetthe{remark}

\theoremstyle{definition}
\newaliascnt{fact}{theorem}
\newtheorem{fact}[fact]{Fact}
\aliascntresetthe{fact}

\newaliascnt{setup}{theorem}
\theoremstyle{definition}

\aliascntresetthe{setup} 

\newaliascnt{example}{theorem}
\newtheorem{example}[example]{Example}
\aliascntresetthe{example}

\newaliascnt{condition}{theorem}

\aliascntresetthe{condition}

\newaliascnt{construction}{theorem}

\aliascntresetthe{construction}

\newaliascnt{question}{theorem}

\aliascntresetthe{question}

\newaliascnt{conjecture}{theorem}

\aliascntresetthe{conjecture}

\newaliascnt{claim}{theorem}
\newtheorem{claim}[claim]{Claim}
\aliascntresetthe{claim}

\theoremstyle{plain}
\newtheorem{introthm}{Theorem}

\usepackage[hyperfootnotes=false]{hyperref}
\usepackage[capitalise,noabbrev]{cleveref}
\crefname{theorem}{Theorem}{Theorems}
\crefname{lemma}{Lemma}{Lemmas}
\crefname{proposition}{Proposition}{Propositions}
\crefname{corollary}{Corollary}{Corollaries}
\crefname{definition}{Definition}{Definitions}
\crefname{remark}{Remark}{Remarks}
\crefname{example}{Example}{Examples}
\crefname{condition}{Condition}{Conditions}
\crefname{construction}{Construction}{Constructions}
\crefname{claim}{Claim}{Claims}
\crefname{mainthm}{Theorem}{Theorems}
\crefname{maincor}{Corollary}{Corollaries}
\crefname{setup}{Setup}{Setups}
\crefname{fact}{Fact}{Facts}
\crefname{question}{Question}{Questions}
\crefname{conjecture}{Conjecture}{Conjectures}
\crefname{introthm}{Theorem}{Theorems}

\def\Ob{\operatorname{Ob}}

\def\Ab{\mathsf{Ab}}
\def\add{\operatorname{\mathsf{add}}}

\newcommand{\op}{\mathsf{op}}

\newcommand{\Ker}{\operatorname{Ker}}

\newcommand{\id}{\mathsf{id}}
\newcommand{\xto}{\xrightarrow}

\newcommand{\dg}{{\rm dg}}

\newcommand{\lra}{\longrightarrow}
\newcommand{\dra}{\dashrightarrow}

\newcommand{\wtil}[1]{\widetilde{#1}}
\newcommand{\ovl}[1]{\overline{#1}}

\newcommand{\Db}{\mathcal{D}^{\rm b}}

	\newcommand{\sse}{\subseteq}
	
    \newcommand{\fs}{\mathfrak{s}}
    \newcommand{\ft}{\mathfrak{t}}

	\newcommand{\BE}{\mathbb{E}}
	\newcommand{\BF}{\mathbb{F}}

	\newcommand{\BZ}{\mathbb{Z}}

	\newcommand{\CA}{\mathcal{A}}
	\newcommand{\CB}{\mathcal{B}}
	\newcommand{\CC}{\mathcal{C}}
	\newcommand{\CD}{\mathcal{D}}

	\newcommand{\CH}{\mathcal{H}}
	\newcommand{\CI}{\mathcal{I}}
	
	\newcommand{\CK}{\mathcal{K}}
	\newcommand{\CL}{\mathcal{L}}
	
	\newcommand{\CN}{\mathcal{N}}
	
	\newcommand{\CP}{\mathcal{P}}
	
	\newcommand{\CR}{\mathcal{R}}
	\newcommand{\CS}{\mathcal{S}}
	\newcommand{\CT}{\mathcal{T}}
	\newcommand{\CU}{\mathcal{U}}
	\newcommand{\CV}{\mathcal{V}}
	\newcommand{\CW}{\mathcal{W}}

    \newcommand{\A}{\mathscr{A}}

	\renewcommand{\SS}{\mathscr{S}}

\newcommand{\Ar}[3]{\ar[from=#1,to=#2,#3]}

\usetikzlibrary{matrix,arrows,decorations.pathmorphing,positioning,decorations.pathreplacing}
\tikzset{commutative diagrams/.cd, 
mysymbol/.style = {start anchor=center, end anchor = center, draw = none}}
\tikzset{
labl/.style={anchor=north, rotate=90, inner sep=1mm}
}

\let\amph=& %
	
\tikzcdset{every label/.append style = {font = \footnotesize}}

\begin{document}
\setlength{\baselineskip}{15pt}
\title
[Extended heart construction (I)]{Extended heart construction (I): the heart of $n$-cotorsion pairs on triangulated categories}

\author[Mochizuki]{Nao Mochizuki}
	\address{\darkblueedit{Graduate School of Mathematics, Nagoya University, Furo-cho, Chikusa-ku, Nagoya 464-8602, Japan}}
	\email{mochizuki.nao.n8@s.mail.nagoya-u.ac.jp} %

\author[Nakaoka]{Hiroyuki Nakaoka}
	\address{\darkblueedit{Department of Mathematical Sciences, University of the Ryukyus, 1 Senbaru, Nishihara, Okinawa 903-0213, Japan}}
    \email{hiroyuki.nakaoka\_2030@nagoya-u.jp}

\author[Ogawa]{Yasuaki Ogawa}
	\address{Faculty of Engineering Science, Kansai University, Suita-shi, Osaka 564-8680, Japan}
	\email{y\_ogawa@kansai-u.ac.jp} %


\begin{abstract}
\color{blue}
The heart of a $t$-structure and the ideal quotient by a cluster tilting subcategory are classical constructions that produce abelian categories from triangulated categories.
Their higher analogues, namely $n$-extended hearts and ideal quotient categories by $(n+1)$-cluster tilting subcategories, are generally no longer abelian, but are known to carry both pretriangulated and extriangulated structures when the underlying triangulated category is algebraic.

In this article, we introduce the notion of an abelian $n$-truncated category as a common framework for such higher constructions.
We extend the heart construction for cotorsion pairs to $n$-cotorsion pairs on arbitrary triangulated categories, and prove that the resulting extended heart naturally carries compatible pretriangulated and extriangulated structures forming an abelian $n$-truncated category.
This construction simultaneously generalizes the $n$-extended heart of a $t$-structure and the ideal quotient by an $(n+1)$-cluster tilting subcategory.
It may also be regarded as a higher-dimensional generalization of the general heart construction for cotorsion pairs on triangulated categories.

Finally, we show that the heart can be realized as an extriangulated localization of a suitable relative extriangulated structure on the ambient triangulated category.
\normalcolor
\end{abstract}
\maketitle
\tableofcontents

\section*{Introduction}\label{sec:intro}

\color{blue}
Classically, it is well known that the heart $t^{\ge0}\cap t^{\le0}$ of a $t$-structure $(t^{\le0},t^{\ge0})$ on a triangulated category $\CT$ is an abelian category \cite{BBD82}.
Likewise, the ideal quotient of a triangulated category by a $2$-cluster tilting subcategory is also an abelian category. 
This was first proved for $2$-Calabi--Yau triangulated categories by Keller and Reiten \cite{KR07}, and later extended to arbitrary triangulated categories by Koenig and Zhu \cite{KZ08}.

Higher analogues of these constructions arise naturally.
A $t$-structure gives rise to the \emph{$n$-extended heart} $t^{\ge0}\cap t^{\le n-1}$, while an $(n+1)$-cluster tilting subcategory $\CC\subset\CT$ again gives rise to the ideal quotient $\CT/[\CC]$.
In contrast to the classical case, these categories are no longer abelian in general when $n>1$.
Nevertheless, when $\CT$ is algebraic, it follows from \cite{Moc25,Moc26} that they naturally carry both a pretriangulated structure in the sense of Beligiannis--Reiten and an extriangulated structure introduced by H.N. and Palu \cite{NP19}.
This motivates the development of a common framework for describing the structures shared by these two higher constructions without assuming that the underlying triangulated category is algebraic.
In this first paper, we establish the foundations of this framework. An enhanced version of the theory will be developed in a subsequent paper \cite{MNO}.

In recent years, extended hearts associated with $t$-structures and extended module categories have been studied extensively from various perspectives, including silting and tilting theory, torsion pairs, semibricks and wide subcategories, and Auslander--Reiten theory (\cite{Gup24,Zho24,AMP25,GZ25,WZ25,MP26,Run26,HZ26}).
They have also been investigated from the viewpoint of DG-categories, revealing close connections with abelian $n$-truncated DG-categories, Auslander correspondence, and higher cluster tilting theory (\cite{Moc25,Moc26b,Plo26,Moc26}).

To capture the common structures arising in these higher constructions, we first introduce the notion of an \emph{abelian $n$-truncated category}, which may be viewed as a higher-dimensional analogue of an abelian category.
We then extend the heart construction for cotorsion pairs introduced in \cite{Nak11} to the setting of $n$-cotorsion pairs (\cite[Definition~3.1]{HZ22}).
We show that the resulting heart (\cref{def:heart}) naturally carries a pretriangulated structure and an extriangulated structure, which together form an abelian $n$-truncated category.

Our main theorem is the following.
\begin{introthm}(\cref{thm:quasitri})
For any $n$-cotorsion pair $(\CU,\CV)$ on a triangulated category $\CT$, the heart $\CH/[\CW]$ of $(\CU,\CV)$ is equipped with a pretriangulated structure
$(\Sigma,\Omega,\vartriangleright,\vartriangleleft)$ and an extriangulated structure $(\BF,\ft)$, which together form an abelian $n$-truncated category.
\end{introthm}
Besides recovering \cite[Theorem~6.4]{Nak11} when $n=1$, this theorem shows that the constructed extriangulated structure is naturally viewed as part of the higher abelian structure carried by the heart through its compatibility with the pretriangulated structure.
This extended heart construction simultaneously generalizes the $n$-extended heart of a $t$-structure and the quotient category associated with an $(n+1)$-cluster tilting subcategory.
Indeed, if $(t^{\le0},t^{\ge0})$ is a $t$-structure on $\CT$, then $(\CU,\CV)=(t^{\le-1},t^{\ge n})$ forms an $n$-cotorsion pair whose heart agrees with $t^{\ge 0}\cap t^{\le n-1}$.
On the other hand, a full subcategory $\CC\subset\CT$ is an $(n+1)$-cluster tilting subcategory if and only if the pair
$(\CC,\CC)$ is an $n$-cotorsion pair (\cite[Theorem~3.1]{HZ22}).
Thus, $(n+1)$-cluster tilting subcategories may be regarded as degenerate cases of $n$-cotorsion pairs.
In this case the heart agrees with $\CT/[\CC]$.

Section~\ref{section:compatibility} develops the abstract framework needed for the main theorem.
We first formulate what it means for a category to carry \emph{compatible} pretriangulated and extriangulated structures.
We then introduce the notions of an (\emph{$n$-truncated}) \emph{quasi-triangulated} category and an \emph{abelian $n$-truncated} category, which describe situations in which these two structures are compatible in a stronger sense.

A quasi-triangulated category simultaneously generalizes both quasi-abelian categories and triangulated categories; indeed, both arise as special cases of this notion
(see \cref{ex:quasi-tri_tri_abel}). 
More precisely, a quasi-triangulated category consists of an extriangulated structure together with a compatible pretriangulated structure consisting of left and right triangles, which play the roles of kernels and cokernels, respectively.
An abelian $n$-truncated category is then defined to be a quasi-triangulated category satisfying $\Sigma^n=\Omega^n=0$, in which deflations and inflations coincide with $n$-epimorphisms and $n$-monomorphisms, respectively.
When $n=1$, a $1$-truncated quasi-triangulated category is precisely a quasi-abelian category, while an abelian $1$-truncated category is precisely
an abelian category.

Section~\ref{section:heart} is devoted to the heart construction for $n$-cotorsion pairs and culminates in the proof of \cref{thm:quasitri}.
We define the heart $\CH/[\CW]$ of an $n$-cotorsion pair in Section~\ref{subsection:heart}.
When $n=1$, this recovers the heart of a cotorsion pair introduced in \cite{Nak11}, which is known to be an abelian category.
To extend this result to arbitrary positive integers $n$, we first introduce a suitable relative extriangulated structure on $\CT$, from which the heart $\CH/[\CW]$ naturally inherits an extriangulated structure in Section~\ref{subsection:ET_heart}.
To clarify the intrinsic role of the extriangulated structure on the heart, we next construct a pretriangulated structure on $\CH/[\CW]$ in Section~\ref{subsection:PT_heart}, and prove in Section~\ref{subsection:nA_heart} that these two structures are compatible.
In fact, we prove that they together form an abelian $n$-truncated category in our main theorem stated above.

We also expect the notion of an abelian $n$-truncated category to provide an axiomatic framework for the underlying $1$-categorical structures appearing in the homotopy categories of abelian $(n,1)$-categories (\cite[Definition~6.2.4]{Ste23}) and abelian $n$-truncated DG-categories (\cite[Definition~3.12]{Moc25}).
Supporting this expectation, the main theorem together with the known realization theorems implies that the homotopy category $h\mathcal{A}$ of any small abelian $(n,1)$-category $\mathcal{A}$ and the homotopy category $H^0(\A)$ of any small abelian $n$-truncated DG-category $\A$ are both equivalent to abelian $n$-truncated categories (\cref{cor:homotopy_category_abelian_truncated}).

Finally, in Section~\ref{subsec:via_extri_quotient} we apply the results of \cite{Oga24} to describe the heart from the viewpoint of localization.
More precisely, we show that the natural functor $H\colon \CT\to\CH/[\CW]$ can be realized as the extriangulated localization of $\CT$ associated with the relative extriangulated structure, with respect to $\Ker H$.
\bigskip

\normalcolor

\color{blue}
Throughout this article, let $(\CT,[1],\Delta)$ denote a triangulated category. Thus, $\CT$ is an additive category, $[1]\colon\CT\to\CT$ is an auto-equivalence, and $\Delta$ denotes the class of distinguished triangles. Also, let $[-1]$ be a quasi-inverse of the shift functor $[1]$, and we fix natural isomorphisms
\[
\eta^{\CT}\colon \id_{\CT}\overset{\cong}{\Longrightarrow}[-1]\circ [1]
\quad\text{and}\quad
\varepsilon^{\CT}\colon [1]\circ [-1]\overset{\cong}{\Longrightarrow}\id_{\CT},
\]
which are respectively the unit and counit of the adjunction $[1]\dashv[-1]$.
These data will be fixed throughout this 
\color{blue}
article. 
\normalcolor
For a pair of full subcategories $\CA,\CA'\subset\CT$, we denote by $\CA\ast\CA'$ the full subcategory of $\CT$ consisting of all objects $X\in\CT$ for which there exists a distinguished triangle $A\to X\to A'\to A[1]$ in $\CT$ with $A\in\CA$ and $A'\in\CA'$.
For a subcategory $\CA\subset\CT$, we put $\CA_{p}^{q}=\CA[p]*\cdots *\CA[q]$ for any integers $p\leq q$.

\normalcolor

\section{Compatibility between extriangulated and pretriangulated structures}
\label{section:compatibility}

\subsection{Conventions for Triangulated Categories}

\color{blue}

We shall occasionally regard the triangulated category
$(\CT,[1],\Delta)$ as an extriangulated category introduced in
\cite{NP19}, and consider its relative theory (which is equivalent to
considering the proper classes of triangles introduced by Beligiannis
\cite{Bel20}; see, for example, \cite[Appendix~A]{Sak23} for details) and related constructions.
For the definition of extriangulated categories and their basic properties,
we refer the reader to \cite{NP19}, \cite{LN19}, and \cite{INP24}.

When $(\CT,[1],\Delta)$ is regarded as an extriangulated category $(\CT,\BE,\fs)$, the biadditive functor
$\BE\colon\CT^{\op}\times\CT\to\Ab$
is given by
$\BE(C,A)=\CT(C,A[1])$
for all objects $A,C\in\CT$, where $\Ab$ denotes the category of small
abelian groups.
Moreover, for morphisms $a\in\CT(A,A')$ and
$c\in\CT(C',C)$, the induced homomorphisms
\[ a_{\ast}=\BE(C,a)\colon\BE(C,A)\to\BE(C,A'),
\quad
c^{\ast}=\BE(c,A)\colon\BE(C,A)\to\BE(C',A) \]
are given respectively by
\[ a[1]\circ-\colon\CT(C,A[1])\to\CT(C,A'[1]),
\quad
-\circ c\colon\CT(C,A[1])\to\CT(C',A[1]). \]
Finally, for $\delta\in\BE(C,A)=\CT(C,A[1])$, the corresponding
$\fs$-triangle
$A\xrightarrow{f}B\xrightarrow{g}C\overset{\delta}{\dashrightarrow}$
(called an $\BE$-triangle in the terminology of \cite{NP19} and
\cite{LN19}) is induced by the distinguished triangle
$A\xrightarrow{f}B\xrightarrow{g}C\xrightarrow{\delta}A[1]$.

In this article, 
\normalcolor
we will define the \emph{heart} of an \emph{$n$-cotorsion pair} (\cref{def:heart}) and endow it with a pretriangulated structure together with a compatible extriangulated structure. \color{blue}
In this section, 
\normalcolor
we develop the framework needed to formulate the compatibility between these two structures.

\color{blue}
Since we also consider 
\normalcolor
pretriangulated categories in the sense of \cite{BR07} later, we explicitly describe here how the triangulated category $\CT$, together with the above choice of $[-1]$, $\eta^{\CT}$, and $\varepsilon^{\CT}$, can be regarded as a pretriangulated category.
First, a distinguished right triangle is simply a distinguished triangle
$X\xrightarrow{f}Y\xrightarrow{g}Z\xrightarrow{h}X[1]$ in $\CT$, that is, a triangle belonging to $\Delta$.
Next, we define a distinguished left triangle to be a sequence
\[
Z[-1]\xrightarrow{e}X\xrightarrow{f}Y\xrightarrow{g}Z
\]
such that
$X\xrightarrow{f}Y\xrightarrow{g}Z
\xrightarrow{-(e[1])\circ\eta^{\CT}_Z}
X[1]$
is a distinguished triangle in $\CT$.
The minus sign in the last morphism is put so that conditions {\rm (iv)} and {\rm (v)} of \cite[Definition~1.1]{BR07} hold naturally. 
Denote by $\nabla$ the class of distinguished left triangles. 
In other words, a sequence 
$X\xrightarrow{f}Y\xrightarrow{g}Z\xrightarrow{h}X[1]$
is a distinguished triangle in $\CT$ if and only if
$Z[-1]\xrightarrow{-\varepsilon^{\CT}_X\circ(h[-1])}
X\xrightarrow{f}Y\xrightarrow{g}Z$
belongs to $\nabla$.

\begin{remark}
Suppose moreover that the shift functor $[1]$ is an isomorphism of categories and that $[-1]$ is its strict inverse.
\begin{enumerate}
\item
If both $\eta^{\CT}$ and $\varepsilon^{\CT}$ are chosen to be the identity natural transformations, then $\nabla$ consists of sequences of the form
\[
Z[-1]\xrightarrow{-h[-1]}X\xrightarrow{f}Y\xrightarrow{g}Z,
\]
obtained by rotating distinguished triangles $X\xrightarrow{f}Y\xrightarrow{g}Z\xrightarrow{h}X[1]$.
Writing $Z'=Z[-1]$, this sequence becomes
$Z'\xrightarrow{-h[-1]}X\xrightarrow{f}Y\xrightarrow{g}Z'[1]$,
which is again a distinguished triangle.
\item
On the other hand, if both $\eta^{\CT}$ and $\varepsilon^{\CT}$ are chosen to be minus the identity natural transformations, then $\nabla$ consists of sequences of the form
\[
Z[-1]\xrightarrow{h[-1]}X\xrightarrow{f}Y\xrightarrow{g}Z,
\]
arising from distinguished triangles
$X\xrightarrow{f}Y\xrightarrow{g}Z\xrightarrow{h}X[1]$.
Although this way of defining $\nabla$ is also natural, such a sequence is, in general, not itself a distinguished triangle in $\Delta$. Rather, it differs from a distinguished triangle by a sign change in one of its morphisms.
\end{enumerate}
Thus the resulting class $\nabla$ depends on the choice of $\eta^{\CT}$ and $\varepsilon^{\CT}$ and, although our exposition will occasionally favor the convention in {\rm (1)}, the choice between {\rm (1)} and {\rm (2)} is essentially a matter of convention.
\end{remark}

In view of the above, we shall also use the following abbreviated terminology in order to make dual arguments easier to follow.

\begin{definition}
Let $(\CT,[1],\Delta)$ and $[-1],\eta^{\CT},\varepsilon^{\CT}$ be fixed as above.
Throughout this paper, we shall often use the following abbreviated terminology. We shall say that the sequence
\begin{equation}\label{rotated_triangle}
Z[-1]\xrightarrow{e}X\xrightarrow{f}Y\xrightarrow{g}Z
\end{equation}
is a distinguished triangle in $\CT$. More precisely, by this we mean that
$X\xrightarrow{f}Y\xrightarrow{g}Z
\xrightarrow{-(e[1])\circ\eta^{\CT}_Z}
X[1]$
is a distinguished triangle in $\CT$.
Equivalently, (\ref{rotated_triangle}) is a distinguished left triangle when $(\CT,[1],\Delta)$ is regarded as a pretriangulated category via the construction described above.
\end{definition}

\subsection{Compatibility and quasi-triangulated categories}
We define the compatibility between the extriangulated and pretriangulated structures, as well as the notion of a quasi-triangulated category, as follows.
\begin{definition}\label{def:compatible_ET_PT}
Let $\CB$ be an additive category.
\begin{enumerate}
\item Assume that $(\CB,\Sigma,\Omega,\vartriangleright,\vartriangleleft)$ is a pretriangulated category in the sense of \cite{BR07}, with the unit $\eta$ and the counit $\varepsilon$ for the adjunction $\Sigma\dashv\Omega$.
An extriangulated structure $(\BF,\ft)$ on $\CB$ is said to be \emph{compatible with} the pretriangulated structure $(\Sigma,\Omega,\vartriangleright,\vartriangleleft)$ if it has a pair of natural transformations $\partial^-\colon\BF\Rightarrow\CB(\Omega(-),-)$ and $\partial^+\colon\BF\Rightarrow\CB(-,\Sigma(-))$ between functors $\CB^\op\times \CB\to\Ab$ that satisfies the following conditions.
\begin{enumerate}
    \item  For any $\ft$-triangle $A\xrightarrow{f}B\xrightarrow{g}C\overset{\delta}{\dashrightarrow}$, the sequence
\[
A\xrightarrow{f}B\xrightarrow{g}C\xrightarrow{\partial^+_{C,A}(\delta)}\Sigma A
\]
belongs to $\vartriangleright$, and the sequence
\[
\Omega C\xrightarrow{\partial^-_{C,A}(\delta)}A\xrightarrow{f}B\xrightarrow{g}C
\]
belongs to $\vartriangleleft$.
    \item {\color{blue}For any $\delta\in \BF(C,A)$, the equality $\Sigma(\partial^-_{C,A}(\delta))=-\partial^+_{C,A}(\delta)\circ\varepsilon_C$ (or equivalently, $\Omega(\partial^+_{C,A}(\delta))=-\eta_A\circ \partial^-_{C,A}(\delta)$) holds.}
\end{enumerate}
\item Assume that $\CB$ is equipped with a pretriangulated structure $(\Sigma,\Omega,\vartriangleright,\vartriangleleft)$ and an extriangulated structure $(\BF,\ft)$ compatible with it. The category $\CB$ with these structures is called a \emph{quasi-triangulated} category if it moreover satisfies the following conditions.
\begin{itemize}
\item For any $A\xrightarrow{f}B\xrightarrow{g}C\xrightarrow{h}\Sigma A
$ in $\vartriangleright$, the morphism $g$ is a $\ft$-deflation.
\item For any $\Omega C\xrightarrow{e}A\xrightarrow{f}B\xrightarrow{g}C$ in $\vartriangleleft$, the morphism $f$ is a $\ft$-inflation.
\end{itemize}
\end{enumerate}
\end{definition}

\color{blue}
\begin{remark}\label{rem:negative_extensions}
Assume that an extriangulated structure $(\BF,\ft)$ on a
pretriangulated category
$(\CB,\Sigma,\Omega,\vartriangleright,\vartriangleleft)$ is compatible
with the pretriangulated structure. Set $\BF^0:=\CB(-,-)$ and, for each
integer $r\geq 1$, define
\[
\BF^{-r}(C,A):=\CB(C,\Omega^r A)
\cong \CB(\Sigma^r C,A),
\]
where the isomorphism is induced by the adjunction
$\Sigma^r\dashv\Omega^r$. Then $\BF^{-1}$ is a negative first
extension of $(\CB,\BF,\ft)$ in the sense of
\cite[Definition~2.3]{AET23}. Indeed, for a $\ft$-triangle
\[
A\xrightarrow{f}B\xrightarrow{g}C
\overset{\delta}{\dashrightarrow},
\]
the two connecting morphisms required in the definition are
\[
\begin{split}
\BF^{-1}(W,C)=\CB(W,\Omega C)
&\xrightarrow{\CB(W,\partial^-_{C,A}(\delta))}
\CB(W,A),\\
\BF^{-1}(A,W)\cong\CB(\Sigma A,W)
&\xrightarrow{\CB(\partial^+_{C,A}(\delta),W)}
\CB(C,W).
\end{split}
\]
Their required exactness follows from the left and right triangles
associated with the above $\ft$-triangle.

More generally, this construction extends arbitrarily far in the
negative direction. For every $r\geq 1$, the morphisms induced by
$\Omega^{r-1}\partial^-_{C,A}(\delta)$ and
$\Sigma^{r-1}\partial^+_{C,A}(\delta)$ fit into exact sequences
\[
\begin{split}
\BF^{-r}(W,A)&\longrightarrow\BF^{-r}(W,B)
\longrightarrow\BF^{-r}(W,C)
\longrightarrow\BF^{-(r-1)}(W,A),\\
\BF^{-r}(C,W)&\longrightarrow\BF^{-r}(B,W)
\longrightarrow\BF^{-r}(A,W)
\longrightarrow\BF^{-(r-1)}(C,W).
\end{split}
\]
\end{remark}
\normalcolor

\begin{definition}
\label{def:n-trun_pretr}
Let $(\CB,\Sigma,\Omega,\vartriangleright,\vartriangleleft)$ be a
pretriangulated category, and let $n$ be a positive integer.
If $\Sigma^n=0$, or equivalently, if $\Omega^n=0$, then we call $\CB$
\emph{$n$-truncated}.
\end{definition}

\color{blue}

\begin{remark}
A $1$-truncated quasi-triangulated category is precisely a
quasi-abelian category. Indeed, $1$-truncatedness means that
$\Sigma=\Omega=0$, so distinguished right and left triangles reduce
to cokernel and kernel sequences, respectively. The
quasi-triangulated axioms then imply that cokernels are stable under
pullbacks and kernels are stable under pushouts. Conversely, the
canonical exact structure of a quasi-abelian category, together with
the zero suspension and loop functors, yields a $1$-truncated
quasi-triangulated structure.
\end{remark}

\begin{example}\label{ex:quasi-tri_tri_abel}
Triangulated categories and quasi-abelian categories form two basic
classes of quasi-triangulated categories.
\begin{enumerate}
    \item Every triangulated category, endowed with its standard
    pretriangulated and extriangulated structures described above, is
    quasi-triangulated. Notice that a nonzero triangulated category is
    not $n$-truncated for any positive integer $n$, since its suspension
    functor is an autoequivalence.

    \item Every quasi-abelian category, endowed with its canonical exact
    structure \cite[Remark~1.1.11]{Sch99} and the zero suspension and
    loop functors, is a $1$-truncated quasi-triangulated category.
\end{enumerate}
\end{example}

\normalcolor

\begin{definition}
\label{def:n-mono_pretr}
Let $(\CB,\Sigma,\Omega,\vartriangleright,\vartriangleleft)$ be a
pretriangulated category, let $n$ be a positive integer, and let
$f\colon X\to Y$ be a morphism in $\CB$.
\begin{enumerate}
    \item We call $f$ an \emph{$n$-monomorphism} if $\Omega^{n-1}f$ is a
    monomorphism and $\Omega^i f$ is an isomorphism for every
    $i>n-1$.
    \item Dually, we call $f$ an \emph{$n$-epimorphism} if $\Sigma^{n-1}f$ is an
    epimorphism and $\Sigma^i f$ is an isomorphism for every
    $i>n-1$.
\end{enumerate}
\end{definition}

\color{blue}

The following proposition generalizes the fact that every morphism in
a quasi-abelian category admits a factorization into a regular
epimorphism followed by a monomorphism and, dually, a factorization
into an epimorphism followed by a regular monomorphism; see
\cite[Remark~1.1.2(b) and Proposition~1.1.4]{Sch99}.

\begin{proposition}
\label{prop:factorization_quasi_tr}
Let $\CB$ be a quasi-triangulated category with pretriangulated structure $(\Sigma,\Omega,\vartriangleright,\vartriangleleft)$ and extriangulated structure $(\BF,\ft)$. 
Then every morphism
$f\colon X\to Y$ admits factorizations
\[
f=m\circ e=i\circ p,
\]
where $e\colon X\to D$ is a $\ft$-deflation and
$m\colon D\to Y$ is a $1$-monomorphism, while
$p\colon X\to D'$ is a $1$-epimorphism and
$i\colon D'\to Y$ is a $\ft$-inflation.
\end{proposition}
\begin{proof}
We prove only the assertion concerning the factorization
$f=m\circ e$; the assertion concerning $f=i\circ p$ follows by
duality.
Take a left triangle
\[
\Omega Y\xrightarrow{b}K\xrightarrow{k}X\xrightarrow{f}Y.
\]
By the definition of a quasi-triangulated category, $k$ is a $\ft$-inflation. Hence, there is a $\ft$-triangle
\[
K\xrightarrow{k}X\xrightarrow{e}D\overset{\delta}{\dashrightarrow}.
\]
Accordingly, we obtain the following left and right triangles:
\[
\Omega D\xrightarrow{\partial^-(\delta)}K\xrightarrow{k}X\xrightarrow{e}D
, \qquad
K\xrightarrow{k}X\xrightarrow{e}D
\xrightarrow{\partial^+(\delta)}\Sigma K.
\]

Since $f\circ k=0$, there is a morphism $m'\colon D\to Y$ such that $m'\circ e=f$.
The axioms of left triangulated category imply that there is a morphism $\alpha\colon K\to K$ making the following diagram commute:
\[
\begin{tikzcd}[column sep=large]
\Omega D
  \arrow[r,"\partial^-(\delta)"]
  \arrow[d,"\Omega m'"']
&
K
  \arrow[r,"k"]
  \arrow[d,"\alpha"]
&
X
  \arrow[r,"e"]
  \arrow[d,equal]
&
D
  \arrow[d,"m'"]
\\
\Omega Y
  \arrow[r,"b"]
&
K
  \arrow[r,"k"]
&
X
  \arrow[r,"f"]
&
Y.
\end{tikzcd}
\]

Since $k\circ(\id_K-\alpha)=0$, we have a morphism $t\colon K\to\Omega Y$ such that $b\circ t=\id_K-\alpha$. 
Let $r\colon\Sigma K\to Y$ be the morphism corresponding to $t$ under the adjunction $\Sigma\dashv\Omega$. Thus $\Omega r\circ\eta_K=t$ holds,
where $\eta_K\colon K\to\Omega\Sigma K$ is the unit of the adjunction. 
Define $m:=m'-r\circ\partial^+(\delta)\colon D\to Y$.

By the compatibility between the two structures, we have $\Omega\partial^+(\delta)=-\eta_K\circ\partial^-(\delta)$.
It follows that
\begin{align*}
b\circ\Omega m
&=
b\circ\Omega m'
-b\circ\Omega r\circ\Omega\partial^+(\delta)\\
&=
b\circ\Omega m'
+b\circ\Omega r\circ\eta_K\circ\partial^-(\delta)\\
&=
b\circ\Omega m'
+b\circ t\circ\partial^-(\delta)\\
&=
\alpha\circ\partial^-(\delta)
+(\id_K-\alpha)\circ\partial^-(\delta)\\
&=
\partial^-(\delta).
\end{align*}
Moreover, $m\circ e=m'\circ e-r\circ\partial^+(\delta)\circ e=f$, because $\partial^+(\delta)\circ e=0$. Hence, the following diagram is commutative:
\[
\begin{tikzcd}[column sep=large]
\Omega D
  \arrow[r,"\partial^-(\delta)"]
  \arrow[d,"\Omega m"']
&
K
  \arrow[r,"k"]
  \arrow[d,equal]
&
X
  \arrow[r,"e"]
  \arrow[d,equal]
&
D
  \arrow[d,"m"]
\\
\Omega Y
  \arrow[r,"b"]
&
K
  \arrow[r,"k"]
&
X
  \arrow[r,"f"]
&
Y.
\end{tikzcd}
\]

For each $i\geq1$, this morphism of left triangles induces the following commutative diagram, whose rows are exact:
\[
\begin{tikzcd}[column sep=small]
\CB(-,\Omega^iK)
  \arrow[r]
  \arrow[d,equal]
&
\CB(-,\Omega^iX)
  \arrow[r]
  \arrow[d,equal]
&
\CB(-,\Omega^iD)
  \arrow[r]
  \arrow[d,"{\CB(-,\Omega^im)}"]
&
\CB(-,\Omega^{i-1}K)
  \arrow[r]
  \arrow[d,equal]
&
\CB(-,\Omega^{i-1}X)
  \arrow[d,equal]
\\
\CB(-,\Omega^iK)
  \arrow[r]
&
\CB(-,\Omega^iX)
  \arrow[r]
&
\CB(-,\Omega^iY)
  \arrow[r]
&
\CB(-,\Omega^{i-1}K)
  \arrow[r]
&
\CB(-,\Omega^{i-1}X).
\end{tikzcd}
\]
Applying the five lemma objectwise, we see that $\CB(-,\Omega^im)$ is an isomorphism. 
Therefore, by the Yoneda lemma, $\Omega^im$ is an isomorphism for every $i\geq1$. It remains to show that $m$ is a monomorphism.

Take a left triangle
\[
\Omega Y\xrightarrow{p}Q\xrightarrow{q}D\xrightarrow{m}Y.
\]
Since $\Omega m$ is an isomorphism, exactness of
\[
\CB(-,\Omega D)
\xrightarrow{\CB(-,\Omega m)}
\CB(-,\Omega Y)
\xrightarrow{\CB(-,p)}
\CB(-,Q)
\xrightarrow{\CB(-,q)}
\CB(-,D)
\]
shows that $\CB(-,q)$ is a monomorphism. Hence, $q$ is a monomorphism.
Taking a $\ft$-triangle
\begin{equation}\label{tKZQ}
K\xrightarrow{u}Z\xrightarrow{v}Q\overset{q^{\ast}\delta}{\dashrightarrow},
\end{equation}
we obtain the following morphism of $\ft$-triangles:
\[
\begin{tikzcd}[column sep=large]
K
  \arrow[r,"u"]
  \arrow[d,equal]
&
Z
  \arrow[r,"v"]
  \arrow[d,"l"]
&
Q
  \arrow[r,dashed,"q^*\delta"]
  \arrow[d,"q"]
&
{}
\\
K
  \arrow[r,"k"]
&
X
  \arrow[r,"e"]
&
D
  \arrow[r,dashed,"\delta"]
&
{}.
\end{tikzcd}
\]
Then $f\circ l=m\circ e\circ l=m\circ q\circ v=0$ because $m\circ q=0$. Since
\[
\Omega Y\longrightarrow K\xrightarrow{k}X\xrightarrow{f}Y
\]
is a left triangle, there is a morphism $s\colon Z\to K$ such that $k\circ s=l$. Therefore, $q\circ v=e\circ l=e\circ k\circ s=0$. 
Since $q$ is a monomorphism, it follows that $v=0$.

The long exact sequence associated with the left triangle
\[
\Omega Y\xrightarrow{p}Q\xrightarrow{q}D\xrightarrow{m}Y
\]
contains the exact sequence
\[
\CB(-,\Omega^2D)
\xrightarrow{\CB(-,\Omega^2m)}
\CB(-,\Omega^2Y)
\longrightarrow
\CB(-,\Omega Q)
\longrightarrow
\CB(-,\Omega D)
\xrightarrow{\CB(-,\Omega m)}
\CB(-,\Omega Y).
\]
Since both $\Omega m$ and $\Omega^2m$ are isomorphisms, this exact sequence shows that $\CB(-,\Omega Q)=0$. Hence, the Yoneda lemma gives $\Omega Q=0$.

Finally, the left triangle associated with \eqref{tKZQ} %
is
\[
\Omega Q
\xrightarrow{\partial^-(q^*\delta)}
K\xrightarrow{u}Z\xrightarrow{v}Q.
\]
It induces the exact sequence
\[
\CB(-,\Omega Q)
\xrightarrow{\CB(-,\partial^-(q^*\delta))}
\CB(-,K)
\xrightarrow{\CB(-,u)}
\CB(-,Z)
\xrightarrow{\CB(-,v)}
\CB(-,Q).
\]
Since $\Omega Q=0$ and $v=0$, the natural transformation $\CB(-,u)$ is an isomorphism. By the Yoneda lemma, $u$ is an isomorphism. Since \eqref{tKZQ} 
is a $\ft$-triangle, it follows that $Q=0$. Thus, $m$ is a monomorphism and therefore a $1$-monomorphism.
\end{proof}

\begin{proposition}
\label{prop:fac_system}
Let $\CB$ be a quasi-triangulated category with pretriangulated structure $(\Sigma,\Omega,\vartriangleright,\vartriangleleft)$ and extriangulated structure $(\BF,\ft)$. Let $f\colon X\to Y$ be a morphism in $\CB$. If there are two factorizations
\[
f = m_1 \circ e_1 = m_2 \circ e_2,
\qquad
e_i\colon X\to D_i,\quad m_i\colon D_i\to Y
\quad (i=1,2),
\]
where $e_1,e_2$ are $\ft$-deflations and $m_1,m_2$ are
$1$-monomorphisms, then there exists a unique isomorphism
$u\colon D_1\to D_2$ making the following diagram commute:
\[
\begin{tikzcd}
X \arrow[r,"e_1"]\arrow[d,equal] & D_1 \arrow[d,"u"] \arrow[r,"m_1"] & Y \arrow[d,equal] \\
X \arrow[r,"e_2"] & D_2 \arrow[r,"m_2"] & Y.
\end{tikzcd}
\]
Dually, suppose that
\[
f=i_1\circ p_1=i_2\circ p_2,\qquad
p_j\colon X\to D_j',\quad i_j\colon D_j'\to Y
\quad (j=1,2),
\]
where $p_1,p_2$ are $1$-epimorphisms and $i_1,i_2$ are
$\ft$-inflations. Then there exists a unique isomorphism
$v\colon D_1'\to D_2'$ such that
\[
v\circ p_1=p_2
\qquad\text{and}\qquad
i_2\circ v=i_1.
\]
\end{proposition}
\begin{proof}
We prove only the assertion for factorizations into an
$\ft$-deflation followed by a $1$-monomorphism. The assertion for the
dual factorizations follows by duality.
Take a $\ft$-triangle $K_1\xrightarrow{k_1}X\xrightarrow{e_1}D_1\overset{\delta_1}{\dashrightarrow}$.
By the compatibility, this induces a right triangle
\[
K_1\xrightarrow{k_1}X\xrightarrow{e_1}D_1
\xrightarrow{\partial^+(\delta_1)}\Sigma K_1.
\]
Then $m_2\circ e_2\circ k_1=m_1\circ e_1\circ k_1=0$ holds. Since $m_2$ is a $1$-monomorphism, it is in particular a monomorphism. This implies that $e_2\circ k_1=0$. Thus, there exists a morphism $u_0\colon D_1\to D_2$ such that $u_0\circ e_1=e_2$.

Then $(m_2\circ u_0-m_1)\circ e_1=m_2\circ e_2-m_1\circ e_1=0$ holds. Thus, we obtain a morphism $a\colon\Sigma K_1\to Y$ such that
$m_2\circ u_0-m_1=a\circ \partial^+(\delta_1)$. Since $\Omega m_2$ is an isomorphism, the homomorphism
\[
\CB(\Sigma K_1,m_2)\colon
\CB(\Sigma K_1,D_2)\to
\CB(\Sigma K_1,Y)
\]
is an isomorphism by the adjunction $\Sigma\dashv\Omega$.
Thus, there exists a morphism
$b\colon\Sigma K_1\to D_2$ such that $m_2\circ b=a$.
Define a morphism $u\colon D_1\to D_2$ by
\[
u:=u_0-b\circ \partial^+(\delta_1).
\]
Then we have $m_2\circ u=m_2\circ u_0-a\circ \partial^+(\delta_1)=m_1$, and also $u\circ e_1=e_2$ since $\partial^+(\delta_1)\circ e_1=0$.

Since $m_2$ is a monomorphism, the morphism $u$ is uniquely determined by the commutativity of the diagram.  
By similar arguments, we can construct a morphism $v\colon D_2\to D_1$ such that $v\circ e_2=e_1$ and $m_1\circ v=m_2$. Then one can easily check that $v\circ u=\id_{D_1}$ and $u\circ v=\id_{D_2}$,
and this shows that $u$ is an isomorphism.
\end{proof}

\begin{example}
Let $f\colon X\to Y$ be a morphism in $\CB$.
\begin{enumerate}
    \item Suppose that $\CB$ is a triangulated category endowed with
    its standard quasi-triangulated structure. In this case every
    morphism is both a $\ft$-deflation and a $\ft$-inflation, while
    every $1$-monomorphism and every $1$-epimorphism is an
    isomorphism. Hence, the two factorizations in
    \cref{prop:factorization_quasi_tr} are represented respectively by
    \[
    X\xrightarrow{f}Y\xrightarrow{\id_Y}Y
    \qquad\text{and}\qquad
    X\xrightarrow{\id_X}X\xrightarrow{f}Y.
    \]

    \item Suppose that $\CB$ is a quasi-abelian category endowed with
    its canonical $1$-truncated quasi-triangulated structure. In this
    case the $\ft$-deflations are the strict epimorphisms (\cite[Definition~1.1.1]{Sch99}), hence regular
    epimorphisms, and the $\ft$-inflations are the strict
    monomorphisms, hence regular monomorphisms. Moreover, the
    $1$-monomorphisms and $1$-epimorphisms are precisely the
    monomorphisms and epimorphisms, respectively. Thus, the two
    factorizations in \cref{prop:factorization_quasi_tr} are the
    canonical factorizations
    \[
    X\twoheadrightarrow\operatorname{Coim}f
    \rightarrowtail Y
    \qquad\text{and}\qquad
    X\twoheadrightarrow\operatorname{Im}f
    \rightarrowtail Y,
    \]
    where the first morphism in the first factorization is the
    cokernel of $\ker f$, while the second morphism in the second
    factorization is the kernel of $\operatorname{coker}f$; see
    \cite[Remark~1.1.2(b) and Proposition~1.1.4]{Sch99}.
\end{enumerate}
\end{example}

\normalcolor

\begin{lemma}
\label{lem:chara_n-mono_ker}
Let $(\CB,\Sigma,\Omega,\vartriangleright,\vartriangleleft)$ be a
pretriangulated category, and let $n$ be a positive integer.
Consider a left triangle
$$
\Omega C \xrightarrow{h} A
\xrightarrow{f} B
\xrightarrow{g} C.
$$
Then $g$ is an $n$-monomorphism if and only if
$\Omega^{n-1}A=0$.
\end{lemma}

\begin{proof}
For every object $X\in\CB$, the long exact sequence associated with
the given left triangle contains
$$
\begin{aligned}
\CB(X,\Omega^n B)
&\xrightarrow{\CB(X,\Omega^n g)}
\CB(X,\Omega^n C)
\longrightarrow
\CB(X,\Omega^{n-1}A)\\
&\xrightarrow{\CB(X,\Omega^{n-1}f)}
\CB(X,\Omega^{n-1}B)
\xrightarrow{\CB(X,\Omega^{n-1}g)}
\CB(X,\Omega^{n-1}C).
\end{aligned}
$$

Suppose first that $g$ is an $n$-monomorphism. Then
$\Omega^{n-1}g$ is a monomorphism and $\Omega^n g$ is an
isomorphism. By these assumptions, we have an isomorphism $\CB(X,\Omega^n g)$ and monomorphism $\CB(X,\Omega^{n-1}g)$. This implies that $\CB(X,\Omega^{n-1}A)=0$. 
Since this holds for every $X\in\CB$, we conclude that $\Omega^{n-1}A=0$.

Conversely, suppose that $\Omega^{n-1}A=0$. Then
$\Omega^i A=0$ for every $i\geq n-1$. This immediately shows that
$\CB(X,\Omega^{n-1}g)$ is injective for every $X\in\CB$. Hence $\Omega^{n-1}g$ is a monomorphism. Moreover, for every $i> n-1$, the exact sequence
$$
\CB(X,\Omega^i A)
\longrightarrow
\CB(X,\Omega^i B)
\xrightarrow{\CB(X,\Omega^i g)}
\CB(X,\Omega^i C)
\longrightarrow
\CB(X,\Omega^{i-1}A)
$$
has zero outer terms. Thus $\CB(X,\Omega^i g)$ is an isomorphism
for every $X\in\CB$. By the Yoneda lemma, $\Omega^i g$ is an
isomorphism for every $i>n-1$. Therefore $g$ is an
$n$-monomorphism.
\end{proof}

\begin{lemma}
\label{lem:defl_are_n-epi}
Let $n$ be a positive integer, and let $(\CB,\Sigma,\Omega,\vartriangleright,\vartriangleleft)$ be an $n$-truncated pretriangulated category.
Assume that an extriangulated structure $(\BF,\ft)$ on $\CB$ 
\color{blue}
satisfies condition {\rm (a)} in \cref{def:compatible_ET_PT}.
\normalcolor
Then every $\ft$-deflation is an
$n$-epimorphism. Dually, every $\ft$-inflation is an
$n$-monomorphism.
\end{lemma}

\begin{proof}
Let $g\colon B\to C$ be a $\ft$-deflation. Then there exists an
$\ft$-triangle
$$
A\xrightarrow{f}B\xrightarrow{g}C
\overset{\delta}{\dashrightarrow}.
$$
By the compatibility of the extriangulated and pretriangulated
structures, this $\ft$-triangle induces a right triangle
$$
A\xrightarrow{f}B\xrightarrow{g}C
\xrightarrow{\partial^+_{C,A}(\delta)}\Sigma A.
$$
Rotating this right triangle, we obtain a right triangle
$$
B\xrightarrow{g}C
\xrightarrow{\partial^+_{C,A}(\delta)}
\Sigma A\xrightarrow{-\Sigma f}\Sigma B.
$$
Since $\Sigma^n=0$, we have $\Sigma^{n-1}(\Sigma A)=\Sigma^n A=0$. 
Therefore, the dual of Lemma~\ref{lem:chara_n-mono_ker} shows that
$g$ is an $n$-epimorphism. 

The assertion for $\ft$-inflations follows dually.
\end{proof}

\begin{definition}
\label{def:abel_n_trun_cat}
Let $\CB$ be a quasi-triangulated category equipped with the
pretriangulated structure
$(\Sigma,\Omega,\vartriangleright,\vartriangleleft)$
and the extriangulated structure $(\BF,\ft)$, and let $n$ be a
positive integer.
\begin{enumerate}
    \item We call $\CB$ an \emph{$n$-truncated quasi-triangulated
    category} if its underlying pretriangulated category is
    $n$-truncated.

    \item We call $\CB$ an \emph{abelian $n$-truncated category} if it satisfies the following
    conditions:
    \begin{enumerate}
        \item $\CB$ is an $n$-truncated quasi-triangulated category.
        \item The $\ft$-deflations are precisely the
        $n$-epimorphisms.
        \item The $\ft$-inflations are precisely the
        $n$-monomorphisms.
    \end{enumerate}
\end{enumerate}
\end{definition}

\color{blue}

\begin{remark}
\label{rem:abel_1_trun}
An abelian $1$-truncated category is precisely an abelian category.
Indeed, $1$-monomorphisms and $1$-epimorphisms are ordinary
monomorphisms and epimorphisms, respectively, and a $1$-truncated
quasi-triangulated category is quasi-abelian. Thus, the additional
conditions in Definition~\ref{def:abel_n_trun_cat}\textnormal{(2)}
amount to requiring every monomorphism and every epimorphism to be
regular, which is equivalent to $\CB$ being abelian.
\end{remark}

\normalcolor

\begin{remark}
By Lemma~\ref{lem:defl_are_n-epi}, every $\ft$-deflation is an
$n$-epimorphism and every $\ft$-inflation is an $n$-monomorphism if $\CB$ is an $n$-truncated quasi-triangulated category. 
Thus, conditions~\textnormal{(2)(b)} and~\textnormal{(2)(c)} in
Definition~\ref{def:abel_n_trun_cat} amount to requiring the
converse implications.
\end{remark}

\begin{remark}
\label{rem:defl_n_epi_reg_epi}
Let $\CB$ be an abelian $n$-truncated category, and let $f$ be a
morphism in $\CB$. Then the following conditions are equivalent:
\begin{enumerate}
    \item[(i)] $f$ is a $\ft$-inflation.
    \item[(ii)] $f$ fits into a left triangle $\Omega C\to A\xrightarrow{f}B\to C$.
    \item[(iii)] $f$ is an $n$-monomorphism.
\end{enumerate}
The equivalence of (i) and (ii) follows
from the definition of a quasi-triangulated category
(Definition~\ref{def:compatible_ET_PT}\textnormal{(2)}), while the
equivalence of (i) and (iii) follows from
the definition of an abelian $n$-truncated category
(Definition~\ref{def:abel_n_trun_cat}).

Dually, for any morphism $g$ in $\CB$, the following conditions are
equivalent:
\begin{enumerate}
    \item[(i)] $g$ is a $\ft$-deflation.
    \item[(ii)] $g$ fits into a right triangle $A\to B\xrightarrow{g}C\to\Sigma A$.
    \item[(iii)] $g$ is an $n$-epimorphism.
\end{enumerate}
\end{remark}

\color{blue}

\begin{corollary}
\label{cor:n_epi_1_mono_factorization}
Let $\CB$ be an abelian $n$-truncated category. Then every morphism
$f\colon X\to Y$ admits a factorization
\[
f=m\circ e,\qquad
X\xrightarrow{e}D\xrightarrow{m}Y,
\]
where $e$ is an $n$-epimorphism and $m$ is a $1$-monomorphism.
Dually, it admits a factorization
\[
f=i\circ p,\qquad
X\xrightarrow{p}D'\xrightarrow{i}Y,
\]
where $p$ is a $1$-epimorphism and $i$ is an $n$-monomorphism.
Each factorization is unique up to a unique isomorphism of the
intermediate object.
\end{corollary}

\begin{proof}
We prove only the assertion concerning the factorization with an $n$-epimorphism followed by a $1$-monomorphism; the assertion concerning the factorization with a $1$-epimorphism followed by an $n$-monomorphism follows by duality.
By \cref{prop:factorization_quasi_tr}, the morphism $f$ admits a
factorization $f=m\circ e$, where $e$ is a $\ft$-deflation and $m$
is a $1$-monomorphism. By \cref{rem:defl_n_epi_reg_epi}, the morphism
$e$ is an $n$-epimorphism. Conversely, every $n$-epimorphism is an
$\ft$-deflation by the same remark. Therefore, the uniqueness
assertion follows from \cref{prop:fac_system}.
\end{proof}

\normalcolor
\section{The heart of an \texorpdfstring{$n$}{n}-cotorsion pair}\label{section:heart}
Similarly to the previous section, let $(\CT,[1],\Delta)$ be a triangulated category, and fix natural isomorphisms $\eta^{\CT}$ and $\varepsilon^{\CT}$ as the unit and counit of the adjunction $[1]\dashv[-1]$. We also fix a positive integer $n$ throughout the rest of the paper.

If a morphism $X\xto{f}Y$ in $\CT$ factors through some object in a subcategory $\CI\sse\CT$, we also say it \emph{factors through $\CI$} with a slight abuse of terminology. 
We denote by $[\CI]$ the ideal of $\CT$ generated by morphisms that factor through $\CI$. 
\normalcolor

\subsection{Definition of the heart}
\label{subsection:heart}
In this subsection, we will define the \emph{heart} of any $n$-cotorsion pair. 
The notion of an $n$-cotorsion pair (that is, a pair which is both a left and a right $n$-cotorsion pair) was introduced for abelian categories by Huerta, Mendoza, and P\'erez in \cite{HMP21}, and was later generalized to extriangulated categories by He and Zhou in \cite{HZ22}.
Throughout this paper, when we refer to an $n$-cotorsion pair in a triangulated category, we always mean the notion obtained by specializing the definition of \cite{HZ22} to triangulated categories, rather than the weaker notion introduced in \cite{CZ25} (see \cite[Example~3.5]{CZ25} and the preceding paragraph for an explanation of the possible difference between the two notions).

\begin{definition}[\cite{HZ22}, Definition~3.1 specialized to triangulated categories]
Let $\CT$ be a triangulated category, as before.
A pair of subcategories $(\CU,\CV)$ of $\CT$ is called an \emph{$n$-cotorsion pair} if the following conditions are satisfied.
\begin{enumerate}
\item $\CU,\CV\subset\CT$ are closed under direct summands.
\item $\BE^k(\CU,\CV)=0$ for all $1\le k\le n$. Here, $\BE^k(\CU,\CV)=0$ means that $\BE^k(U,V)=0$ for all $U\in\CU$ and $V\in\CV$, where $\BE^k$ denotes the bifunctor $\CT(-,-[k])$.
\item $\CT=\CU\ast\CV_1^n=\CU_{-n}^{-1}\ast\CV$.
\end{enumerate}
\end{definition}

\normalcolor

\begin{definition}\label{def:heart}
Let $(\CU,\CV)$ be an $n$-cotorsion pair on $\CT$. %
We associate the following full subcategories of $\CT$.
\begin{itemize}
\item $\CW=\CU\cap\CV$.
\item $\CT^-=\CU_{-n}^0=\CU[-n]\ast\CU[-n+1]\ast\cdots\ast\CU$.
\item $\CT^+=\CV_0^n=\CV\ast\cdots\ast\CV[n-1]\ast\CV[n]$.
\item $\CH=\CT^-\cap\CT^+$.
\item $\CN=\add(\CU\ast\CV)$.
\end{itemize}
\color{blue}
The \emph{heart} of the $n$-cotorsion pair $(\CU,\CV)$ is defined to be the ideal quotient $\CH/[\CW]$.
\normalcolor
\end{definition}

\begin{remark}
When $n=1$, an $n$-cotorsion pair is nothing but a cotorsion pair on $\CT$ in the sense of \cite[Definition~2.1]{Nak11}. Moreover, its heart coincides with the heart introduced in \cite[Definition~3.7]{Nak11}.
\color{blue}
We also remark that the heart of a cotorsion pair in an extriangulated category has been studied in \cite{LN19}.
\normalcolor
\end{remark}

\color{blue}
\begin{remark}
In \cite{WY25}, Wang and Yao introduced another notion of heart for a twin
\emph{left} $n$-cotorsion pair $((\CU',\CV'),(\CU,\CV))$ (denoted by $((\CS,\CT),(\CU,\CV))$ in their notation) on a triangulated category $\CT$, and proved that the
heart in their sense is preabelian. They also proved that it is abelian
when $(\CU',\CV')=(\CU,\CV)$. If we regard an $n$-cotorsion pair
$(\CU,\CV)$ as the twin \emph{left} $n$-cotorsion pair
$((\CU,\CV),(\CU,\CV))$, then the subcategories
$\CW,\CC^-,\CC^+$ of $\CT$ defined in
\cite[Definition~2.2]{WY25} can be interpreted in our setting as follows.
\begin{itemize}
\item Their $\CW$ coincides with our $\CW$ by
\cref{lem:inclusions_for_n-CT} below.
\item Their $\CC^-$ coincides with our $\CT^-$ by
\cref{lem:equation_T-T+} below.
\item Their $\CC^+$ is equal to $\CU[-1]\ast\CW$.
\end{itemize}
Consequently, the heart $(\CC^-\cap\CC^+)/[\CW]$ they define is a full subcategory of the heart
$\CH/[\CW]$ considered in this article. In general, this inclusion is
strict, except in the case $n=1$.
\end{remark}
\normalcolor

\begin{example}\label{ex:t_structure_n_cotorsion_pair}
The following are examples of $n$-cotorsion pairs.
\begin{enumerate}
\item $(\CU,\CV)=(t^{\le-1},t^{\ge n})$ for any $t$-structure $(t^{\le0},t^{\ge0})$ on $\CT$. In this case, we have $\CW=0$, $\CT^-=t^{\le n-1}$, $\CT^+=t^{\ge 0}$, and $\CH/[\CW]=\CH=t^{\ge 0}\cap t^{\le n-1}$. 
\item $(\CU,\CV)=(\CC,\CC)$ for any $(n+1)$-cluster tilting subcategory $\CC$ of $\CT$. See \cite[Theorem~3.1]{HZ22} for the details. In this case, we have $\CW=\CN=\CC$, $\CT^-=\CT^+=\CH=\CT$, and $\CH/[\CW]=\CT/[\CC]$. 
\end{enumerate}
\end{example}

\begin{lemma}\label{lem:inclusions_for_n-CT}
For any $n$-cotorsion pair $(\CU,\CV)$, the following holds.
\begin{enumerate}
\item For any integer $j$ with $0\le j\le n$, we have $\CU\cap\CV_0^j=\CU_{-j}^0\cap\CV=\CW$.
\item $\CU\subset\CU[-1]\ast\CW\subset\CT^-$ and $\CV\subset\CW\ast\CV[1]\subset\CT^+$ hold. 
\end{enumerate}
\end{lemma}
\begin{proof}
{\rm (1)} $\CU\cap\CV_0^j\supset\CW$ is obvious. Conversely, if $U\in\CU$ belongs to $\CV_0^j$, then there is a distinguished triangle $V\to U\xrightarrow{x}X\to V[1]$ with $V\in\CV$ and $X\in\CV_1^j$. Since $\CT(U,X)=0$, we have $x=0$ and $U$ becomes a direct summand of $V$. Since $\CV\subset\CT$ is closed under taking direct summands, we obtain $U\in\CU\cap\CV=\CW$.

{\rm (2)} $\CU[-1]\ast\CW\subset\CT^-$ is obvious. Let us show $\CU\subset\CU[-1]\ast\CW$. For any $U\in\CU$, if we take a distinguished triangle
\[
U'[-1]\to U\to R\to U'
\]
with $U'\in\CU$ and $R\in\CV_0^{n-1}$, then we have $R\in\CU$ since $\CU\subset\CT$ is extension-closed. Thus, by {\rm (1)}, $R\in\CU\cap\CV_0^{n-1}=\CW$ follows. 
$\CV\subset\CW\ast\CV[1]\subset\CT^+$ can be shown in a dual manner.
\end{proof}

The following lemma will be used later.
\begin{lemma}\label{lem:for_adjoint_pair_+-}
Let $X\xrightarrow{a}W\xrightarrow{b}Y\xrightarrow{c}X[1]$ and $X'\xrightarrow{a'}W'\xrightarrow{b'}Y'\xrightarrow{c'}X'[1]$ be distinguished triangles with $W,W'\in\CW$.
\begin{enumerate}
\item Assume that $\CT(Y[-1],W')=0$ holds. 
\color{blue}Then, 
\normalcolor
for any $x\in\CT(X,X')$, there exists $y_x\in\CT(Y,Y')$ such that $x[1]\circ c=c'\circ y_x$. Moreover, $\overline{y_x}$ is uniquely determined in $\CT/[\CW]$ by $\overline{x}\in(\CT/[\CW])(X,X')$. This gives a well-defined additive group homomorphism
\begin{equation}\label{correspondence_+-1}
(\CT/[\CW])(X,X')\to(\CT/[\CW])(Y,Y')\, ;\, \overline{x}\mapsto\overline{y_x}.
\end{equation}
\item Assume that $\CT(W,X'[1])=0$ holds. %
Dually to {\rm (1)}, we have a well-defined additive group homomorphism
\begin{equation}\label{correspondence_+-2}
(\CT/[\CW])(Y,Y')\to(\CT/[\CW])(X,X')\, ;\, \overline{y}\mapsto\overline{x_y}.
\end{equation}
\item Assume that $\CT(W,X'[1])=0$ and $\CT(Y[-1],W')=0$ hold. Then the maps $(\ref{correspondence_+-1})$ and $(\ref{correspondence_+-2})$ are mutually inverses to each other, hence 
\color{blue}
give
\normalcolor
a bijection.
\end{enumerate}
\end{lemma}
\begin{proof}
{\rm (1)} is straightforward.
\normalcolor
{\rm (2)} can be shown dually.
{\rm (3)} is obvious from {\rm (1)} and {\rm (2)}.
\end{proof}

In the rest of this article, let $(\CU,\CV)$ be an $n$-cotorsion pair.

\begin{lemma}\label{lem:equation_T-T+}
For any integer $j$ with $1\le j\le n$, we have
$\CU_{-j}^0=\CU[-j]\ast\CW_{-j+1}^0$ and $\CV_0^j=\CW_0^{j-1}\ast\CV[j]$.
In particular, we have
$\CT^-=\CU[-n]\ast\CW_{-n+1}^0$ and $\CT^+=\CW_0^{n-1}\ast\CV[n]$.
\end{lemma}
\begin{proof}
$\CU_{-j}^0\supset\CU[-j]\ast\CW_{-j+1}^0$ is obvious. Also,
$\CU_{-j}^0\subset\CU[-j]\ast\CW_{-j+1}^0$ is shown by applying $\CU\subset\CU[-1]\ast\CW$ inductively. Similarly, $\CV_0^j=\CW_0^{j-1}\ast\CV[j]$ follows from $\CV\subset\CW\ast\CV[1]$. 
\end{proof}

\begin{lemma}\label{lem:T-_*_<1>}
The following holds.
\begin{enumerate}
\item $\CT^-\ast\CU_{-n+1}^0=\CT^-$.
\item $\CV_0^{n-1}\ast\CT^+=\CT^+$.
\end{enumerate}
\end{lemma}
\begin{proof}
{\rm (1)} By Lemma~\ref{lem:equation_T-T+}, we have $\CT^-=\CU[-n]\ast\CW_{-n+1}^0$. For any $i,j\in\BZ$ satisfying $0\le i-j\le n-1$, we have
\begin{equation}\label{incl_to_be_iterated}
\CW[i]\ast\CU[j]\subset \CU[j]\ast\CW[i]
\end{equation}
since $\CT(\CU[j],\CW[i+1])=0$. An iterated application of $(\ref{incl_to_be_iterated})$ shows $\CT^-\ast\CU_{-n+1}^0\subset\CT^-$. The reverse inclusion $\CT^-\ast\CU_{-n+1}^0\supset\CT^-$ is obvious.
{\rm (2)} can be shown in a dual manner.
\end{proof}

\begin{lemma}\label{lem:cotors_1-n}
For any $0\le i\le j\le k\le n$, we have
$(\add\CU_i^j)\cap(\add\CV_j^k)=\CU_i^j\cap\CV_j^k=\CW[j]$.
In particular, we have $\CU\cap\CT^+=\CT^-\cap\CV=\CW$.
\end{lemma}
\begin{proof}
$(\add\CU_i^j)\cap(\add\CV_j^k)\supset\CU_i^j\cap\CV_j^k\supset\CW[j]$ is obvious. 

Let $X\in(\add\CU_i^j)\cap(\add\CV_j^k)$ be any object.
Since $X\in\add\CU_i^j$, there exist $X'\in\CT$ and a distinguished triangle
\[
L\xrightarrow{f}X\oplus X'\xrightarrow{[g_1\ g_2]}U[j]\to L[1]
\]
with $L\in\CU_i^{j-1}$ and $U\in\CU$. By $X\in\add\CV_j^k$ we have $\CT(L,X)=0$. Thus the projection $[1\ 0]\colon X\oplus X'\to X$ should satisfy $[1\ 0]\circ f=0$, hence it factors through $[g_1\ g_2]$. This shows that $g_1$ is a split monomorphism, which implies $X\in\add(\CU[j])=\CU[j]$. A dual argument shows $X\in\CV[j]$, hence we obtain $X\in\CU[j]\cap\CV[j]=\CW[j]$. This shows $(\add\CU_i^j)\cap(\add\CV_j^k)\subset\CW[j]$.
\end{proof}

\begin{proposition}\label{prop:cotors_1-n}
For any $0\le k\le {n-1}$, the pair $(\CU_0^k,\CV_k^{n-1})$ is a cotorsion pair.
\end{proposition}
\begin{proof}
Since this is obvious when $n=1$, we assume $n\ge2$.  
First we show $\CT=\CU_0^k\ast\CV_k^{n-1}[1]$. By Lemma~\ref{lem:equation_T-T+} we have $\CU_0^{n-1}=\CU\ast\CW_1^{n-1}$. Thus we obtain
\[
\CT=\CU_0^{n-1}\ast\CV[n]
=(\CU\ast\CW_1^{n-1})\ast\CV[n]
=(\CU\ast\CW_1^k)\ast(\CW_{k+1}^{n-1}\ast\CV[n])
\subset \CU_0^k\ast\CV_{k+1}^n,
\]
hence $\CT=\CU_0^k\ast\CV_{k+1}^n=\CU_0^k\ast\CV_k^{n-1}[1]$.

Next, we show $\CU_0^k={}^{\perp}(\CV_k^{n-1}[1])$. Note that $\CU_0^k\subset {}^{\perp}(\CV_k^{n-1}[1])$ is obvious by $\BE(\CU,\CV[i])=0$ $(1\le i\le n)$. Let us show the converse.
Take any object $X\in {}^{\perp}(\CV_k^{n-1}[1])={}^{\perp}(\CV_{k+1}^n)$. By $\CT=\CU_0^k\ast\CV_{k+1}^n$, there is a distinguished triangle
$
R[-1]\to L\to X\xrightarrow{g}R
$
with $L\in\CU_0^k$ and $R\in\CV_{k+1}^n$.
Since $X\in {}^{\perp}(\CV_{k+1}^n)$ we have $g=0$. Thus the triangle splits, and we have $X\oplus R[-1]\cong L\in \CU_0^k$.
We obtain $R\in\CV_{k+1}^n\cap(\add\CU_1^{k+1})=\CW[k+1]$, by Lemma~\ref{lem:cotors_1-n}. By the isomorphism $X\oplus R[-1]\cong L$, there is also a split triangle $X\to L\to R[-1]\xrightarrow{0}X[1]$. Thus we have
\[
X\in\CW[k-1]\ast\CU_0^k=(\CW[k-1]\ast\CU_0^{k-1})\ast\CU[k]%
\subset\CU_0^k
\]
by $(\ref{incl_to_be_iterated})$, 
which shows $\CU_0^k\supset {}^{\perp}(\CV_k^{n-1}[1])$.
\end{proof}

\begin{remark}
By Lemma~\ref{lem:cotors_1-n} and Proposition~\ref{prop:cotors_1-n},
$\CP=\big((\CU,\CV_0^{n-1}),(\CU_{-n+1}^0,\CV)\big)$ is a concentric twin cotorsion pair in the sense of \cite[Definition 3.3]{Nak18}.
\end{remark}

\begin{corollary}\label{cor:TUT}
Let $L[-1]\to X\to X'\to L$ be a distinguished triangle in $\CT$ with $L\in\CU_{-n+1}^0$. Then $X\in\CT^-$ holds if and only if $X'\in\CT^-$.
\end{corollary}
\begin{proof}
If $X\in\CT^-$, then $X'\in\CT^-\ast\CU_{-n+1}^0=\CT^-$ holds by \darkblueedit{Lemma~\ref{lem:T-_*_<1>}}. Conversely if $X'\in\CT^-$, then
$X\in\CU_{-n}^{-1}\ast\CT^-=\CU_{-n}^{-1}\ast\CU_{-n}^{-1}\ast\CU=\CU_{-n}^{-1}\ast\CU=\CT^-$ holds since $\CU_{-n}^{-1}=\CU_0^{n-1}[-n]\subset\CT$ is extension-closed by Proposition~\ref{prop:cotors_1-n}.
\end{proof}
\normalcolor

\begin{lemma}\label{lem:T-toT-}
Assume that $U[-1]\xrightarrow{u}X\xrightarrow{f}X'\xrightarrow{m}U$ is a distinguished triangle with $U\in\CU$. The following holds.
\begin{enumerate}
\item %
$X\in\CT^-$ if and only if $X'\in\CT^-$.
\normalcolor
\item For any $Y\in\CT^+$, the map
\begin{equation}\label{map_to_be_surjective}
(\CT/[\CW])(\overline{f},Y)\colon(\CT/[\CW])(X',Y)\to(\CT/[\CW])(X,Y)
\end{equation}
is injective.
\item Assume moreover that $u\in[\CU]$ and $X'\in\CT^+$. Then, for any $Y\in\CT^+$, the map
\[
\CT(f,Y)\colon\CT(X',Y)\to\CT(X,Y)
\]
is surjective. Consequently, the map $(\ref{map_to_be_surjective})$ in {\rm (2)} becomes bijective in this case.
\end{enumerate}
\end{lemma}
\begin{proof}
{\rm (1)} This is immediate from Corollary~\ref{cor:TUT}.
\normalcolor

{\rm (2)} By $Y\in\CT^+$, there is a distinguished triangle $W\xrightarrow{a}Y\to Y'\to W[1]$ with $W\in\CW$ and $Y'\in\CV_1^n$. 
Let $e\in\CT(X',Y)$ be any morphism satisfying $\overline{e}\circ\overline{f}=0$. This means that $e\circ f$ factors through some object in $\CW$. Since $Y'\in\CV_1^n$, it implies that there exists $w_1\in\CT(X,W)$ such that $e\circ f=a\circ w_1$. By $\CT(U[-1],W)=0$ we have $w_1\circ u=0$, hence we obtain a morphism $w_2\in\CT(X',W)$ satisfying $w_1=w_2\circ f$. Then, since $(e-a\circ w_2)\circ f=a\circ w_1-a\circ w_1=0$, there exists $e'\in\CT(U,Y)$ such that $e-a\circ w_2=e'\circ m$. By $\CT(U,Y')=0$, there exists $w_3\in\CT(U,W)$ such that $e'=a\circ w_3$. Then we have $e=a\circ(w_2+w_3\circ m)$, which means $\overline{e}=0$ as desired.

{\rm (3)} Similarly as in {\rm (2)}, there is a distinguished triangle $W\to Y\to Y'\to W[1]$ with $W\in\CW$ and $Y'\in\CV_1^n$. Since $U[-1]\xrightarrow{u}X\xrightarrow{f}X'\xrightarrow{m}U$ is a distinguished triangle, 
\[
\CT(X',Y)\xrightarrow{\CT(f,Y)}\CT(X,Y)\xrightarrow{\CT(u,Y)}\CT(U[-1],Y)
\]
is exact.
Let $x\in\CT(X,Y)$ be any morphism.
Since $u\in [\CU]$ and $Y'\in\CV_1^n$, \darkblueedit{the morphism} $x\circ u$ factors through $W$. Thus $x\circ u=0$ follows from $\CT(U[-1],W)=0$. This shows $\CT(u,Y)=0$, hence $\CT(f,Y)$ is surjective by the above exactness. The latter part is obvious.
\end{proof}

\begin{definition}\label{def:functor_+}
For each object $X\in\CT$, we associate an object $X_+\in\CT^+$ in the following steps.
\begin{enumerate}
\renewcommand{\labelenumi}{(\roman{enumi})} 
\item Choose a distinguished triangle 
\begin{equation}\label{triangle_i}
U_X\xrightarrow{u}X\xrightarrow{t}R_X\to U_X[1]
\end{equation}
with $U_X\in\CU$ and $R_X\in\CV_1^n$, by using $\CT\subset\CU\ast\CV_1^n$.
\item Then, choose a distinguished triangle
\begin{equation}\label{triangle_ii}
U'_X[-1]\xrightarrow{u'} U_X\to W_X\to U'_X
\end{equation}
with $U'_X\in\CU$ and $W_X\in\CW$, by using $\CU\subset\CU[-1]\ast\CW$.
\item Complete $u\circ u'$ into a distinguished triangle
\begin{equation}\label{triangle_iii}
U'_X[-1]\xrightarrow{u\circ u'} X\xrightarrow{l_X} X_+\xrightarrow{m_X} U'_X.
\end{equation}
\end{enumerate}
By the octahedron axiom, we obtain a commutative diagram in $\CT$
\[
\begin{tikzpicture}[>=stealth]
\node (1) at (-0.98,1.8) {$U'_X[-1]$};
\node (2) at (-1.19,0) {$U_X$};
\node (3) at (-1.4,-1.7) {$W_X$};
\node (4) at (0,0) {$X$};
\node (5) at (0.45,-0.9) {$X_+$};
\node (6) at (2.4,0) {$R_X$};
\draw[->] (1) -- node[left,font=\scriptsize] {$u'$} (2);
\draw[->] (1) -- node[right,font=\scriptsize] {$u\circ u'$} (4);
\draw[->] (2) -- (3);
\draw[->] (2) -- node[below,font=\scriptsize] {$u$} (4);
\draw[->] (3) -- (5);
\draw[->] (4) -- node[left,font=\scriptsize] {$l_X$} (5);
\draw[->] (4) -- node[above,font=\scriptsize] {$t$} (6);
\draw[->] (5) -- node[below,font=\scriptsize] {$t'$} (6);
\end{tikzpicture}
\]
in which 
\begin{equation}\label{triangle_iv}
W_X\to X_+\xrightarrow{t'} R_X\xrightarrow{w} W_X[1]
\end{equation}
is a distinguished triangle, hence indeed we have $X_+\in\CT^+$.
\end{definition}

\begin{proposition}\label{prop:functor_+}
Let $X\in\CT$ be any object, and let $l_X\colon X\to X_+$ be the morphism in $\CT$ obtained in Definition~\ref{def:functor_+}. 
Then, the pair $(X_+,\overline{l_X})$ is a reflection of $X$ along the inclusion functor $\CT^+/[\CW]\hookrightarrow\CT/[\CW]$, in the sense of \cite[Definition~3.1.1]{Borc94}.
In fact, the following holds for any $Y\in\CT^+$.
\begin{enumerate}
\item $\CT(l_X,Y)\colon\CT(X_+,Y)\to\CT(X,Y)$ is surjective.
\item $(\CT/[\CW])(\overline{l_X},Y)\colon(\CT/[\CW])(X_+,Y)\to(\CT/[\CW])(X,Y)$ is bijective.
\end{enumerate}
\end{proposition}
\begin{proof}
Since $u\circ u'\in [\CU]$ and $X_+\in\CT^+$, this directly follows from Lemma~\ref{lem:T-toT-} {\rm (3)}.
\end{proof}

\begin{corollary}\label{cor:functor_+}
The assignment $X\mapsto X_+$ in Definition~\ref{def:functor_+} gives a functor 
$(\ )_+\colon\CT/[\CW]\to \CT^+/[\CW]$, which is left adjoint to the inclusion $\CT^+/[\CW]\hookrightarrow\CT/[\CW]$. In particular, this functor is additive, and determined uniquely up to a natural isomorphism, regardless of the choices made in Definition~\ref{def:functor_+}.
\end{corollary}
\begin{proof}
As in \cite[Proposition~3.1.3]{Borc94}, this is a formal consequence of Proposition~\ref{prop:functor_+}.
\end{proof}

\begin{proposition}\label{prop:functor_+_property}
Let $X\in\CT$ be any object. With respect to the construction given in Definition~\ref{def:functor_+}, the following holds.
\begin{enumerate}
\item %
$X_+\in \CH$ if and only if $X\in\CT^-$. 
\normalcolor
\item $X_+\in\CW$ if and only if $X\in\CU$. 
\item $X_+\in\CV$ if and only if $X\in\CN$. 
\end{enumerate}
By {\rm (1)}, the restriction of $(\ )_+$ to $\CT^-/[\CW]$ induces an additive functor
\[
(\ )_+\colon\CT^-/[\CW]\longrightarrow\CH/[\CW],
\]
which we denote by the same symbol.
\end{proposition}
\begin{proof}
{\rm (1)} 
Since $X_+\in\CT^+$, this is immediate 
from Lemma~\ref{lem:T-toT-} {\rm (1)}.
\normalcolor

{\rm (2)} Since $(\ref{triangle_iii})$ is a distinguished triangle and $\CU\subset\CT$ is extension-closed, $X\in\CU$ implies $X_+\in\CU$. Then, we obtain $X_+\in\CT^+\cap\CU=\CW$
by \cref{lem:cotors_1-n}.
\normalcolor
Conversely, if $X_+\in\CW$, then we have $t'=0$. Since $t$ in $(\ref{triangle_i})$ factors through $t'$, this forces $t=0$, hence $X$ becomes a direct summand of $U_X$. Since $\CU$ is closed by taking direct summands, we obtain $X\in\CU$.  

{\rm (3)} From the distinguished triangles $(\ref{triangle_i})$ and $(\ref{triangle_iv})$, by the octahedron axiom we obtain a commutative diagram in $\CT$
\[
\begin{tikzpicture}[>=stealth]
\node (1) at (-2.4,0) {$U_X$};
\node (2) at (-0.42,0.88) {$P$};
\node (3) at (1.4,1.7) {$X_+$};
\node (4) at (0,0) {$X$};
\node (5) at (1.2,0) {$R_X$};
\node (6) at (0.98,-1.8) {$W_X[1]$};
\draw[->] (1) -- (2);
\draw[->] (1) -- node[below,font=\scriptsize] {$u$} (4);
\draw[->] (2) -- (3);
\draw[->] (2) -- (4);
\draw[->] (3) -- node[right,font=\scriptsize] {$t'$} (5);
\draw[->] (4) -- node[above,font=\scriptsize] {$t$} (5);
\draw[->] (4) -- node[left,font=\scriptsize] {$w\circ t$} (6);
\draw[->] (5) -- node[right,font=\scriptsize] {$w$} (6);
\end{tikzpicture}
\]
for some $P\in\CT$, in which $U_X\to P\to X_+\to U_X[1]$ and $W_X\to P\to X\xrightarrow{w\circ t}W_X[1]$ are distinguished triangles. Since $w\circ t=w\circ t'\circ l_X=0$, we have $P\cong X\oplus W_X$. Thus if $X_+\in\CV$, it follows $X\in\add(\CU\ast\CV)=\CN$.

For the converse, 
\color{blue}
first we remark that $\CV/[\CW]\subset\CT/[\CW]$ is closed under taking direct summands. Indeed, if $A,A'\in\CT$ satisfy $A\oplus A'\cong V_0$ in $\CT/[\CW]$ for some $V_0\in\CV$, then there exist an object $W_0\in\CW$ and a split monomorphism $s\in\CT(A\oplus A',V_0\oplus W_0)$. Then, since $\CT$ is a triangulated category, this means that $A$ and $A'$ are direct summands of $V_0\oplus W_0\in\CV$.
Since $\CV\subset\CT$ is closed under taking direct summands,
it follows that $A,A'\in\CV$.
This shows that $\CV/[\CW]\subset\CT/[\CW]$ is closed under taking direct summands. 

Since $(\ )_+\colon\CT/[\CW]\to \CT^+/[\CW]$ is an additive functor by Corollary~\ref{cor:functor_+}, it is now enough to prove that 
\normalcolor
$X\in\CU\ast\CV$ implies $X_+\in\CV$.
Suppose that $X\in\CU\ast\CV$ holds, and take a distinguished triangle
\[
U\to X\to V\to U[1]
\]
with $U\in\CU,V\in\CV$. Then, by applying the octahedron axiom to this triangle and $(\ref{triangle_iii})$, we obtain a commutative diagram in $\CT$
\[
\begin{tikzpicture}[>=stealth]
\node (1) at (-1.4,1.7) {$U$};
\node (2) at (-1.2,0) {$X$};
\node (3) at (-0.98,-1.8) {$V$};

\node (4) at (0.4,0.9) {$C$};
\node (5) at (0,0) {$X_+$};
\node (6) at (2.4,0) {$U'_X$};
\draw[->] (1) -- (2);
\draw[->] (1) -- (4);
\draw[->] (2) -- (3);
\draw[->] (2) -- node[above,font=\scriptsize] {$l_X$} (5);
\draw[->] (4) -- (5);
\draw[->] (4) -- (6);
\draw[->] (5) -- node[right,font=\scriptsize] {$h$} (3);
\draw[->] (5) -- node[below,font=\scriptsize] {$m_X$} (6);
\end{tikzpicture}
\]
for some $C\in\CT$, in which $C\to X_+\xrightarrow{h}V\to C[1]$ and $U\to C\to U'_X\to U[1]$ are distinguished triangles. Since $\CU\subset\CT$ is extension-closed, we have $C\in\CU$. Then $\CT(C,R_X)=0$, hence the morphism $t'$ in $(\ref{triangle_iv})$ factors through $h$. For any $P\in\CU_{-n}^{-1}$, we obtain an exact sequence
\[
\CT(P,W_X)\to\CT(P,X_+)\xrightarrow{\CT(P,t')}\CT(P,R_X)
\]
from $(\ref{triangle_iv})$. Since $\CT(P,W_X)=0$ and $\CT(P,t')=0$, it follows $\CT(P,X_+)=0$, which means $X_+\in\CV$.
\end{proof}

\begin{definition}\label{def:functor_-}
\color{blue}
The right adjoint functor $(\ )_-\colon\CT/[\CW]\to\CT^-/[\CW]$ to the inclusion $\CT^-/[\CW]\to\CT/[\CW]$ is defined dually to the construction in Definition~\ref{def:functor_+}. More precisely, for each $X\in\CT$, we associate an object $X_-\in\CT^-$ by the following diagram,
\normalcolor
\[
\begin{tikzpicture}[>=stealth]
\node (1) at (-2.4,0) {$L_X$};
\node (2) at (-0.42,0.88) {$X_-$};
\node (3) at (1.4,1.7) {$W'_X$};
\node (4) at (0,0) {$X$};
\node (5) at (1.2,0) {$V_X$};
\node (6) at (0.98,-1.8) {$V'_X[1]$};
\draw[->] (1) -- (2);
\draw[->] (1) -- (4);
\draw[->] (2) -- (3);
\draw[->] (2) -- node[right,font=\scriptsize] {$r_X$} (4);
\draw[->] (3) -- (5);
\draw[->] (4) -- node[above,font=\scriptsize] {$v$} (5);
\draw[->] (4) -- node[left,font=\scriptsize] {$v'\circ v$} (6);
\draw[->] (5) -- node[right,font=\scriptsize] {$v'$} (6);
\end{tikzpicture}
\]
in which
\begin{eqnarray*}
L_X\to X\xrightarrow{v}V_X\to L_X[1], && V'_X\to W'_X\to V_X\xrightarrow{v'}V'_X[1],\\
V'_X\to X_-\xrightarrow{r_X}X\xrightarrow{v'\circ v}V'_X[1], && L_X\to X_-\to W'_X\to L_X[1]
\end{eqnarray*}
are distinguished triangles, with $V_X,V'_X\in\CV$, $W'_X\in\CW$ and $L_X\in\CU_{-n}^{-1}$.

This gives the functor $(\ )_-\colon\CT/[\CW]\to\CT^-/[\CW]$. Dually to Proposition~\ref{prop:functor_+_property}, it also restricts to give a functor $(\ )_-\colon\CT^+/[\CW]\to\CH/[\CW]$.
\end{definition}

\begin{definition}\label{def:functor_H}
Define an additive functor $H\colon\CT\to\CH/[\CW]$ to be the composite of
\[
\CT\xrightarrow{\textcolor{blue}{\pi}}\CT/[\CW]\xrightarrow{(\ )_+}\CT^+/[\CW]\xrightarrow{(\ )_-}\CH/[\CW],
\]
in which $\textcolor{blue}{\pi}$ is the canonical ideal quotient functor, and $(\ )_-\colon\CT^+/[\CW]\to\CH/[\CW]$ is the functor obtained by the construction in Definition~\ref{def:functor_-}.
\end{definition}

The following lemma tells us that the construction of $H$ is self-dual up to isomorphism.
\begin{lemma}\label{lem:self-dual_H}
\darkblueedit{The functor} $H\colon\CT\to\CH/[\CW]$ in Definition~\ref{def:functor_H} is naturally isomorphic to the composite of
\[
\CT\xrightarrow{\textcolor{blue}{\pi}}\CT/[\CW]\xrightarrow{(\ )_-}\CT^-/[\CW]\xrightarrow{(\ )_+}\CH/[\CW].
\]
\end{lemma}
\begin{proof}
Let $X\in\CT$ be any object. 
As in \cref{def:functor_+,def:functor_-}, we have distinguished triangles
\[ U'_X[-1]\xrightarrow{u\circ u'}X\xrightarrow{l_X} X_+\xrightarrow{m_X}U'_X
\quad\text{and}\quad V'_X\to X_-\xrightarrow{r_X}X\xrightarrow{v'\circ v}V'_X[1], \]
where $U'_X\in\CU$ and $V'_X\in\CV$, and $u\colon U_X\to X$, $v\colon X\to V_X$ are morphisms with $U_X\in\CU, V_X\in\CV$. 
Then, for $X_-$ and $X_+$, we have morphisms $l_{(X_-)}\colon X_-\to (X_-)_+$ and $r_{(X_+)}\colon (X_+)_-\to X_+$ again by \cref{def:functor_+,def:functor_-}. 
\color{blue}
By Proposition~\ref{prop:functor_+}{\rm (2)} and its dual, there exists a unique morphism
\normalcolor
$\mathbf{p}_X\in(\CT/[\CW])((X_-)_+,(X_+)_-)$ that makes
\[
\begin{tikzpicture}[>=stealth]
\node (1) at (-1,0.7) {$X_-$};
\node (2) at (1,0.7) {$X_+$};
\node (3) at (-1,-0.7) {$(X_-)_+$};
\node (4) at (1,-0.7) {$(X_+)_-$};
\draw[->] (1) -- node[above,font=\scriptsize] {$\overline{l_X}\circ\overline{r_X}$} (2);
\draw[->] (1) -- node[left,font=\scriptsize] {$\overline{l_{(X_-)}}$} (3);
\draw[->] (4) -- node[right,font=\scriptsize] {$\overline{r_{(X_+)}}$} (2);
\draw[->] (3) -- node[below,font=\scriptsize] {$\mathbf{p}_X$} (4);
\end{tikzpicture}
\]
commutative in $\CT/[\CW]$. Using its uniqueness, we may check that $\mathbf{p}_X$ is natural in $X\in\CT$, in a straightforward argument. Thus, it is enough to show that such $\mathbf{p}_X$ is an isomorphism in $\CT/[\CW]$.

Since $\CT(U_X,V'_X[1])=0$, there exists $u_0\in\CT(U_X,X_-)$ such that $u=r_X\circ u_0$. Put $c=u_0\circ u'$, and complete it into a distinguished triangle
\begin{equation}\label{triangle_UXP}
U'_X[-1]\xrightarrow{c}X_-\xrightarrow{p_1}P\to U'_X.
\end{equation}
By Lemma~\ref{lem:T-toT-}, we have $P\in\CT^-$.
Since $r_X\circ c=u\circ u'$, by the octahedron axiom we obtain a commutative diagram in $\CT$
\[
\begin{tikzpicture}[>=stealth]
\node (1) at (-2.4,0) {$U'_X[-1]$};
\node (2) at (-0.42,0.88) {$X_-$};
\node (3) at (1.4,1.7) {$P$};
\node (4) at (0,0) {$X$};
\node (5) at (1.2,0) {$X_+$};
\node (6) at (0.98,-1.8) {$V'_X[1]$};
\draw[->] (1) -- node[above, font=\scriptsize] {$c$} (2);
\draw[->] (1) -- node[below,font=\scriptsize] {$u\circ u'$} (4);
\draw[->] (2) -- node[above,font=\scriptsize] {$p_1$} (3);
\draw[->] (2) -- node[right,font=\scriptsize] {$r_X$} (4);
\draw[->] (3) -- node[right,font=\scriptsize] {$p_2$} (5);
\draw[->] (4) -- node[below,font=\scriptsize] {$l_X$} (5);
\draw[->] (4) -- node[left,font=\scriptsize] {$v'\circ v$} (6);
\draw[->] (5) -- node[right, font=\scriptsize] {$d$} (6);
\end{tikzpicture}
\]
in which
\begin{equation}\label{triangle_PXV}
V'_X\to P\xrightarrow{p_2}X_+\xrightarrow{d}V'_X[1]
\end{equation}
is a distinguished triangle.
By the dual of Lemma~\ref{lem:T-toT-}, we have $P\in\CT^+$. Thus it follows $P\in\CT^-\cap\CT^+=\CH$. Besides, we have $d\in [\CV]$. Indeed, since $\CT(U'_X[-1],V_X)=0$, there exists $v_0\in\CT(X_+,V_X)$ such that $v_0\circ l_X=v$. Then we have $(d-v'\circ v_0)\circ l_X=v'\circ v-v'\circ v=0$. Since $U'_X[-1]\xrightarrow{u\circ u'} X\xrightarrow{l_X} X_+\xrightarrow{m_X} U'_X
$ is a distinguished triangle, it follows that $d-v'\circ v_0$ factors through $m_X$. By $\CT(U'_X,V'_X[1])=0$, we obtain $d=v'\circ v_0$, which shows $d\in [\CV]$.

In the distinguished triangle $(\ref{triangle_UXP})$, by construction we have $c\in [\CU]$. Thus, by Lemma~\ref{lem:T-toT-} {\rm (3)}, the pair $(P,\overline{p_1})$ is a reflection of $X_-$ along the inclusion $\CT^+/[\CW]\hookrightarrow\CT/[\CW]$. Since so is $((X_-)_+,\overline{l_{(X_-)}})$, there exists an isomorphism $\mathbf{p}_1\in(\CT/[\CW])((X_-)_+,P)$ such that $\mathbf{p}_1\circ\overline{l_{(X_-)}}=\overline{p_1}$ in $\CT/[\CW]$. Dually, the existence of the distinguished triangle $(\ref{triangle_PXV})$ means that $(P,\overline{p_2})$ is a coreflection of $X_+$ along the inclusion functor $\CT^-/[\CW]\hookrightarrow\CT/[\CW]$, and thus there exists an isomorphism $\mathbf{p}_2\in(\CT/[\CW])(P,(X_+)_-)$ such that $\overline{r_{(X_+)}}\circ\mathbf{p}_2=\overline{p_2}$ in $\CT/[\CW]$.
Then, the diagram in $\CT/[\CW]$
\[
\begin{tikzpicture}[>=stealth]
\node (1) at (-1.6,1) {$X_-$};
\node (2) at (1.6,1) {$X_+$};
\node (3) at (-1.6,-1) {$(X_-)_+$};
\node (4) at (1.6,-1) {$(X_+)_-$};
\node (5) at (0,0) {$P$};
\draw[->] (1) -- node[above,font=\scriptsize] {$\overline{l_X}\circ\overline{r_X}$} (2);
\draw[->] (1) -- node[left,font=\scriptsize] {$\overline{l_{(X_-)}}$} (3);
\draw[->] (4) -- node[right,font=\scriptsize] {$\overline{r_{(X_+)}}$} (2);
\draw[->] (1) -- node[below,font=\scriptsize] {$\overline{p_1}$} (5);
\draw[->] (5) -- node[below,font=\scriptsize] {$\overline{p_2}$} (2);
\draw[->] (3) -- node[above,font=\scriptsize] {$\mathbf{p}_1$} node[below,font=\scriptsize] {$\cong$} (5);
\draw[->] (5) -- node[above,font=\scriptsize] {$\mathbf{p}_2$} node[below,font=\scriptsize] {$\cong$} (4);
\end{tikzpicture}
\]
is commutative. By the uniqueness of $\mathbf{p}_X$, this means $\mathbf{p}_X=\mathbf{p}_2\circ\mathbf{p}_1$. Since $\mathbf{p}_1,\mathbf{p}_2$ are isomorphisms, it follows that $\mathbf{p}_X$ is also an isomorphism in $\CH/[\CW]$. This is what we wanted to show.
\end{proof}

\begin{remark}\label{rem:essentially_surjective_H}
Since $(\ )_+$ and $(\ )_-$ are left/right adjoints to the inclusion functors, we see that the functor $H|_{\CH}\colon\CH\to\CH/[\CW]$ obtained by restricting $H$ to $\CH$ is naturally isomorphic to the canonical quotient functor by the ideal \textcolor{red}{$[\CW]$}. In particular, it follows that $H$ is essentially surjective.
\end{remark}

The following is a corollary of Proposition~\ref{prop:functor_+_property}.
\begin{corollary}\label{cor:Ker_H}
For the functor $H\colon\CT\to\CH/[\CW]$, the full subcategory $\Ker H\subset\CT$ agrees with $\CN$. Namely, the following are equivalent for any $X\in\CT$.
\begin{enumerate}
\item $H(X)\cong 0$ in $\CH/[\CW]$.
\item $X\in\CN$.
\end{enumerate}
\end{corollary}
\begin{proof}
Let $X\in\CT$ be any object. By the dual of Proposition~\ref{prop:functor_+_property} {\rm (2)}, we have $(X_+)_-=H(X)\cong 0$ in $\CH/[\CW]$ if and only if $X_+\in\CV$. Then, by Proposition~\ref{prop:functor_+_property} {\rm (3)}, this is also equivalent to $X\in\CN$.
\end{proof}

\subsection{Extriangulated structure}
\label{subsection:ET_heart}
\subsubsection{Relative theory induced by $\CN$}

Since $\CU\ast\CV\subset\CT$ is extension-closed, so is $\CN\subset\CT$. Thus, we may apply the results in \cite[Section~2]{Oga24} to $\CN\subset\CT$.
In particular, a closed subfunctor $\BE_{\CN}\colon\CT^\op\times\CT\to\Ab$ of $\BE=\CT(-,-[1])\colon\CT^\op\times\CT\to\Ab$ is given in \cite[Proposition~2.1]{Oga24}, which gives an extriangulated category $(\CT,\BE_{\CN},\fs_{\CN})$.

Its definition is as follows.
Recall that an element $h\in\BE(C,A)$ is nothing but a morphism $C\xto{h}A[1]$.

\begin{definition}\label{def:substructure_of_T}
Define $\BE^R_{\CN}$ and $\BE^L_{\CN}$ by setting
\begin{align*}
\BE^L_\CN(Z,X) 
    &= \Set{ h\in\BE(Z,X)  |\, \text{for any}\ x\colon N\to Z \text{ with } N\in\CN\text{, we have } hx\in[\,\CN[1]\,] }\ \text{and}\ \\
\BE^R_\CN(Z,X) 
    &= \Set{ h\in\BE(Z,X) |\, \text{for any}\ y\colon X\to N \text{ with } N\in\CN\text{, we have } y\circ h[-1]\in[\,\CN[-1]\,] }
\end{align*}
for all $X,Z\in\CT$, and put $\BE_\CN=\BE^L_\CN\cap\BE^R_\CN$.
By \cite[Proposition~2.1]{Oga24}, these are closed subfunctors of $\BE$.
We denote the extriangulated categories equipped with the resulting relative structures by
\[
\CT_{\CN}^R=(\CT,\BE_\CN^R,\fs_\CN^R),  \quad\CT_{\CN}^L=(\CT,\BE_\CN^L,\fs_\CN^L),\quad \text{and}\quad \CT_{\CN}=(\CT,\BE_\CN,\fs_\CN), 
\]
respectively.
\end{definition}

Note that the extriangulated structure on $\CN$ induced from $\CT_{\CN}$ coincides with that induced from $\CT$. In other words, we have
$(\CN, \BE|_\CN, \fs|_\CN) = (\CN, \BE_\CN|_\CN, \fs_\CN|_\CN)$.
Moreover, $\CN$ becomes a thick subcategory of $(\CT,\BE_{\CN},\fs_{\CN})$,
as shown in \cite[Corollary~2.8]{Oga24}.
\normalcolor
\darkblueedit{Then the classes of morphisms} $\CL,\CR,\SS_{\CN}$ are associated to $\CN$ as introduced in \cite[Definition~4.3]{NOS22}.
In our situation, \darkblueedit{as shown in \cite{Oga24},} these classes have the following alternative description, which we shall adopt as the definition of $\CL,\CR,\SS_{\CN}$ in this article.
\begin{definition}\label{def:LRS}(\cite[Lemmas~2.10 and~2.12]{Oga24})
The classes of morphisms $\CL,\CR,\SS_{\CN}$ in $\CT$ are given by the following.
\begin{enumerate}
\item A morphism $f\in\CT(X,Y)$ belongs to $\CL$ if and only if there is a distinguished triangle $X\xrightarrow{f}Y\xrightarrow{g}N\xrightarrow{h}X[1]$
such that $N\in\CN$ and $h[-1]\in [\CN]$.
\item Dually, a morphism $g\in\CT(Y,Z)$ belongs to $\CR$ if and only if there is a distinguished triangle $N\xrightarrow{f}Y\xrightarrow{g}Z\xrightarrow{h}N[1]$
such that $N\in\CN$ and $h\in [\CN]$.
\item $\SS_{\CN}=\CR\circ\CL$. In addition, by \cite[Proposition~2.11]{Oga24}, a morphism $g\in\CT(Y,Z)$ belongs to $\SS_{\CN}$ if and only if there is a distinguished triangle
$X\xrightarrow{f}Y\xrightarrow{g}Z\xrightarrow{h}X[1]$
such that $f,h\in [\CN]$.
\end{enumerate}
\end{definition}

The following two lemmas will be used to prove Proposition~\ref{prop:fully_faithful}.
\begin{lemma}\label{lem:fully_faithful1}
For any $X\in\CT$, the following holds.
\begin{enumerate}
\item $l_X\colon X\to X_+$ in Definition~\ref{def:functor_+} belongs to $\CL$.
\item $r_X\colon X_-\to X$ in Definition~\ref{def:functor_-} belongs to $\CR$.
\end{enumerate}
\end{lemma}
\begin{proof}
{\rm (1)} is immediate from the definition of $\CL$, since $U_X,U'_X\in\CU\subset\CN$. Dually for {\rm (2)}. 
\end{proof}

\begin{lemma}\label{lem:fully_faithful2}
Let $f\in\CT(X,X')$ be any morphism. The following holds.
\begin{enumerate}
\item Assume $X\in\CT^-$ and $X'\in\CH$. If there is a distinguished triangle $U[-1]\xrightarrow{u}X\xrightarrow{f}X'\to U$ with $U\in\CU$ and $u\in[\CN]$, then $H(f)$ is an isomorphism in 
\color{blue}
$\CH/[\CW]$.
\normalcolor
\item Assume $X\in\CH$ and $X'\in\CT^+$. If there is a distinguished triangle $V\to X\xrightarrow{f}X'\xrightarrow{v}V[1]$ with $V\in\CV$ and $v\in[\CN]$, then $H(f)$ is an isomorphism in 
\color{blue}
$\CH/[\CW]$.
\normalcolor
\end{enumerate}
\end{lemma}
\begin{proof}
Since {\rm (2)} can be shown in a dual manner, we only show {\rm (1)}.
Note that we have  $H(X)\cong X_+$.  Since $X'\in\CH$, to show that $H(f)$ is an isomorphism, it suffices to show that the pair 
\color{blue}
$(X',\overline{f})$ is a reflection of $X$ 
\normalcolor
along the inclusion $\CH/[\CW]\hookrightarrow\CT^-/[\CW]$. Namely, it is enough to show that
\[
(\CT^-/[\CW])(\overline{f},Y)\colon(\CH/[\CW])(X',Y)\to(\CT^-/[\CW])(X,Y)
\]
is bijective for all $Y\in\CH$. Injectivity is already shown in Lemma~\ref{lem:T-toT-} {\rm (2)}. Let us show the surjectivity.

In fact, we can show the surjectivity of 
$\CT(f,Y)\colon\CH(X',Y)\to\CT^-(X,Y)$
for all $Y\in\CH$. Let 
\color{blue}
$g\in\CT^-(X,Y)$ 
\normalcolor
be any morphism. It is enough to show $g\circ u=0$.
Note that we have $U[-1]\in\CT^-$. By Proposition~\ref{prop:functor_+} {\rm (1)}, there exist some $u'\in\CT((U[-1])_+,X_+)$ and $g'\in\CT(X_+,Y)$ which make
\[
\begin{tikzpicture}[>=stealth]
\node (1) at (-2,0.8) {$U[-1]$};
\node (2) at (0,0.8) {$X$};
\node (3) at (-2,-0.8) {$(U[-1])_+$};
\node (4) at (0,-0.8) {$X_+$};
\node (5) at (1.8,-0.8) {$Y$};
\draw[->] (1) -- node[above,font=\scriptsize] {$u$} (2);
\draw[->] (1) -- node[left,font=\scriptsize] {$l_{U[-1]}$} (3);
\draw[->] (2) -- node[left,font=\scriptsize] {$l_X$} (4);
\draw[->] (3) -- node[below,font=\scriptsize] {$u'$} (4);
\draw[->] (2) -- node[above,font=\scriptsize] {$g$} (5);
\draw[->] (4) -- node[below,font=\scriptsize] {$g'$} (5);
\end{tikzpicture}
\]
commutative in $\CT$. By \darkblueedit{the definition of the functor} $(\ )_+$, we have $\overline{g'}=\overline{g}_+$ and $\overline{u'}=\overline{u}_+$ in $\CH/[\CW]$. Since $u\in [\CN]$ by assumption, we have $H(u)=0$ by Corollary~\ref{cor:Ker_H}.
 Thus we have $H(g'\circ u'\circ l_{U[-1]})=H(g\circ u)=0$, which is equivalent to $\overline{g'\circ u'}=0$. 
Thus $g'\circ u'$ factors through some $W\in\CW$, hence so does $g\circ u$.
Since $\CT(U[-1],W)=0$, we obtain $g\circ u=0$, as desired.
\end{proof}

\begin{proposition}\label{prop:fully_faithful}
Let $X,Y\in\CT$ be any pair of objects. For any $f\in\CT(X,Y)$, the following are equivalent.
\begin{enumerate}
\item $f\in\SS_{\CN}$.
\item $H(f)$ is an isomorphism in $\CH/[\CW]$.
\end{enumerate}
\end{proposition}
\begin{proof}
By Proposition~\ref{prop:functor_+} {\rm (1)} and its dual, we have a commutative diagram
\begin{equation}\label{dia_lr}
\begin{tikzcd}[row sep=0.8cm]
X \arrow{d}[swap]{f}\arrow{r}{l_X}
&X_+ \arrow{d}{f'}
&(X_+)_- \arrow{l}[swap]{r_{X_+}}\arrow{d}{f^{\prime\prime}}
\\
Y \arrow{r}[swap]{l_Y}
&Y_+ 
&(Y_+)_- \arrow{l}{r_{Y_+}}
\end{tikzcd}
\end{equation}
in $\CT$. \darkblueedit{By the definitions of the functors} $(\ )_+$ and $(\ )_-$, we have $\overline{f'}=\overline{f}_+$ and $\overline{f^{\prime\prime}}=\overline{f'}_-=H(f)$ in $\CT/[\CW]$.
By \cref{lem:fully_faithful1}, the horizontal morphisms in \eqref{dia_lr} are in $\SS_{\CN}$.
Since $\SS_{\CN}$ satisfies the $2$-out-of-$3$ condition with respect to the composite of morphisms by \cite[Corollary~2.18]{Oga24}, we have $f\in\SS_{\CN}$ if and only if $f^{\prime\prime}\in\SS_{\CN}$.
On the other hand, by the definition of the functor $H$ and by \cref{lem:self-dual_H}, the horizontal arrows in \eqref{dia_lr} are sent to isomorphisms in $\CH/[\CW]$ by the functor $H$. Thus $H(f)$ is an isomorphism in $\CH/[\CW]$ if and only if $H(f^{\prime\prime})$ is so.
Therefore, by replacing $X$ and $Y$ with $(X_+)_-$ and $(Y_+)_-$ if necessary, we may assume $X,Y\in\CH$ from the beginning.

$(1)\Rightarrow(2)$ 
Assume $X,Y\in\CH$.
\normalcolor
Since $\SS_{\CN}=\CR\circ\CL$, \darkblueedit{the morphism} $f$ can be written as a composite of $X\xrightarrow{l}M\xrightarrow{r}Y$ for some $M$, with $l\in\CL$ and $r\in\CR$. It suffices to show that both $H(l)$ and $H(r)$ are isomorphisms in $\CH/[\CW]$. As a dual argument works for $H(l)$, we shall only show that $H(r)$ is an isomorphism in $\CH/[\CW]$ for any morphism $r\colon M\to Y$ in $\CR$ with $Y\in\CH$.

Let $r\colon M\to Y$ in $\CR$ 
\color{blue}
be
\normalcolor
any morphism with $Y\in\CH$. 
\color{blue}
As in Definition~\ref{def:LRS}, there is a distinguished triangle
\begin{equation}\label{triangle_NMY}
N\xrightarrow{e}M\xrightarrow{r}Y\xrightarrow{g}N[1]
\end{equation}
with $N\in\CN$ and $g\in [\CN]$. 
Since $\CN=\add(\CU\ast\CV)$, there exists $K\in\CN$ such that $N\oplus K\in\CU\ast\CV$.
Then, from~$(\ref{triangle_NMY})$, we obtain a distinguished triangle
\[
N\oplus K\xrightarrow{e\oplus\id_K}M\oplus K\xrightarrow{[r\ 0]}Y\xrightarrow{g'} (N\oplus K)[1],
\]
in which $g'\in [\CN]$. Since $H(K)\cong 0$ in $\CH/[\CW]$ by Corollary~\ref{cor:Ker_H}, if $H([r\ 0])$ is an isomorphism in $\CH/[\CW]$, then so is $H(r)$. Thus, replacing $N,M$ with $N\oplus K,M\oplus K$ respectively, we may assume that $N\in\CU\ast\CV$ holds in $(\ref{triangle_NMY})$ from the beginning.

Then, take a distinguished triangle $U_0\xrightarrow{u_0}N\xrightarrow{v_0}V_0\to U_0[1]$ with $U_0\in\CU$ and $V_0\in\CV$.
If we take $r_M\colon M_-\to M$ in Definition~\ref{def:functor_-}, then $r_M\in\CR$ by Lemma~\ref{lem:fully_faithful1} {\rm (2)}.
By Lemma~\ref{lem:self-dual_H}, we know that $H(r_M)$ is an isomorphism in $\CH/[\CW]$. Moreover, since $\CR$ is closed by compositions, we have $r\circ r_M\in\CR$. 
Complete $r\circ r_M$ into a distinguished triangle
\[
Z\to M_-\xrightarrow{r\circ r_M}Y\xrightarrow{g'}Z[1].
\]
By $r\circ r_M\in\CR$, we have $g'\in[\CN]$.
By the octahedron axiom, we obtain a commutative diagram
\[
\begin{tikzpicture}[>=stealth]
\node (1) at (-0.98,1.8) {$M_-$};
\node (2) at (-1.19,0) {$M$};
\node (3) at (-1.4,-1.7) {\darkblueedit{$V'_M[1]$}};
\node (4) at (0,0) {$Y$};
\node (5) at (0.45,-0.9) {$Z[1]$};
\node (6) at (2.4,0) {$N[1]$};
\draw[->] (1) -- node[left,font=\scriptsize] {$r_M$} (2);
\draw[->] (1) -- node[right,font=\scriptsize] {$r\circ r_M$} (4);
\draw[->] (2) -- (3);
\draw[->] (2) -- node[below,font=\scriptsize] {$r$} (4);
\draw[->] (3) -- (5);
\draw[->] (4) -- node[left,font=\scriptsize] {$g'$} (5);
\draw[->] (4) -- node[above,font=\scriptsize] {$g$} (6);
\draw[->] (5) -- node[below,font=\scriptsize] {$z[1]$} (6);
\end{tikzpicture}
\]
in $\CT$, in which $V'_M\to Z\xrightarrow{z}N\xrightarrow{z'}V'_M[1]$ is a distinguished triangle. Since $z'\circ u_0=0$ by $\CT(U_0,V'_M[1])=0$, by the octahedron axiom we obtain a commutative diagram
\[
\begin{tikzpicture}[>=stealth]
\node (1) at (0.98,1.8) {$U_0$};
\node (2) at (1.19,0) {$N$};
\node (3) at (1.4,-1.7) {$V_0$};
\node (4) at (0,0) {$Z$};
\node (5) at (-0.45,-0.9) {$V_0'$};
\node (6) at (-2.4,0) {$V'_M$};
\draw[->] (1) -- node[right,font=\scriptsize] {$u_0$} (2);
\draw[->] (1) -- node[right,font=\scriptsize] {} (4);
\draw[->] (2) -- node[right,font=\scriptsize] {$v_0$} (3);
\draw[->] (4) -- node[below,font=\scriptsize] {} (2);
\draw[->] (5) -- node[below,font=\scriptsize] {} (3);
\draw[->] (4) -- node[left,font=\scriptsize] {} (5);
\draw[->] (6) -- node[above,font=\scriptsize] {} (4);
\draw[->] (6) -- node[below,font=\scriptsize] {} (5);
\end{tikzpicture}
\]
in $\CT$ for some $V_0'\in \CT$, in which 
$V'_M\to V'_0\to V_0\to V'_M[1]$ and $U_0\to Z\to V'_0\to U_0[1]$ are distinguished triangles. Since $\CV\subset\CT$ is extension-closed, we have $V'_0\in\CV$.

Thus, by replacing $r$ with $r\circ r_M$ in \eqref{triangle_NMY} if necessary, we may assume $M\in\CT^-$ and $N\in\CU\ast\CV$ from the beginning.
\normalcolor
In this case, since $\CU\ast\CV\subset(\CU[-1]\ast\CW)\ast\CV=\CU[-1]\ast\CV$, there exists a distinguished triangle
$
U[-1]\xrightarrow{u}N\xrightarrow{v}V\to U
$
with $U\in\CU$ and $V\in\CV$. By the octahedron axiom, we obtain a commutative diagram in $\CT$
\[
\begin{tikzcd}[row sep=0.8cm]
U[-1]\arrow[equal]{d}{}\arrow{r}{u}
&N \arrow{r}{v}\arrow{d}[swap]{e}
&V \arrow{r}{}\arrow{d}{}
&U \arrow[equal]{d}{}
\\
U[-1] \arrow{r}[swap]{e\circ u}
&M \arrow{r}{y}
\arrow{d}[swap]{r}
&Y' \arrow{r}{}\arrow{d}{y'}
&U \arrow{d}{}\\
{}
&Y \arrow[equal]{r}{}\arrow{d}[swap]{g}
&Y \arrow{r}[swap]{g}\arrow{d}{g'}
&N[1]
\\
{}
&N[1] \arrow{r}
&V[1]
\end{tikzcd}
\]
for some $Y'\in\CT$, in which $U[-1]\xrightarrow{e\circ u}M\xrightarrow{y}Y'\to U$ and $V\to Y'\xrightarrow{y'}Y\xrightarrow{g'}V[1]$ are distinguished triangles. Since $M\in\CT^-$, we have $Y'\in\CT^-$ by Lemma~\ref{lem:T-toT-} {\rm (1)}. Dually, we have $Y'\in\CT^+$ since $Y\in\CH$. Thus $Y'\in\CT^-\cap\CT^+=\CH$ follows.
By Lemma~\ref{lem:fully_faithful2} {\rm (1)}, we see that $H(y)$ is an isomorphism in $\CH/[\CW]$. Since $g\in [\CN]$, we have $g'\in [\CN]$. Therefore, by Lemma~\ref{lem:fully_faithful2} {\rm (2)}, we see that $H(y')$ is also an isomorphism in $\CH/[\CW]$. Thus, $H(r)=H(y')\circ H(y)$ becomes an isomorphism in $\CH/[\CW]$.

$(2)\Rightarrow(1)$ 
Assume $X,Y\in\CH$. By assumption $H(f)$ is an isomorphism in $\CH/[\CW]$, which is equivalent to that $\overline{f}$ is an isomorphism in $\CH/[\CW]$. Let $g\in\CT(Y,X)$ \darkblueedit{be a morphism such that} $\overline{g}$ gives the inverse of $\overline{f}$ in $\CH/[\CW]$. By $\overline{f}\circ\overline{g}=\id_Y$, there exist some $W\in\CW$ and \darkblueedit{$w_1\in\CT(Y,W),\;w_2\in\CT(W,Y)$} such that $\id_Y-f\circ g=w_2\circ w_1$ in $\CT$.
Complete $[f\ w_2]\colon X\oplus W\to Y$ into a split distinguished triangle
\[
K\xrightarrow{\begin{bsmallmatrix}
k_1\\
k_2
\end{bsmallmatrix}}
X\oplus W\xrightarrow{[f\ w_2]}Y\xrightarrow{0}K[1]
\]
in $\CT$. By $f\circ k_1+w_2\circ k_2=0$ in $\CT$, we obtain $\overline{k_1}=\overline{g}\circ\overline{f}\circ\overline{k_1}=-\overline{g}\circ\overline{w_2}\circ\overline{k_2}=0$ in $\CT/[\CW]$. Thus $\begin{bsmallmatrix}
k_1\\
k_2
\end{bsmallmatrix}\in [\CW]\subset [\CN]$, hence $[f \ w_2]\in\SS_{\CN}$ follows.
Since $\begin{bsmallmatrix}
1\\
0
\end{bsmallmatrix}\colon X\to X\oplus W$ belongs to $\CL$, we obtain $f=[f \ w_2]\circ\begin{bsmallmatrix}
1\\
0
\end{bsmallmatrix}\in\SS_{\CN}$ as desired.
\end{proof}

\subsubsection{Construction of the induced extriangulated structure via subfactors}

In \cref{cor:proj-inj_in_H}, we equip the heart $\CH/[\CW]$ with the natural extriangulated structure induced from $\CT$ and denote it by $(\BF,\ft)$\darkblueedit{.}

\normalcolor

\begin{lemma}\label{lem:merged}
The following holds.
\begin{enumerate}
\item For any $U\in\CU$ and $X\in\CT^+$, we have $\BE^L_{\CN}(U,X)=0$.
\item Let $X\xrightarrow{f}Y\xrightarrow{g}Z\overset{h}{\dashrightarrow}$ be any $\fs^L_{\CN}$-triangle with $Z\in\CT^+$. Then, $X\in\CT^+$ if and only if $Y\in\CT^+$.
\end{enumerate}
\end{lemma}
\begin{proof}
{\rm (1)} Let $h\in\BE^L_{\CN}(U,X)$ be any element. Since $U\in\CU\subset\CN$, we have $h\in[\CN[1]]$ by the definition of $\BE^L_{\CN}$. This implies that $h$ factors through some object  in $\CU[1]\ast\CV[1]$. By $\CT(U,\CV[1])=0$, it follows that $h$ factors through some $U'[1]\in\CU[1]$, hence there exist morphisms $h_1\in \CT(U,U'[1]),h_2\in\CT(U'[1],X[1])$ such that $h=h_2\circ h_1$. By $X\in\CT^+=\CW\ast\CV_1^{n}$, there exists a distinguished triangle $W[1]\xrightarrow{w} X[1]\to R[1]\to W[2]$ with $W\in\CW$ and $R\in\CV_1^n$. Since $\CT(U'[1],R[1])=0$, there exists $h_3\in\CT(U'[1],W[1])$ such that $h_2=w\circ h_3$. Then, since $\CT(U,W[1])=0$, it  follows $h=h_2\circ h_1=w\circ h_3\circ h_1=w\circ 0=0$. 

{\rm (2)} By $Z\in\CT^+=\CW\ast\CV_1^{n}$, there exists a distinguished triangle $W\xrightarrow{w}Z\to R[1]\to W[1]$ in $\CT$ with $W\in\CW$  and $R\in\CV_0^{n-1}$. Since $h\circ w\in\BE^L_{\CN}(W,X)$, we have $h\circ w=0$ by {\rm (1)}. Thus $w$ factors through $g$, hence by the octahedron axiom we obtain a commutative diagram 
\[
\begin{tikzpicture}[>=stealth]
\node (1) at (0.98,1.8) {$W$};
\node (2) at (1.19,0) {$Z$};
\node (3) at (1.4,-1.7) {$R[1]$};
\node (4) at (0,0) {$Y$};
\node (5) at (-0.45,-0.9) {$Y'$};
\node (6) at (-2.4,0) {$X$};
\draw[->] (1) -- node[right,font=\scriptsize] {$w$} (2);
\draw[->] (1) -- node[right,font=\scriptsize] {} (4);
\draw[->] (2) -- node[left,font=\scriptsize] {} (3);
\draw[->] (4) -- node[below,font=\scriptsize] {$g$} (2);
\draw[->] (5) -- node[below,font=\scriptsize] {$f'$} (3);
\draw[->] (4) -- node[left,font=\scriptsize] {} (5);
\draw[->] (6) -- node[above,font=\scriptsize] {$f$} (4);
\draw[->] (6) -- node[below,font=\scriptsize] {} (5);
\end{tikzpicture}
\]
in $\CT$ for some $Y'\in \CT$, in which $W\to Y\to Y'\to W[1]$ and $X\to Y'\to R[1]\to X[1]$ are distinguished triangles. By the dual of Corollary~\ref{cor:TUT}, we have $X\in\CT^+$ if and only if $Y'\in\CT^+$ if and only if $Y\in\CT^+$.
\end{proof}

\normalcolor
\begin{proposition}\label{prop:H_is_extension-closed}
We have the following assertions.
\begin{enumerate}[label=\textup{(\arabic*)}]
\item 
$\CT^+$ is extension-closed in $\CT_\CN^L=(\CT,\BE_\CN^L,\fs_\CN^L)$.
\item 
$\CT^-$ is extension-closed in $\CT_\CN^R=(\CT,\BE_\CN^R,\fs_\CN^R)$.
\item 
$\CH$ is extension-closed in $\CT_\CN=(\CT,\BE_\CN,\fs_\CN)$.
\end{enumerate}
In particular, $\CH$ admits an extriangulated structure induced from $\CT_\CN$.
\end{proposition}
\begin{proof}

{\rm (1)} follows from \cref{lem:merged} {\rm (2)}.
\normalcolor
The second item (2) follows dually, and (3) is a combination of (1) and (2).
\end{proof}

\begin{corollary}\label{cor:proj-inj_in_H}
The subcategory $\CW$ consists of projective-injective objects in $(\CH,\BE_\CN|_\CH,\fs_\CN|_\CH)$.
This gives a natural extriangulated structure on the heart $\CH/[\CW]$, which we will denote by $(\BF,\ft)$ in the sequel.
\normalcolor
\end{corollary}
\begin{proof}
The projectivity of $W\in\CW$ immediately follows from \cref{lem:merged} {\rm (1)}. The injectivity of $W$ is shown dually.
\normalcolor
Thanks to \cite[Prop.~3.30]{NP19}, the ideal quotient of $(\CH,\BE_\CN|_\CH,\fs_\CN|_\CH)$ by $\CW$ admits a natural extriangulated structure.
\end{proof}

\color{blue}
The following result, which will be used in the proof of \cref{thm:extri_heart}, also follows from \cref{lem:merged}.
\begin{lemma}\label{lem:sN-surj}
Let $X\xrightarrow{f}Y\xrightarrow{g}Z\xrightarrow{\delta}X[1]$ be a distinguished triangle in $\CT$ with $\delta\in\BE_{\CN}(Z,X)$. Suppose that there exist $U\in\CU$, $u\in\CT(U,X)$ and $\delta'\in\BE_{\CN}(Z,U)$ such that $\delta=u_{\ast}\delta'$ in $\BE_{\CN}(Z,X)$, namely, $\delta=u[1]\circ \delta'$ in $\CT(Z,X[1])$. Then, for any $A\in\CT^+$, the map
$f^{\ast}\colon\BE_{\CN}(Y,A)\to\BE_{\CN}(X,A)$ is surjective.
\end{lemma}
\begin{proof}
Let $A$ be any object in $\CT^+$.
Take an $\fs_{\CN}$-triangle $U\xrightarrow{f'}Y'\xrightarrow{g'}Z\overset{\delta'}{\dashrightarrow}$.
Since $X\xrightarrow{f}Y\xrightarrow{g}Z\overset{\delta}{\dashrightarrow}$ is also an $\fs_{\CN}$-triangle and $\delta=u_{\ast}\delta'$,
there exists $y\in\CT(Y',Y)$ such that
\[
U\xrightarrow{{\scriptsize\begin{bmatrix}u\\ f'\end{bmatrix}}}X\oplus Y'\xrightarrow{[-f\ y]}Y\overset{g^{\ast}\delta'}{\dashrightarrow}
\]
is an $\fs_{\CN}$-triangle, by \cite[Proposition~1.20]{LN19}. Then,
\[
\BE_{\CN}(Y,A)\xrightarrow{{\scriptsize\begin{bmatrix}-f^{\ast}\\ y^{\ast}\end{bmatrix}}}
\BE_{\CN}(X,A)\oplus\BE_{\CN}(Y',A)\to0
\]
becomes exact, since $\BE_{\CN}(U,A)=0$ by \cref{lem:merged} {\rm (1)}.
In particular, the map $f^{\ast}\colon\BE_{\CN}(Y,A)\to\BE_{\CN}(X,A)$ is surjective.
\end{proof}

\color{blue}We summarize properties of $l_X\colon X\to X_+$ in the following proposition. 
\begin{proposition}\label{prop:coreflection_tri1}
Let $X\in\CT$ be any object. The distinguished triangle $U'_X[-1]\xrightarrow{u\circ u'}X\xrightarrow{l_X}X_+\xrightarrow{m_X}U'_X$ in \cref{def:functor_+} satisfies the following properties.
\begin{enumerate}
\item $X\xrightarrow{l_X}X_+\xrightarrow{m_X}U'_X\dashrightarrow$ is an $\fs_\CN$-triangle.
\item $H(l_X)$ is an isomorphism in $\CH/[\CW]$.
\item $(l_X)^{\ast}\colon\BE_{\CN}(X_+,A)\to\BE_{\CN}(X,A)$ is an isomorphism for any $A\in\CT^+$.
\end{enumerate}
Dually, the distinguished triangle $V'_X\to X_-\xrightarrow{r_X}X\xrightarrow{v'\circ v}V'_X[1]$ in \cref{def:functor_-} gives an $\fs_{\CN}$-triangle $V'_X\to X_-\xrightarrow{r_X}X\dashrightarrow$, \darkblueedit{for which} $H(r_X)$ is an isomorphism in $\CH/[\CW]$ and $(r_X)_{\ast}\colon\BE_{\CN}(C,X_-)\to\BE_{\CN}(C,X)$ is an isomorphism for any $C\in\CT^-$.
\end{proposition}
\begin{proof}
{\rm (1)} is by the definition of $(\BE_{\CN},\fs_{\CN})$. 
{\rm (2)} \darkblueedit{is immediate from} \cref{lem:fully_faithful1} and \cref{prop:fully_faithful}.
Let us show {\rm (3)}. For any $A\in\CT^+$,
\[ \BE_{\CN}(U'_X,A)\xrightarrow{(m_X)^{\ast}}\BE_{\CN}(X_+,A)\xrightarrow{(l_X)^{\ast}}\BE_{\CN}(X,A) \]
is exact by {\rm (1)}. Since $\BE_{\CN}(U'_X,A)=0$ by \cref{lem:merged} {\rm (1)}, this shows that $(l_X)^{\ast}\colon\BE_{\CN}(X_+,A)\to\BE_{\CN}(X,A)$ is injective. Its surjectivity follows from \cref{lem:sN-surj}.
\end{proof}
\normalcolor

\subsection{Pretriangulated structure}
\label{subsection:PT_heart}
\subsubsection{Shift functors on the heart}

For each $X\in\CT^-$, we associate $X\langle1\rangle\in\CT^-$ as follows.
\begin{definition}\label{def:functor_<1>}
For each object $X\in\CT^-$, let $0\le j\le n$ be the minimal integer such that $X\in\CU_{-j}^0$, and choose a distinguished triangle
\begin{equation}\label{triangle_<1>}
X\xrightarrow{a_X}W^X\xrightarrow{b_X}X\langle 1\rangle\xrightarrow{c_X}X[1]
\end{equation}
that satisfies $W^X\in\CW$ and 
\[
X\langle 1\rangle\in
\begin{cases}
\, \CU & \text{if}\ j=0\\
\, \CU_{-j+1}^0 & \text{if}\ j>0
\end{cases},
\]
by using $\CU\subset\CU[-1]\ast\CW$ and Lemma~\ref{lem:equation_T-T+}.
Since $\CU\subset\CU_{-j+1}^0\subset\CT^-$, we have $X\langle1\rangle\in\CT^-$.
For any non-negative integer $k$, we denote by $X\langle k\rangle\in\CT^-$ the object obtained by applying $\langle1\rangle$ $k$-times to $X$.
\end{definition}

\begin{proposition}\label{prop:functor_<1>}
The following holds.
\begin{enumerate}
\item For any $X\in\CT^-$, we have $X\langle k\rangle\in\CU_{-n+k}^0$ for all $0\le k\le n$, and  $X\langle k\rangle\in\CU$ for all $k\ge n$.
\item The assignment $X\mapsto X\langle1\rangle$ in Definition~\ref{def:functor_<1>} gives an additive endofunctor 
$\langle1\rangle\colon\CT^-/[\CW]\to \CT^-/[\CW]$. Moreover, this functor is determined uniquely up to a natural isomorphism, regardless of the choice made in Definition~\ref{def:functor_<1>}.
\end{enumerate}

\end{proposition}
\begin{proof}
{\rm (1)} This is immediate from Definition~\ref{def:functor_<1>}.

{\rm (2)} This is shown by the usual argument using Lemma~\ref{lem:for_adjoint_pair_+-} {\rm (1)}.
In fact, for any $x\in\CT^-(X,X')$, we define $\overline{x}\langle1\rangle$ by $\overline{x}\langle1\rangle=\overline{y_x}$, where $y_x\in\CT^-(X\langle1\rangle,X'\langle1\rangle)$ is a morphism that makes
\[
\begin{tikzcd}[row sep=0.8cm, column sep=1.1cm]
X\langle1\rangle \arrow{r}{c_X}\arrow{d}[swap]{y_x}
&X{[}1{]} \arrow[equal]{d}{x[1]}
\\
X'\langle1\rangle \arrow{r}[swap]{c_{X'}}
&X'{[}1{]}
\end{tikzcd}
\]
commutative in $\CT$. By Lemma~\ref{lem:for_adjoint_pair_+-} {\rm (1)}, such $\overline{x}\langle1\rangle$ is unique independently of the choices of representative $x$ of $\overline{x}$ and $y_x$. 
Using the uniqueness, we can check that the correspondence $\langle1\rangle$ indeed preserves composition, identities, and the addition of morphisms.
\normalcolor
\end{proof}

\color{blue}
\begin{remark}
\label{rem:c_is_natural}
As the proof of \cref{prop:functor_<1>} suggests, for any morphism $x\in\CT^-(X,X')$, we have $\ovl{x[1]}\circ \ovl{c_X}=\ovl{c_{X'}}\circ \ovl{x}\langle 1\rangle$ in $\CT/[\CW]$.
\end{remark}
\normalcolor

\begin{definition}
For any non-negative integer $k$, we denote by $\langle k\rangle\colon\CT^-/[\CW]\to \CT^-/[\CW]$ the $k$-times iteration of the endofunctor $\langle1\rangle$.
\end{definition}

We can define a functor $\langle-1\rangle\colon\CT^+/[\CW]\to\CT^+/[\CW]$ in a dual manner, as follows. 
\begin{definition}\label{def:functor_<-1>}
For each object $X\in\CT^+$, let $0\le j\le n$ be the minimal integer such that $X\in\CV_0^j$, and choose a distinguished triangle
\begin{equation}\label{dtriXXWX}
X[-1]\xrightarrow{{}_Xc}X\langle -1\rangle\xrightarrow{{}_Xb}{}^XW\xrightarrow{{}_Xa}X
\end{equation}
that satisfies ${}^XW\in\CW$ and 
\[
X\langle -1\rangle\in
\begin{cases}
\, \CV & \text{if}\ j=0\\
\, \CV_0^{j-1} & \text{if}\ j>0
\end{cases}.
\]
This gives an additive \darkblueedit{endofunctor} $\langle-1\rangle\colon\CT^+/[\CW]\to\CT^+/[\CW]$.
\end{definition}

\normalcolor

Moreover, the following holds.
\begin{lemma}\label{lem:adjoint_+-}
For any pair of objects $X\in\CT^-$ and $Y\in\CT^+$, there is a bijection
\[ \theta_{X,Y}\colon (\CT/[\CW])(X\langle1\rangle,Y)\xrightarrow{\cong}(\CT/[\CW])(X,Y\langle-1\rangle) \]
\normalcolor
which is natural in $X$ and $Y$.
\end{lemma}
\begin{proof}
This is shown by the usual argument using Lemma~\ref{lem:for_adjoint_pair_+-}.
In fact, for any $\overline{f}\in(\CT/[\CW])(X\langle1\rangle,Y)$, we define $\theta_{X,Y}(\overline{f})\in(\CT/[\CW])(X,Y\langle-1\rangle)$ by $\theta_{X,Y}(\overline{f})=\overline{g}$, where $g\in\CT(X,Y\langle-1\rangle)$ is a morphism that makes
\[
\begin{tikzcd}[row sep=0.8cm, column sep=1.2cm]
X\langle1\rangle \arrow{r}{c_X}\arrow{d}[swap]{f}
&X{[}1{]} \arrow[equal]{d}{g[1]}
\\
Y\arrow{r}[swap]{-({}_Yc){[}1{]}}
&Y\langle-1\rangle{[}1{]}
\end{tikzcd}
\]
commutative in $\CT$. Lemma~\ref{lem:for_adjoint_pair_+-} 
\color{blue}
shows that 
\normalcolor 
$\theta_{X,Y}$ is indeed a well-defined bijection. Naturality of $\theta_{X,Y}$ in $X,Y$ can be also checked in a straightforward manner.
\normalcolor
\end{proof}

\begin{definition}\label{def:functor_Sigma}
Define an additive endofunctor $\Sigma\colon\CH/[\CW]\to\CH/[\CW]$ to be the composite of
\[
\CH/[\CW]\overset{\iota}{\hookrightarrow}\CT^-/[\CW]\xrightarrow{\langle 1\rangle}\CT^-/[\CW]\xrightarrow{(\ )_+}\CH/[\CW],
\]
in which $\iota$ is the inclusion. For any non-negative integer $k$, we denote by $\Sigma^k$ the $k$-times iteration of $\Sigma$.
Dually, an additive endofunctor $\Omega\colon\CH/[\CW]\to\CH/[\CW]$ is defined to be the composite of
\[
\CH/[\CW]\hookrightarrow\CT^+/[\CW]\xrightarrow{\langle -1\rangle}\CT^+/[\CW]\xrightarrow{(\ )_-}\CH/[\CW],
\]
and its $k$-times iteration will be denoted by $\Omega^k$.
\end{definition}

\begin{proposition}\label{prop:adjoint_Sigma_Omega}
$\Sigma$ is left adjoint to $\Omega$.
In fact, we have a natural isomorphism
\[
\Theta\colon(\CH/[\CW])(\Sigma(-),-)\overset{\cong}{\Longrightarrow}(\CH/[\CW])(-,\Omega(-))
\]
where $\Theta_{X,Y}$ is defined to be the composite of
\begin{eqnarray*}
(\CH/[\CW])(\Sigma X,Y)&=&(\CH/[\CW])((X\langle1\rangle)_+,Y)\\
&\underset{\cong}{\xrightarrow{-\circ\overline{l_{X\langle1\rangle}}}}&
(\CT^-/[\CW])(X\langle1\rangle,Y)\\
&\underset{\cong}{\xrightarrow{(\theta_{X,Y})}}&(\CT^+/[\CW])(X,Y\langle-1\rangle)\\
&\underset{\cong}{\xrightarrow{(\overline{r_{Y\langle-1\rangle}}\circ-)^{-1}}}&(\CH/[\CW])(X,(Y\langle-1\rangle)_-)\, =\, (\CH/[\CW])(X,\Omega Y)
\end{eqnarray*}
for each $X,Y\in\CH$. In particular, the unit $\eta\colon\id_{\CH/[\CW]}\Rightarrow\Omega\circ \Sigma$ and the counit $\varepsilon\colon\Sigma\circ\Omega\Rightarrow\id_{\CH/[\CW]}$ for the adjunction $\Sigma\dashv\Omega$ are given by
$\eta_X=\Theta_{X,\Sigma X}(\id_{\Sigma X})$ and 
$\varepsilon_Y=\Theta_{\Omega Y,Y}^{-1}(\id_{\Omega Y})$ for any $X,Y\in\CH$, respectively.
\normalcolor
\end{proposition}
\begin{proof}
This follows from Lemma~\ref{lem:adjoint_+-}, Proposition~\ref{prop:functor_+_property} and its dual.
\normalcolor
\end{proof}

\begin{lemma}\label{lem:functor_Sigma}
For any $X\in\CT^-$, the morphism $l_X\in\CT(X,X_+)$ induces an isomorphism
\[
\tau_X=(\overline{l_X}\langle1\rangle)_+\colon X\langle1\rangle_+\xrightarrow{\cong}(X_+\langle1\rangle)_+=\Sigma(X_+)
\]
in $\CH/[\CW]$.
\end{lemma}
\begin{proof}
As in Definition~\ref{def:functor_<1>}, we have a distinguished triangle
\[
X\xrightarrow{a_X}W^X\xrightarrow{b_X}X\langle1\rangle\xrightarrow{c_X}X[1]
\]
with $W^X\in\CW$ and $X\langle1\rangle\in\CT^-$.
As in Definition~\ref{def:functor_+}, we have a distinguished triangle
\[
U'_X[-1]\xrightarrow{u\circ u'}X\xrightarrow{l_X}X_+\to U'_X
\]
with $U'_X\in\CU$.
By the octahedron axiom, we obtain a commutative diagram in $\CT$
\begin{equation}\label{octahedron_UXX}
\begin{tikzcd}[row sep=0.8cm, column sep=0.6cm]
U'_X[-1]\arrow{d}[swap]{u\circ u'}\arrow[equal]{r}{}
&U'_X[-1] \arrow{d}{}
\\
X \arrow{r}{a_X}\arrow{d}[swap]{l_X}
&W^X \arrow{d}{e} \arrow{r}{b_X}
&X\langle1\rangle \arrow{r}{c_X} \arrow[equal]{d}{}
&X[1] \arrow{d}{l_X[1]}
\\
X_+\arrow{r}[swap]{a'}\arrow{d}{}
&U \arrow{d}{} \arrow{r}[swap]{b'}
&X \langle1\rangle \arrow{r}[swap]{c'}
&X_+[1]
\\
U'_X \arrow[equal]{r}
&U'_X
\end{tikzcd}
\end{equation}
for some $U\in\CT$, in which $U'_X[-1]\to W^X\xrightarrow{e} U\to U'_X$ and $X_+\xrightarrow{a'}U\xrightarrow{b'}X\langle1\rangle\xrightarrow{c'}X_+[1]$ are distinguished triangles. Since $\CT(U'_X[-1],W^X)=0$, we have $U\cong W^X\oplus U'_X\in\CU$.
Since $\CU\subset\CU[-1]\ast\CW$, we may take a distinguished triangle $U'[-1]\to U\xrightarrow{f} W'\to U'$ with $U'\in\CU$ and $W'\in\CW$. By the octahedron axiom, we obtain a commutative diagram in $\CT$
\begin{equation}\label{octahedron_UUW}
\begin{tikzcd}[row sep=0.8cm, column sep=0.6cm]
{}
&U'[-1]\arrow{d}{}\arrow[equal]{r}{}
&U'[-1] \arrow{d}{}
\\
X_+ \arrow{r}{a'}\arrow[equal]{d}{}
&U \arrow{d}{f} \arrow{r}{b'}
&X\langle1\rangle \arrow{r}{c'} \arrow{d}{g}
&X_+[1] \arrow[equal]{d}{}
\\
X_+\arrow{r}[swap]{a^{\prime\prime}}
&W' \arrow{d}{} \arrow{r}[swap]{b^{\prime\prime}}
&Y\arrow{r}[swap]{c^{\prime\prime}}\arrow{d}{}
&X_+[1]
\\
{}
&U' \arrow[equal]{r}
&U'
\end{tikzcd}
\end{equation}
in which $U'[-1]\to X\langle1\rangle\xrightarrow{g}Y\to U'$ and $X_+\xrightarrow{a^{\prime\prime}}W'\xrightarrow{b^{\prime\prime}}Y\xrightarrow{c^{\prime\prime}}X_+[1]$ are distinguished triangles.
By $U,U'\in\CU$, we have $g\in\CL$, and hence $H(g)$ is an isomorphism by Proposition~\ref{prop:fully_faithful}. 
Moreover, since $X\langle1\rangle\in\CT^-$, we have $Y\in\CT^-$ by Lemma~\ref{lem:T-toT-} {\rm (1)}. Thus it follows that $\overline{g}_+\colon X\langle1\rangle_+\to Y_+$ is an isomorphism in $\CH/[\CW]$. By $(\ref{octahedron_UXX})$ and $(\ref{octahedron_UUW})$, we have a morphism
\[
\begin{tikzcd}[row sep=0.8cm, column sep=0.6cm]
X \arrow{r}{a_X}\arrow{d}[swap]{l_X}
&W^X \arrow{d}{f\circ e} \arrow{r}{b_X}
&X\langle1\rangle \arrow{r}{c_X} \arrow{d}{g}
&X[1] \arrow{d}{l_X[1]}
\\
X_+ \arrow{r}[swap]{a^{\prime\prime}}
&W'  \arrow{r}[swap]{b^{\prime\prime}}
&Y \arrow{r}[swap]{c^{\prime\prime}}
&X_+[1]
\end{tikzcd}
\]
of distinguished triangles. Since $W'\in\CW$, there exists $y\in\CT(Y,X_+\langle1\rangle)$ such that 
\[
\begin{tikzcd}[row sep=0.8cm, column sep=0.8cm]
X_+ \arrow{r}{a^{\prime\prime}}\arrow[d,equals]
&W' \arrow{d}{} \arrow{r}{b^{\prime\prime}}
&Y \arrow{r}{c^{\prime\prime}} \arrow{d}{y}
&X_+[1] \arrow[d,equals]
\\
X_+ \arrow{r}[swap]{a_{(X_+)}}
&W^{(X_+)}  \arrow{r}[swap]{b_{(X_+)}}
&X_+\langle1\rangle \arrow{r}[swap]{c_{(X_+)}}
&X_+[1]
\end{tikzcd}
\]
becomes a morphism of distinguished triangles.
This gives an isomorphism $\overline{y}\colon Y\xrightarrow{\cong} X_+\langle1\rangle$ in $\CT^-/[\CW]$. Proposition~\ref{prop:functor_<1>} tells us that $\overline{l_X}\langle1\rangle=\overline{y}\circ\overline{g}$ holds in $\CT^-/[\CW]$. Therefore, it follows that $(\overline{l_X}\langle1\rangle)_+=\overline{y}_+\circ\overline{g}_+\colon (X\langle1\rangle)_+\xrightarrow{\cong}(X_+\langle1\rangle)_+$ is an isomorphism in $\CH/[\CW]$. 
\normalcolor
\end{proof}

\begin{proposition}\label{prop:functor_Sigma}
The following holds.
\begin{enumerate}
\item For any non-negative integer $k$, the functor $\Sigma^k$ is isomorphic to the composite of 
\[
\CH/[\CW]\overset{\iota}{\hookrightarrow}\CT^-/[\CW]\xrightarrow{\langle  k\rangle}\CT^-/[\CW]\xrightarrow{(\ )_+}\CH/[\CW].
\]
\item $\Sigma^k=0$ holds for all $k\ge n$.
\end{enumerate}
\end{proposition}
\begin{proof}
{\rm (1)} It suffices to show that
\begin{equation}\label{diagram_to_be_commutative}
\begin{tikzpicture}[>=stealth]
\node (1) at (-2.5,0.8) {$\CT^-/[\CW]$};
\node (2) at (0,0.8) {$\CT^-/[\CW]$};
\node (3) at (-2.5,-0.8) {$\CH/[\CW]$};
\node (4) at (0,-0.8) {$\CT^-/[\CW]$};
\node (5) at (2.5,-0.8) {$\CT^-/[\CW]$};
\node (6) at (2.5,0.8) {$\CH/[\CW]$};
\draw[->] (1) -- node[above,font=\scriptsize] {$\langle1\rangle$} (2);
\draw[->] (2) -- node[above,font=\scriptsize] {$(\ )_+$} (6);
\draw[->] (1) -- node[left,font=\scriptsize] {$(\ )_+$} (3);
\draw[->] (3) -- node[below,font=\scriptsize] {$\iota$} (4);
\draw[->] (4) -- node[below,font=\scriptsize] {$\langle1\rangle$} (5);
\draw[->] (5) -- node[right,font=\scriptsize] {$(\ )_+$} (6);
\end{tikzpicture}
\end{equation}
is commutative up to a natural isomorphism.
Let $X\in\CT^-$ be any object. 
As in Definition~\ref{def:functor_+}, we have a distinguished triangle
\[
U'_X[-1]\xrightarrow{u\circ u'}X\xrightarrow{l_X}X_+\to U'_X
\]
with $U'_X\in\CU$.
We note that $\{\overline{l_X}\}_{X\in\CT/[\CW]}$ forms a natural transformation, which is the unit for the left adjoint functor $(\ )_+$ of the inclusion $\CT^+/[\CW]\hookrightarrow\CT/[\CW]$.
By Lemma~\ref{lem:functor_Sigma}, each
\darkblueedit{$(\overline{l_X}\langle1\rangle)_+$}
is an isomorphism.
Since $\overline{l_X}$ is natural in $X\in\CT$, so is $(\overline{l_X}\langle1\rangle)_+$.  
Thus $(\ref{diagram_to_be_commutative})$ is commutative up to this natural isomorphism.

{\rm (2)} Let $X\in\CH$ be any object, and let $k$ be any integer with $k\ge n$. Then $X\langle k\rangle\in\CU$ by Proposition~\ref{prop:functor_<1>} {\rm (1)}, hence $(X\langle k\rangle)_+\cong 0$ in $\CH/[\CW]$ by Proposition~\ref{prop:functor_+_property} {\rm (2)}. By {\rm (1)}, this means that $\Sigma^kX\cong 0$ holds in $\CH/[\CW]$.
\end{proof}

\subsubsection{Right/left triangles}

We define the class of distinguished right triangles in $\CH/[\CW]$ by the following. 
\begin{definition}\label{def:standard_right_triangle}
Let $f\in\CH(X,Y)$ be any morphism.
Let $X\xrightarrow{a_X}W^X\xrightarrow{b_X}X\langle 1\rangle\xrightarrow{c_X}X[1]$ be the distinguished triangle $(\ref{triangle_<1>})$ in Definition~\ref{def:functor_<1>}.
Complete $\begin{bmatrix}f\\ a_X\end{bmatrix}\in\CT(X,Y\oplus W^X)$ into a distinguished triangle
\[
X\xrightarrow{{\scriptsize\begin{bmatrix}f\\ a_X\end{bmatrix}}}Y\oplus W^X\xrightarrow{{\scriptsize\begin{bmatrix}g_f&w_f\end{bmatrix}}}M_f\xrightarrow{k_f}X[1]
\]
in $\CT$.
Then, by the octahedron axiom, we obtain a commutative diagram in $\CT$
\begin{equation}\label{octahedron_h_f}
\begin{tikzcd}[row sep=0.8cm, column sep=1.2cm]
{}
&Y\arrow{d}[swap]{\begin{bmatrix}\id\\0\end{bmatrix}}\arrow[equal]{r}{}
&Y \arrow{d}{g_f}
\\
X \arrow{r}{\begin{bmatrix}f\\ a_X\end{bmatrix}}\arrow[equal]{d}{}
&Y\oplus W^X \arrow{d}{[0\ \id]} \arrow{r}{[g_f\ w_f]}
&M_f \arrow{r}{k_f} \arrow{d}{h_f}
&X[1] \arrow[equal]{d}{}
\\
X\arrow{r}[swap]{a_X}
&W^X \arrow{d}{0} \arrow{r}[swap]{b_X}
&X\langle1\rangle\arrow{r}[swap]{c_X}\arrow{d}{s_f}
&X[1]\arrow{d}{\begin{bmatrix}f\\ a_X\end{bmatrix}{[}1{]}}
\\
{}
&Y[1] \arrow[equal]{r}
&Y[1] \arrow{r}[swap]{-\begin{bmatrix}\id\\0\end{bmatrix}{[}1{]}}
&(Y\oplus W^X)[1]
\end{tikzcd}
\end{equation}
in which $Y\xrightarrow{g_f}M_f\xrightarrow{h_f}X\langle1\rangle\xrightarrow{s_f}Y[1]$ is a distinguished triangle in $\CT$.
By Lemma~\ref{lem:T-_*_<1>} {\rm (1)}, we have $M_f\in\CT^-$.
We define \darkblueedit{the} \emph{standard right triangle associated to $f$} to be the sequence
\begin{equation}\label{standard_right_triangle}
X\xrightarrow{\overline{f}}Y\xrightarrow{\overline{l_{M_f}}\circ \overline{g_f}}(M_f)_+\xrightarrow{\overline{(h_f)}_+}\Sigma X
\end{equation}
in $\CH/[\CW]$, obtained in this way.
A \emph{distinguished right triangle} is a right triangle %
in $\CH/[\CW]$ isomorphic to the standard right triangle associated to some morphism $f\in\CH(X,Y)$. We denote by $\vartriangleright$ the class of distinguished right triangles. 
\end{definition}

\begin{remark}\label{rem:standard_h_f}
In fact, it is not necessary to choose $h_f$ via the octahedron axiom:
any morphism $h_f$ for which the diagram
\[
\begin{tikzcd}[row sep=0.8cm, column sep=1.2cm]
X \arrow{r}{\begin{bmatrix}f\\ a_X\end{bmatrix}}\arrow[d, equal]
&Y\oplus W^X \arrow{d}{{[}0\ \id{]}} \arrow{r}{{[}g_f\ w_f{]}}
&M_f\arrow{r}{k_f} \arrow{d}{h_f}
&X{[}1{]} \arrow[d, equal]
\\
X\arrow{r}[swap]{a_X}
&W^X\arrow{r}[swap]{b_X}
&X \langle1\rangle \arrow{r}[swap]{c_X}
&X{[}1{]}
\end{tikzcd}\]
is a morphism of distinguished triangles in $\CT$,
already gives the standard right triangle~(\ref{standard_right_triangle}).
Indeed, if any other $h'\in\CT(M_f,X\langle1\rangle)$ satisfies $c_X\circ h'=k_f$, then $h_f-h'$ should factor through $b_X$, which means that $\overline{h_f}=\overline{h'}$ holds in 
\color{blue}
$\CT^-/[\CW]$.
\normalcolor
\end{remark}

\begin{lemma}\label{lem:H(f)_monom}
Let $X,Y\in\CT$ be any pair of objects, and let $f\in\CT(X,Y)$ be any morphism. Suppose that there exists a distinguished triangle in $\CT$
\[
V\xrightarrow{e}X\xrightarrow{f}Y\to V[1]
\]
with $V\in\CV$. Then $H(f)$ is a monomorphism in $\CH/[\CW]$.
\end{lemma}
\begin{proof}
\textbf{Step 1: Reduction to the case $Y\in\CT^-$.} By the definition of $Y_-$, there exists a distinguished triangle
\[
V'_Y\to Y_-\xrightarrow{r_Y}Y\xrightarrow{v^{\prime\prime}}V'_Y[1]
\]
in $\CT$ with $V'_Y\in\CV$ and $v^{\prime\prime}\in[\CV]$. Then, by the octahedron axiom, we obtain a commutative diagram in $\CT$
\[
\begin{tikzpicture}[>=stealth]
\node (1) at (-2.4,0) {$V$};
\node (2) at (-0.42,0.88) {$X'$};
\node (3) at (1.4,1.7) {$Y_-$};
\node (4) at (0,0) {$X$};
\node (5) at (1.2,0) {$Y$};
\node (6) at (0.98,-1.8) {$V'_Y[1]$};
\draw[->] (1) -- node[above, font=\scriptsize] {} (2);
\draw[->] (1) -- node[below,font=\scriptsize] {$e$} (4);
\draw[->] (2) -- node[above,font=\scriptsize] {$f'$} (3);
\draw[->] (2) -- node[right,font=\scriptsize] {$x$} (4);
\draw[->] (3) -- node[right,font=\scriptsize] {$r_Y$} (5);
\draw[->] (4) -- node[below,font=\scriptsize] {$f$} (5);
\draw[->] (4) -- node[left,font=\scriptsize] {$v^{\prime\prime}\circ f$} (6);
\draw[->] (5) -- node[right, font=\scriptsize] {$v^{\prime\prime}$} (6);
\end{tikzpicture}
\]
in which $V\to X'\xrightarrow{f'}Y_-\to V[1]$ and $V'_Y\to X'\xrightarrow{x}X\xrightarrow{v^{\prime\prime}\circ f}V'_Y[1]$ are distinguished triangles. By Lemma~\ref{lem:fully_faithful1} {\rm (2)}, we have $r_Y\in\CR$.
Similarly, since $v^{\prime\prime}\circ f\in[\CV]$, we have $x\in\CR$.
Thus, by Proposition~\ref{prop:fully_faithful}, it follows that $H(r_Y)$ and $H(x)$ are isomorphisms in $\CH/[\CW]$.
By the commutativity of 
\[
\begin{tikzpicture}[>=stealth]
\node (1) at (-1.2,0.7) {$H(X')$};
\node (2) at (1.2,0.7) {$H(Y_-)$};
\node (3) at (-1.2,-0.7) {$H(X)$};
\node (4) at (1.2,-0.7) {$H(Y)$};
\draw[->] (1) -- node[above,font=\scriptsize] {$H(f')$} (2);
\draw[->] (1) -- node[left,font=\scriptsize] {$H(x)$}
node[right,font=\scriptsize] {$\cong$}(3);
\draw[->] (2) -- node[left,font=\scriptsize] {$\cong$}
node[right,font=\scriptsize] {$H(r_Y)$}(4);
\draw[->] (3) -- node[below,font=\scriptsize] {$H(f)$} (4);
\end{tikzpicture}
\]
it suffices to show that $H(f')$ is a monomorphism. Replacing $Y$ with $Y_-$ if necessary, we may assume $Y\in\CT^-$ from the beginning.

\textbf{Step 2: Reduction to the case $X\in\CT^+$ and $Y\in\CH$.}
We assume $Y\in\CT^-$. By the definition of $Y_+$, there exist morphisms $u\in\CT(U_Y,Y)$, $u'\in\CT(U'_Y[-1],U_Y)$ with $U_Y,U'_Y\in\CU$ and a distinguished triangle
\[
U'_Y[-1]\xrightarrow{u\circ u'}Y\xrightarrow{l_Y}Y_+\to U'_Y
\]
in $\CT$.
Since $\CT(U_Y,V[1])=0$, there exists $g\in\CT(U_Y,X)$ such that $u=f\circ g$. By the octahedron axiom, we obtain a commutative diagram in $\CT$
\[
\begin{tikzpicture}[>=stealth]
\node (1) at (-0.98,1.8) {$U'_Y[-1]$};
\node (2) at (-1.19,0) {$X$};
\node (3) at (-1.4,-1.7) {$X'$};
\node (4) at (0,0) {$Y$};
\node (5) at (0.45,-0.9) {$Y_+$};
\node (6) at (2.4,0) {$V[1]$};
\draw[->] (1) -- node[left,font=\scriptsize] {$g\circ u'$} (2);
\draw[->] (1) -- node[right,font=\scriptsize] {$u\circ u'$} (4);
\draw[->] (2) -- node[left,font=\scriptsize] {$x$} (3);
\draw[->] (2) -- node[below,font=\scriptsize] {$f$} (4);
\draw[->] (3) -- node[below,font=\scriptsize] {$f'$} (5);
\draw[->] (4) -- node[left,font=\scriptsize] {$l_Y$} (5);
\draw[->] (4) -- node[above,font=\scriptsize] {} (6);
\draw[->] (5) -- node[below,font=\scriptsize] {} (6);
\end{tikzpicture}
\]
for some $X'\in\CT$, in which 
$U'_Y[-1]\xrightarrow{g\circ u'}X\xrightarrow{x} X'\to U'_Y$
and
$V\to X'\xrightarrow{f'}Y_+\to V[1]$
are distinguished triangles. In particular, we have $X'\in\CV\ast\CT^+=\CT^+$.
By Lemma~\ref{lem:fully_faithful1} {\rm (1)}, we have $l_Y\in\CL$. Similarly, since $g\circ u'\in[\CU]$, we have $x\in\CL$.
Thus, by Proposition~\ref{prop:fully_faithful}, $H(l_Y)$ and $H(x)$ are isomorphisms in $\CH/[\CW]$.
By the commutativity of 
\[
\begin{tikzpicture}[>=stealth]
\node (1) at (-1.2,0.7) {$H(X)$};
\node (2) at (1.2,0.7) {$H(Y)$};
\node (3) at (-1.2,-0.7) {$H(X')$};
\node (4) at (1.2,-0.7) {$H(Y_+)$};
\draw[->] (1) -- node[above,font=\scriptsize] {$H(f)$} (2);
\draw[->] (1) -- node[left,font=\scriptsize] {$H(x)$}
node[right,font=\scriptsize] {$\cong$}(3);
\draw[->] (2) -- node[left,font=\scriptsize] {$\cong$}
node[right,font=\scriptsize] {$H(l_Y)$}(4);
\draw[->] (3) -- node[below,font=\scriptsize] {$H(f')$} (4);
\end{tikzpicture}
\]
it suffices to show that $H(f')$ is a monomorphism. Replacing $X$,$Y$ with $X'$,$Y_+$ if necessary, we may assume $X\in\CT^+$ and $Y\in\CH$ from the beginning.

\textbf{Step 3: Reduction to the case $X,Y\in\CH$.}
We assume $X\in\CT^+$ and $Y\in\CH$. By the definition of $X_-$, there exists a distinguished triangle
\[
V'_X\to X_-\xrightarrow{r_X}X\xrightarrow{v^{\prime\prime}}V'_X[1]
\]
in $\CT$ with $V'_X\in\CV$ and $v^{\prime\prime}\in[\CV]$. By the octahedron axiom, we obtain a commutative diagram in $\CT$
\[
\begin{tikzpicture}[>=stealth]
\node (1) at (-0.98,1.8) {$X_-$};
\node (2) at (-1.19,0) {$X$};
\node (3) at (-1.4,-1.7) {$V'_X[1]$};
\node (4) at (0,0) {$Y$};
\node (5) at (0.45,-0.9) {$Z[1]$};
\node (6) at (2.4,0) {$V[1]$};
\draw[->] (1) -- node[left,font=\scriptsize] {$r_X$} (2);
\draw[->] (1) -- node[right,font=\scriptsize] {$f\circ r_X$} (4);
\draw[->] (2) -- node[left,font=\scriptsize] {$v^{\prime\prime}$} (3);
\draw[->] (2) -- node[below,font=\scriptsize] {$f$} (4);
\draw[->] (3) -- (5);
\draw[->] (4) -- node[left,font=\scriptsize] {} (5);
\draw[->] (4) -- node[above,font=\scriptsize] {} (6);
\draw[->] (5) -- node[below,font=\scriptsize] {} (6);
\end{tikzpicture}
\]
for some $Z\in\CT$, in which 
$V\to V'_X[1]\to Z[1]\to V[1]$
and $Z\to X_-\xrightarrow{f\circ r_X} Y\to Z[1]$
are distinguished triangles. In particular we have $Z\in\CV$. Since $H(r_X)$ is an isomorphism in $\CH/[\CW]$, by a similar argument as before, \darkblueedit{we may replace} $X$ with $X_-$ and assume that $X\in\CH$.

\textbf{Step 4: Proof in the case $X,Y\in\CH$.}
We assume $X,Y\in\CH$. It suffices to show that $\overline{f}$ is a monomorphism in $\CH/[\CW]$.
Suppose that $A\in\CH$ and $a\in\CH(A,X)$ satisfies $\overline{f}\circ\overline{a}=0$ in $\CH/[\CW]$. It means that there are $W\in\CW$, $w_1\in\CT(A,W)$ and $w_2\in\CT(W,Y)$ such that $f\circ a=w_2\circ w_1$ in $\CT$. By $\CT(W,V[1])=0$, there exists $w_3\in\CT(W,X)$ such that $w_2=f\circ w_3$. Since 
\[
f\circ(a-w_3\circ w_1)=f\circ a-w_2\circ w_1=0,
\]
there exists $a'\in\CT(A,V)$ such that $a-w_3\circ w_1=e\circ a'$. By $A\in\CH$ and $V\in\CV$, we see that $a'\in[\CW]$, and hence $a=e\circ a'+w_3\circ w_1\in[\CW]$. This shows that $\overline{f}$ is a monomorphism.
\end{proof}

\begin{lemma}\label{lem:for_right_triangle_replace}
Let $X\xrightarrow{f}Y\xrightarrow{g}Z\xrightarrow{h}X[1]$ be any distinguished triangle in $\CT$ with $X\in\CT^-$.
Let
$X\xrightarrow{a_X}W^X\xrightarrow{b_X}X\langle 1\rangle\xrightarrow{c_X}X[1]$
be the distinguished triangle as in Definition~\ref{def:functor_<1>}, and assume that there exists \darkblueedit{a morphism} $a\in\CT(Y,W^X)$ such that $a\circ f=a_X$ in $\CT$.
Then, the following holds.
\begin{enumerate}
\item There exists $v\in\CT(Z,X\langle1\rangle)$ such that $c_X\circ v=h$ in $\CT$. 
In particular we have $H(c_X)\circ H(v)=H(h)$ in $\CH/[\CW]$. 
\item Morphism $\overline{v}_+\in(\CH/[\CW])(Z_+,X\langle1\rangle_+)$ is uniquely determined by the equality $H(c_X)\circ H(v)=H(h)$.
\end{enumerate}
\end{lemma}
\begin{proof}
{\rm (1)} This immediately follows from \darkblueedit{{\rm (TR3)}} applied to the commutative diagram
\[
\begin{tikzcd}[row sep=0.8cm, column sep=0.8cm]
X \arrow{r}{f}\arrow[d, equal]
&Y\arrow{d}{a} \arrow{r}{g}
&Z\arrow{r}{h} %
&X{[}1{]} \arrow[d, equal]
\\
X\arrow{r}[swap]{a_X}
&W^X\arrow{r}[swap]{b_X}
&X \langle1\rangle \arrow{r}[swap]{c_X}
&X{[}1{]}
\end{tikzcd}
\]
in $\CT$.

{\rm (2)} Note that we have a commutative diagram
\[
\begin{tikzpicture}[>=stealth]
\node (1) at (-2.2,0.6) {$H(Z)$};
\node (2) at (0,0.6) {$H(X\langle1\rangle)$};
\node (3) at (-2.2,-0.6) {$Z_+$};
\node (4) at (0,-0.6) {$X\langle1\rangle_+$};
\node (5) at (2.5,0.6) {$H(X[1])$};
\draw[->] (1) -- node[below,font=\scriptsize] {$H(v)$} (2);
\draw[->,bend left=20] (1) to node[above,font=\scriptsize] {$H(h)$} (5);
\draw[->] (1) -- node[left,font=\scriptsize] {$\cong$} (3);
\draw[->] (2) -- node[right,font=\scriptsize] {$\cong$} (4);
\draw[->] (3) -- node[below,font=\scriptsize] {$\overline{v}_+$} (4);
\draw[->] (2) -- node[below,font=\scriptsize] {$H(c_X)$} (5);
\end{tikzpicture}
\]
in $\CH/[\CW]$. Since $H(c_X)$ is monomorphic by Lemma~\ref{lem:H(f)_monom},
\darkblueedit{the morphism} $\overline{v}_+$
is uniquely determined by the equality $H(c_X)\circ H(v)=H(h)$.
\end{proof}

\begin{lemma}\label{lem:right_triangle_replace}
Let $X\xrightarrow{f}Y\xrightarrow{g}Z\xrightarrow{h}X[1]$ be any distinguished triangle in $\CT$ with $X,Y,Z\in\CT^-$.
Let
$X\xrightarrow{a_X}W^X\xrightarrow{b_X}X\langle 1\rangle\xrightarrow{c_X}X[1]$
be the distinguished triangle as in Definition~\ref{def:functor_<1>}, and assume that there exists \darkblueedit{a morphism} $a\in\CT(Y,W^X)$ such that $a\circ f=a_X$ in $\CT$.
Let $\overline{v}_+\in(\CH/[\CW])(Z_+,X\langle1\rangle_+)$ be the unique morphism satisfying $H(c_X)\circ H(v)=H(h)$ obtained in Lemma~\ref{lem:for_right_triangle_replace}.
Then, 
\[ X_+\xrightarrow{\overline{f}_+}Y_+\xrightarrow{\overline{g}_+}Z_+\xrightarrow{\tau_X\circ \overline{v}_+}\Sigma (X_+)
\]
belongs to $\vartriangleright$, where $\tau_X=(\overline{l_X}\langle1\rangle)_+\colon X\langle1\rangle_+\xrightarrow{\cong}\Sigma(X_+)$ is the isomorphism obtained in Lemma~\ref{lem:functor_Sigma}.
\end{lemma}
\begin{proof}
First we show the following claim, which corresponds to the special case of $X,Y\in\CH$.
\begin{claim}\label{claim:if_XYinH}
If moreover $X,Y\in\CH$, then 
$X\xrightarrow{\overline{f}}Y\xrightarrow{\overline{l_Z\circ g}}Z_+\xrightarrow{\overline{v}_+}\Sigma X$
belongs to $\vartriangleright$.
\end{claim}
\begin{proof}[Proof of Claim~\ref{claim:if_XYinH}]
Take the standard right triangle associated to $f$. We use the notation in Definition~\ref{def:standard_right_triangle}.
Since $X\xrightarrow{f}Y\xrightarrow{g}Z\xrightarrow{h}X[1]$ is a distinguished triangle in $\CT$, so is
\[
X\xrightarrow{{\scriptsize\begin{bmatrix}f\\ 0\end{bmatrix}}}Y\oplus W^X\xrightarrow{g\oplus \id}Z\oplus W^X\xrightarrow{{\scriptsize\begin{bmatrix}h&0\end{bmatrix}}}X[1].
\]
Put $y=\begin{bmatrix}\id&0\\ a&\id\end{bmatrix}\colon Y\oplus W^X\xrightarrow{\cong}Y\oplus W^X$. It is an isomorphism in $\CT$, and hence there is an isomorphism $\begin{bmatrix}z&m\end{bmatrix}\in\CT(Z\oplus W^X,M_f)$ that makes
\[
\begin{tikzcd}[row sep=0.8cm, column sep=1.0cm]
X \arrow{r}{\begin{bmatrix}f\\ 0\end{bmatrix}}\arrow[d, equal]
&Y\oplus W^X \arrow{d}{\cong}[swap]{y} \arrow{r}{g\oplus\id}
&Z\oplus W^X\arrow{r}{[h\ 0]} \arrow{d}{\cong}[swap]{{[}z\ m{]}}
&X{[}1{]} \arrow[d, equal]
\\
X\arrow{r}[swap]{\begin{bmatrix}f\\ a_X\end{bmatrix}}
&Y\oplus W^X\arrow{r}[swap]{[g_f\ w_f]}
&M_f \arrow{r}[swap]{k_f}
&X{[}1{]}
\end{tikzcd}
\]
an isomorphism of distinguished triangles in $\CT$.
In particular, $\overline{z}\colon Z\to M_f$ is an isomorphism in $\CT/[\CW]$. If we put $v=h_f\circ z$, then it satisfies $c_X\circ v=k_f\circ z=h$. Thus we may use this $v$ to give the morphism $\overline{v}_+$ in the statement, by Lemma~\ref{lem:for_right_triangle_replace} {\rm (2)}. Since
\[
\begin{tikzpicture}[>=stealth]
\node (1) at (-3.2,0.8) {$X$};
\node (2) at (-1.2,0.8) {$Y$};
\node (3) at (1.2,0.8) {$(M_f)_+$};
\node (4) at (3.6,0.8) {$X\langle1\rangle_+$};
\node (5) at (1.2,-0.8) {$Z_+$};
\node (6) at (4.64,0.82) {$=\Sigma X$};
\draw[->] (1) -- node[above,font=\scriptsize] {$\overline{f}$} (2);
\draw[->] (2) -- node[above,font=\scriptsize] {$\overline{l_{M_f}\circ g_f}$} (3);
\draw[->] (3) -- node[above,font=\scriptsize] {$\overline{(h_f)}_+$} (4);
\draw[->] (2) -- node[left,font=\scriptsize] {$\overline{l_Z\circ g}\ \ $} (5);
\draw[->] (5) -- node[left,font=\scriptsize] {$\cong$} node[right,font=\scriptsize] {$\overline{z}_+$} (3);
\draw[->] (5) -- node[right,font=\scriptsize] {$\ \ \overline{v}_+$} (4);
\end{tikzpicture}
\]
is commutative in $\CH/[\CW]$, this shows that $X\xrightarrow{\overline{f}}Y\xrightarrow{\overline{l_Z\circ g}}Z_+\xrightarrow{\overline{v}_+}\Sigma X$ is isomorphic to the standard right triangle associated to $f$, hence belongs to $\vartriangleright$. Thus Claim~\ref{claim:if_XYinH} is shown.
\end{proof}
We proceed to prove Lemma~\ref{lem:right_triangle_replace}. Take $X_+$. By its definition, there exists a distinguished triangle
\[ U'_X[-1]\xrightarrow{u_X^{\prime\prime}}X\xrightarrow{l_X}X_+\xrightarrow{m_X}U'_X \]
 with $U'_X\in\CU$ and $u_X^{\prime\prime}\in[\CU]$.
By the octahedron axiom, we obtain a commutative diagram in $\CT$
\begin{equation}\label{octa:X_+}
\begin{tikzcd}[row sep=0.8cm, column sep=0.6cm]
U'_X[-1]\arrow{d}[swap]{u_X^{\prime\prime}}\arrow[equal]{r}{}
&U'_X[-1] \arrow{d}{f\circ u_X^{\prime\prime}}
\\
X \arrow{r}{f}\arrow{d}[swap]{l_X}
&Y \arrow{d}{y} \arrow{r}{g}
&Z \arrow{r}{h} \arrow[equal]{d}{}
&X[1] \arrow{d}{l_X[1]}
\\
X_+\arrow{r}[swap]{f'}\arrow{d}[swap]{m_X}
&Y' \arrow{d}{m'} \arrow{r}[swap]{g'}
&Z \arrow{r}[swap]{h'}
&X_+[1]
\\
U'_X \arrow[equal]{r}
&U'_X
\end{tikzcd}
\end{equation}
for some $Y'$, in which $U'_X[-1]\xrightarrow{f\circ u_X^{\prime\prime}}Y\xrightarrow{y}Y'\xrightarrow{m'} U'_X$ and $X_+\xrightarrow{f'}Y'\xrightarrow{g'}Z\xrightarrow{h'} X_+[1]$ are distinguished triangles in $\CT$. Since $Y\in\CT^-$, we have $Y'\in \CT^-\ast\CU=\CT^-$. Moreover, since $f\circ u_X^{\prime\prime}\in[\CU]$ and $U'_X\in\CU$, we have $y\in\CL$. Proposition~\ref{prop:fully_faithful} shows that \color{blue}
$\overline{l_X}\colon X\to X_+$ 
\normalcolor
and $\overline{y}_+\colon Y_+\to Y'_+$ are isomorphisms in $\CH/[\CW]$. 

By the definition of $Y'_+$, there exists a distinguished triangle
\[ U'_{Y'}[-1]\xrightarrow{u_{Y'}^{\prime\prime}}Y'\xrightarrow{l_{Y'}}Y'_+\xrightarrow{m_{Y'}}U'_{Y'} \]
in $\CT$ with  $U'_{Y'}\in\CU$ and $u_{Y'}^{\prime\prime}\in [\CU]$.
By the octahedron axiom, we obtain
\begin{equation}\label{octa:Y'+}
\begin{tikzcd}[row sep=0.8cm, column sep=0.6cm]
{}
&U'_{Y'}[-1]\arrow{d}[swap]{u_{Y'}^{\prime\prime}}\arrow[equal]{r}{}
&U'_{Y'}[-1] \arrow{d}{g'\circ u_{Y'}^{\prime\prime}}
\\
X_+ \arrow{r}{f'}\arrow[equal]{d}{}
&Y' \arrow{d}{l_{Y'}} \arrow{r}{g'}
&Z \arrow{r}{h'} \arrow{d}{z}
&X_+[1] \arrow[equal]{d}{}
\\
X_+\arrow{r}[swap]{f^{\prime\prime}}
&Y'_+ \arrow{d}{m_{Y'}} \arrow{r}[swap]{g^{\prime\prime}}
&Z'\arrow{r}[swap]{h^{\prime\prime}}\arrow{d}{}
&X_+[1]
\\
{}
&U'_{Y'} \arrow[equal]{r}
&U'_{Y'}
\end{tikzcd}
\end{equation}
for some $Z'\in\CT$, in which $U'_{Y'}[-1]\xrightarrow{g'\circ u_{Y'}^{\prime\prime}}Z\xrightarrow{z}Z'\to U'_{Y'}$ and
$X_+\xrightarrow{f^{\prime\prime}}Y'_+\xrightarrow{g^{\prime\prime}}Z'\xrightarrow{h^{\prime\prime}}X_+[1]$ are distinguished triangles in $\CT$.
Since $Z\in\CT^-$, we have $Z'\in\CT^-$. Since $U'_{Y'}\in\CU$ and $g'\circ u^{\prime\prime}_{Y'}\in[\CU]$, we have $z\in \CL$. By Proposition~\ref{prop:fully_faithful}, 
\color{blue}
$\overline{l_{Y'}}\colon Y'\to Y'_+$ 
\normalcolor 
and $\overline{z}_+\colon Z_+\to Z'_+$ are isomorphisms in $\CH/[\CW]$. 

To apply Claim~\ref{claim:if_XYinH} to $X_+\xrightarrow{f^{\prime\prime}}Y'_+\xrightarrow{g^{\prime\prime}}Z'\xrightarrow{h^{\prime\prime}}X_+[1]$, we show that there exists a morphism $Y'_+\to W^{(X_+)}$ that makes the diagram
\begin{equation}\label{diagram_to_commute_XYW}
\begin{tikzpicture}[>=stealth]
\node (1) at (-0.8,0.5) {$X_+$};
\node (2) at (0.8,0.5) {$Y'_+$};
\node (3) at (0,-0.7) {$W^{(X_+)}$};
\draw[->] (1) -- node[above,font=\scriptsize] {$f^{\prime\prime}$} (2);
\draw[->] (1) -- node[left,font=\scriptsize] {$a_{(X_+)}$} (3);
\draw[->] (2) -- node[above,font=\scriptsize] {} (3);
\end{tikzpicture}
\end{equation}
commutative in $\CT$.
By assumption, there exists $a\in\CT(Y,W^X)$ such that $a\circ f=a_X$ in $\CT$. 
\color{blue}
Since $a\circ(f\circ u^{\prime\prime}_X)=0$ by $\CT(U'_X[-1],W^X)=0$, 
\normalcolor
there exists $a'\in\CT(Y',W^X)$ such that $a'\circ y=a$ in $\CT$. 
Note that there is $w\in\CT(W^X,W^{(X_+)})$ that makes
\[
\begin{tikzpicture}[>=stealth]
\node (1) at (-0.9,0.7) {$X$};
\node (2) at (0.9,0.7) {$X_+$};
\node (3) at (-0.9,-0.7) {$W^X$};
\node (4) at (0.9,-0.7) {$W^{(X_+)}$};
\draw[->] (1) -- node[above,font=\scriptsize] {$l_X$} (2);
\draw[->] (1) -- node[left,font=\scriptsize] {$a_X$} (3);
\draw[->] (2) -- node[right,font=\scriptsize] {$a_{(X_+)}$} (4);
\draw[->] (3) -- node[below,font=\scriptsize] {$w$} (4);
\end{tikzpicture}
\]
commutative in $\CT$.
Then, since 
\[
(a_{(X_+)}-w\circ a'\circ f')\circ l_X=w\circ a_X-w\circ a'\circ y\circ f=w\circ(a_X-a\circ f)=0,
\]
\darkblueedit{there exists} %
$q\in\CT(U'_X,W^{(X_+)})$ such that $a_{(X_+)}-w\circ a'\circ f'=q\circ m_X$.
If we put $a^{\prime\prime}=w\circ a'+q\circ m'\in\CT(Y',W^{(X_+)})$, then it satisfies $a_{(X_+)}=a^{\prime\prime}\circ f'$ in $\CT$.
By Lemma~\ref{lem:T-toT-} {\rm (3)}, we obtain $a^{\prime\prime\prime}\in\CT(Y'_+,W^{(X_+)})$ that satisfies $a^{\prime\prime\prime}\circ l_{Y'}=a^{\prime\prime}$, hence makes $(\ref{diagram_to_commute_XYW})$ commutative.

Thus, we are able to apply Claim~\ref{claim:if_XYinH} to $X_+\xrightarrow{f^{\prime\prime}}Y'_+\xrightarrow{g^{\prime\prime}}Z'\xrightarrow{h^{\prime\prime}}X_+[1]$.
It shows that 
\[
X_+\xrightarrow{\overline{f^{\prime\prime}}}Y'_+\xrightarrow{\overline{l_{Z'}\circ g^{\prime\prime}}}Z'_+\xrightarrow{\overline{v'}_+}\Sigma (X_+)
\]
belongs to $\vartriangleright$, where $\overline{v'}_+$ is the morphism uniquely determined by the equality $H(c_{(X_+)})\circ H(v')=H(h^{\prime\prime})$ in $\CH/[\CW]$.
Then, in the diagram below
\begin{equation}\label{to_be_a_morph_of_right_tri}
\begin{tikzcd}[row sep=0.8cm, column sep=1.2cm]
X_+ \arrow{r}{\overline{f}_+}\arrow[d, equal]
&Y_+ \arrow{d}{\overline{y}_+}[swap]{\cong} \arrow{r}{\overline{g}_+}
&Z_+\arrow{r}{\tau_X\circ\overline{v}_+} \arrow{d}[swap]{\cong}{\overline{z}_+}
&\Sigma(X_+) \arrow[d, equal]
\\
X_+\arrow{r}[swap]{\overline{f^{\prime\prime}}}
&Y'_+\arrow{r}[swap]{\overline{l_{Z'}\circ g^{\prime\prime}}}
&Z'_+ \arrow{r}[swap]{\overline{v'}_+}
&\Sigma(X_+)
\end{tikzcd}
\end{equation}
in $\CH/[\CW]$, the left and middle squares commute by the commutativity of $(\ref{octa:X_+})$ and $(\ref{octa:Y'+})$.
We claim that the right square is also commutative.
By the definition of $\tau_X$, if we take a morphism $l'\in\CT(X\langle1\rangle,X_+\langle1\rangle)$ such that $c_{(X_+)}\circ l'=l_X[1]\circ c_X$, then it satisfies 
\color{blue}
$\tau_X=\overline{l'}_+$. 
\normalcolor
By the commutativity of
\[
\begin{tikzpicture}[>=stealth]
\node (1) at (-4.8,1.2) {$Z_+$};
\node (2) at (-2.4,1.2) {$X\langle1\rangle_+$};
\node (3) at (0,1.2) {$\Sigma(X_+)$};
\node (4) at (2.4,1.2) {$Z'_+$};
\node (5) at (4.8,1.2) {$Z_+$};
\node (11) at (-4.8,0) {$H(Z)$};
\node (12) at (-2.4,0) {$H(X\langle1\rangle)$};
\node (13) at (0,0) {$H(X_+\langle1\rangle)$};
\node (14) at (2.4,0) {$H(Z')$};
\node (15) at (4.8,0) {$H(Z)$};
\node (22) at (-2.4,-1.4) {$H(X{[}1{]})$};
\node (23) at (0,-1.4) {$H(X_+{[}1{]})$};
\draw[->] (1) -- node[above,font=\scriptsize] {$\overline{v}_+$} (2);
\draw[->] (2) -- node[above,font=\scriptsize] {$\tau_X$} (3);
\draw[->] (4) -- node[above,font=\scriptsize] {$\overline{v'}_+$} (3);
\draw[->] (5) -- node[above,font=\scriptsize] {$\overline{z}_+$} (4);
\draw[->] (1) -- node[left,font=\scriptsize] {$\cong$} (11);
\draw[->] (2) -- node[left,font=\scriptsize] {$\cong$} (12);
\draw[->] (3) -- node[left,font=\scriptsize] {$\cong$} (13);
\draw[->] (4) -- node[left,font=\scriptsize] {$\cong$} (14);
\draw[->] (5) -- node[left,font=\scriptsize] {$\cong$} (15);
\draw[->] (11) -- node[above,font=\scriptsize] {$H(v)$} (12);
\draw[->] (12) -- node[above,font=\scriptsize] {$H(l')$} (13);
\draw[->] (14) -- node[above,font=\scriptsize] {$H(v')$} (13);
\draw[->] (15) -- node[above,font=\scriptsize] {$H(z)$} (14);
\draw[->,bend right=5] (11) to node[left,font=\scriptsize] {$H(h)\ $} (22);
\draw[->] (12) -- node[left,font=\scriptsize] {$H(c_X)$} (22);
\draw[->] (13) -- node[left,font=\scriptsize] {$H(c_{(X_+)})$} (23);
\draw[->,bend left=5] (14) to node[right,font=\scriptsize] {$\ H(h^{\prime\prime})$} (23);
\draw[->,bend left=8] (15) to node[right,font=\scriptsize] {$\ H(h)$} (23);
\draw[->] (22) -- node[below,font=\scriptsize] {$H(l_X{[}1{]})$} (23);
\end{tikzpicture}
\]
in $\CH/[\CW]$, it follows $\overline{v'}_+\circ \overline{z}_+=\tau_X\circ\overline{v}_+$ in $\CH/[\CW]$ since $H(c_{(X_+)})$ is monomorphic.
This shows that the right square of 
$(\ref{to_be_a_morph_of_right_tri})$ is commutative, and hence $(\ref{to_be_a_morph_of_right_tri})$ is an isomorphism of right triangles. Thus, $X_+\xrightarrow{\overline{f}_+}Y_+\xrightarrow{\overline{g}_+}Z_+\xrightarrow{\tau_X\circ \overline{v}_+}\Sigma (X_+)$ belongs to $\vartriangleright$.
\end{proof}

We now prove that $(\CH/[\CW],\Sigma,\vartriangleright)$ is a right triangulated category, as stated in Proposition~\ref{prop:right_triangulated}.
By definition, $\vartriangleright$ is closed under isomorphisms of right triangles.
Thus it remains to verify axioms {\rm (RT1)}--{\rm (RT4)}, which are dual to
axioms {\rm (LT1)}--{\rm (LT4)} in \cite[Definition~2.2]{BM94}.
The verification is carried out in the following lemmas.

\begin{lemma}\label{lem:RTR1}
The following holds.
\begin{enumerate}
\item For any $X\in\CH/[\CW]$, the sequence $0\to X\xrightarrow{\id}X\to 0$ belongs to $\vartriangleright$.
\item For any $\mathbf{f}\in(\CH/[\CW])(X,Y)$, there exists an object in $\vartriangleright$ of the form $X\xrightarrow{\mathbf{f}}Y\xrightarrow{\mathbf{g}}Z\xrightarrow{\mathbf{h}}\Sigma X$.
\end{enumerate}
\end{lemma}
\begin{proof}
{\rm (1)} Indeed, $0\to X\xrightarrow{\id}X\to 0$ is isomorphic to the standard right triangle associated to $0\to X$.

{\rm (2)} For $f\in\CH(X,Y)$ such that $\mathbf{f}=\overline{f}$, the standard right triangle associated to $f$ gives such a right triangle.
\end{proof}

\begin{lemma}\label{lem:RTR2}
If a sequence of morphisms
$
X\xrightarrow{\mathbf{f}}Y\xrightarrow{\mathbf{g}}Z\xrightarrow{\mathbf{h}}\Sigma X
$
in $\CH/[\CW]$ belongs to $\vartriangleright$, then so does
$
Y\xrightarrow{\mathbf{g}}Z\xrightarrow{\mathbf{h}}\Sigma X\xrightarrow{-\Sigma\mathbf{f}}\Sigma Y
$.
\end{lemma}
\begin{proof}
We may assume that $X\xrightarrow{\mathbf{f}}Y\xrightarrow{\mathbf{g}}Z\xrightarrow{\mathbf{h}}\Sigma X$ is a standard right triangle $(\ref{standard_right_triangle})$ associated to a morphism $f\in\CH(X,Y)$. We use the notation in Definition~\ref{def:standard_right_triangle}. In particular we have $Z=(M_f)_+$. By its definition, there exist morphisms $u\in\CT(U_{M_f},M_f)$ and $u'\in\CT(U'_{M_f}[-1],U_{M_f})$ with $U_{M_f},U'_{M_f}\in\CU$ and a distinguished triangle
\[
U'_{M_f}[-1]\xrightarrow{u\circ u'}M_f\xrightarrow{l_{M_f}}Z\to U'_{M_f}
\]
in $\CT$.
As in Definition~\ref{def:standard_right_triangle}, we take $h_f$ so that $(\ref{octahedron_h_f})$ is commutative and that $Y\xrightarrow{g_f}M_f\xrightarrow{h_f}X\langle1\rangle\xrightarrow{s_f}Y[1]$ is a distinguished triangle in $\CT$. In particular we have $s_f=-f[1]\circ c_X$ in $\CT$, and hence $\overline{s_f}=-\overline{c_Y}\circ\overline{f}\langle1\rangle$ in $\CT/[\CW]$.

Let 
$Y\xrightarrow{a_Y}W^Y\xrightarrow{b_Y}Y\langle 1\rangle\xrightarrow{c_Y}Y[1]$
be the distinguished triangle in $\CT$ with $W^Y\in\CW$ and $Y\langle1\rangle\in\CU_{-n+1}^0$, as in Definition~\ref{def:functor_<1>}.
Since $Y\xrightarrow{g_f}M_f\xrightarrow{h_f}X\langle1\rangle\xrightarrow{s_f}Y[1]$ is a distinguished triangle and 
\color{blue}
$\CT((X\langle1\rangle)[-1],W^Y)=0$, 
\normalcolor
there exists $a\in\CT(M_f,W^Y)$ such that $a\circ g_f=a_Y$. 
By Lemma~\ref{lem:right_triangle_replace} applied to $Y\xrightarrow{g_f}M_f\xrightarrow{h_f}X\langle1\rangle\xrightarrow{s_f}Y[1]$, we see that
\[
Y_+\xrightarrow{(\overline{g_f})_+}(M_f)_+\xrightarrow{(\overline{h_f})_+}\Sigma X\xrightarrow{\tau_Y\circ\overline{v}_+}\Sigma(Y_+)
\] belongs to $\vartriangleright$.
By the definition of $\tau_Y$, we have $\tau_Y=(\overline{l_Y}\langle1\rangle)_+=\Sigma(\overline{l_Y})$ since $Y\in\CH$. Here, $\overline{v}_+$ is the morphism determined by the equation $H(c_Y)\circ H(v)=H(s_f)$. Since $\overline{s_f}=-\overline{c_Y}\circ\overline{f}\langle1\rangle$ holds in $\CT/[\CW]$, we have $\overline{v}_+=-(\overline{f}\langle1\rangle)_+=-\Sigma\overline{f}$.
Thus, 
\[
\begin{tikzcd}[row sep=0.8cm, column sep=1.2cm]
Y\arrow[r, "\overline{l_{M_f}\circ g_f}=\mathbf{g}"]
  \arrow[d, "\overline{l_Y}" swap, "\cong"]
&Z\arrow[d, double, no head]
  \arrow[r, "\mathbf{h}"]
&\Sigma X\arrow[d, double, no head]
  \arrow[r, "-\Sigma\overline{f}"]
&\Sigma Y\arrow[d, "\cong" swap, "\Sigma(\overline{l_Y})"]
\\
Y_+\arrow[r, swap, "(\overline{g_f})_+"]
&(M_f)_+ \arrow[r, swap, "(\overline{h_f})_+"]
&\Sigma X \arrow[r, swap, "\tau_Y\circ\overline{v}_+"]
&\Sigma(Y_+)
\end{tikzcd}
\]
gives an isomorphism of right triangles in $\CH/[\CW]$, hence $Y\xrightarrow{\mathbf{g}}Z\xrightarrow{\mathbf{h}}\Sigma X\xrightarrow{-\Sigma\mathbf{f}}\Sigma Y$ belongs to $\vartriangleright$.
\normalcolor
\end{proof}

\begin{lemma}\label{lem:RTR3}
Assume that $
X\xrightarrow{\mathbf{f}}Y\xrightarrow{\mathbf{g}}Z\xrightarrow{\mathbf{h}}\Sigma X
$
and 
$
X'\xrightarrow{\mathbf{f}'}Y'\xrightarrow{\mathbf{g}'}Z'\xrightarrow{\mathbf{h}'}\Sigma X'
$
belong to $\vartriangleright$.
If $\mathbf{x}\in(\CH/[\CW])(X,X')$ and $\mathbf{y}\in(\CH/[\CW])(Y,Y')$ satisfy
$\mathbf{f}'\circ \mathbf{x}=\mathbf{y}\circ \mathbf{f}$, then there exists $\mathbf{z}\in(\CH/[\CW])(Z,Z')$
that makes 
\[
\begin{tikzcd}[row sep=0.8cm, column sep=1.2cm]
X\arrow[r, "\mathbf{f}"]
\arrow[d, swap, "\mathbf{x}" ]
&Y\arrow[r, "\mathbf{g}"]
\arrow[d, "\mathbf{y}"]
&Z \arrow[r, "\mathbf{h}"]
\arrow[d, "\mathbf{z}"]
&\Sigma X
\arrow[d, "\Sigma\mathbf{x}"]
\\
X'\arrow[r, swap, "\mathbf{f}'"]
&Y'\arrow[r, swap, "\mathbf{g}'"]
&Z' \arrow[r, swap, "\mathbf{h}'"]
&\Sigma X'
\end{tikzcd}
\]
a morphism of right triangles.
\end{lemma}
\begin{proof}
By replacing each right triangle with an isomorphic standard one, it suffices to check the condition for standard right triangles. 
Let $(\ref{standard_right_triangle})$ and
\[
X'\xrightarrow{\overline{f'}}Y'\xrightarrow{\overline{l_{M_{f'}}}\circ \overline{g_{f'}}}(M_{f'})_+\xrightarrow{\overline{(h_{f'})}_+}\Sigma X'
\]
be the standard right triangles associated to $f\in\CH(X,Y)$ and $f'\in\CH(X',Y')$, respectively. We use the same notation introduced in Definition~\ref{def:standard_right_triangle} for $f$, and denote the corresponding objects and morphisms associated to $f'$ by replacing $f$ with $f'$ in the notation. Suppose that $x\in\CH(X,X')$ and $y\in\CH(Y,Y')$ \darkblueedit{make} the left square of the following diagram
\[
\begin{tikzcd}[row sep=0.8cm, column sep=1.2cm]
X\arrow[r, "\overline{f}"]
\arrow[d, swap, "\overline{x}" ]
&Y\arrow[r, "\overline{l_{M_f}}\circ \overline{g_f}"]
\arrow[d, "\overline{y}"]
&(M_f)_+ \arrow[r, "(\overline{h_f})_+"]
&\Sigma X
\arrow[d, "\Sigma\overline{x}"]
\\
X'\arrow[r, swap, "\overline{f'}"]
&Y'\arrow[r, swap, "\overline{l_{M_{f'}}}\circ\overline{g_{f'}}"]
&(M_{f'})_+ \arrow[r, swap, "(\overline{h_{f'}})_+"]
&\Sigma X'
\end{tikzcd}
\]
commutative in $\CH/[\CW]$. This means $f'\circ x-y\circ f\in[\CW]$. Thus it factors through $a_X$, namely, there exists $y_1\in\CT(W^X,Y')$ such that $f'\circ x-y\circ f=y_1\circ a_X$ in $\CT$. Also, there is $y_2\in\CT(W^X,W^{X'})$ such that $a_{X'}\circ x=y_2\circ a_X$ in $\CT$.

Put $q=\begin{bmatrix}y&y_1\\0&y_2\end{bmatrix}\in\CT(Y\oplus W^X,Y'\oplus W^{X'})$.
Then, in the following diagram whose rows are distinguished triangles, the leftmost square is commutative. Thus there exists $z\in\CT(M_f,M_{f'})$ that makes
\begin{equation}\label{morph_of_tri_XYW}
\begin{tikzcd}[row sep=0.8cm, column sep=1.2cm]
X\arrow[r, "\begin{bmatrix}f\\ a_X\end{bmatrix}"]
\arrow[d, swap, "x" ]
&Y\oplus W^X\arrow[r, "{[}g_f\ w_f{]}"]
\arrow[d, "q"]
&M_f \arrow[r, "k_{f}"]
\arrow[d, "z"]
&X[1]
\arrow[d, "x{[}1{]}"]
\\
X'\arrow[r, swap, "\begin{bmatrix}f'\\ a_{X'}\end{bmatrix}"]
&Y'\oplus W^{X'}\arrow[r, swap, "{[}g_{f'}\ w_{f'}{]}"]
&M_{f'} \arrow[r, swap, "k_{f'}"]
&X'[1]
\end{tikzcd}
\end{equation}
a morphism of distinguished triangles in $\CT$.
In particular we have $z\circ g_f=g_{f'}\circ y$ in $\CT$. 

It suffices to show that
\begin{equation}\label{resulting_morph_of_right_triangles}
\begin{tikzcd}[row sep=0.8cm, column sep=1.2cm]
X\arrow[r, "\overline{f}"]
\arrow[d, swap, "\overline{x}" ]
&Y\arrow[r, "\overline{l_{M_f}}\circ \overline{g_f}"]
\arrow[d, "\overline{y}"]
&(M_f)_+ \arrow[r, "(\overline{h_f})_+"]
\arrow[d, "\overline{z}_+"]
&\Sigma X
\arrow[d, "\Sigma\overline{x}"]
\\
X'\arrow[r, swap, "\overline{f'}"]
&Y'\arrow[r, swap, "\overline{l_{M_{f'}}}\circ\overline{g_{f'}}"]
&(M_{f'})_+ \arrow[r, swap, "(\overline{h_{f'}})_+"]
&\Sigma X'
\end{tikzcd}
\end{equation}
is a morphism of right triangles in $\CH/[\CW]$.
Commutativity of the middle square is immediate from $z\circ g_f=g_{f'}\circ y$ in $\CT$. %
To show the commutativity of the rightmost square of $(\ref{resulting_morph_of_right_triangles})$, note that the commutativity of the rightmost square of $(\ref{morph_of_tri_XYW})$ is equivalent to the commutativity of the following diagram.
\[
\begin{tikzcd}[row sep=0.8cm, column sep=1.0cm]
&M_f \arrow[r, "h_{f}"]
\arrow[d, swap, "z"]
&X\langle1\rangle \arrow[r, "c_X"]
&X[1]
\arrow[d, "x{[}1{]}"]
\\
&M_{f'} \arrow[r, swap, "h_{f'}"]
&X'\langle1\rangle \arrow[r, swap, "c_{X'}"]
&X'[1]
\end{tikzcd}
\]
We remark that $\overline{x}\langle1\rangle\in(\CT/[\CW])(X\langle1\rangle,X'\langle1\rangle)$ is given by $\overline{x}\langle1\rangle=\overline{y_x}$, where $y_x\in\CT(X\langle1\rangle,X'\langle1\rangle)$ is a morphism such that $x[1]\circ c_X=c_{X'}\circ y_x$ in $\CT$. Since $H(c_{X'})$ is monomorphic by Lemma~\ref{lem:H(f)_monom}, the commutativity of the above diagram implies that the rightmost square of the diagram $(\ref{resulting_morph_of_right_triangles})$ is also commutative. 
\end{proof}

\begin{lemma}\label{lem:RTR4}
Assume that $X\xrightarrow{\mathbf{f}}Y\xrightarrow{\mathbf{g}}Z\xrightarrow{\mathbf{h}}\Sigma X$ and 
$X'\xrightarrow{\mathbf{f}'}Y'\xrightarrow{\mathbf{g}'}Z'\xrightarrow{\mathbf{h}'}\Sigma Y$
belong to $\vartriangleright$, and satisfy $Y=X'$.
Then \darkblueedit{there exists a right triangle}
\color{blue}
\[
X\xrightarrow{\mathbf{f'}\circ\mathbf{f}}Y'\xrightarrow{\mathbf{g^{\prime\prime}}}Z^{\prime\prime}\xrightarrow{\mathbf{h}^{\prime\prime}}\Sigma X
\]
\normalcolor
\darkblueedit{in $\vartriangleright$. Moreover, there exist morphisms} $\mathbf{d}\in(\CH/[\CW])(Z,Z^{\prime\prime})$, $\mathbf{z}\in(\CH/[\CW])(Z^{\prime\prime}, Z')$
such that
\[
\begin{tikzcd}[row sep=0.9cm, column sep=1.2cm]
{}
X
\arrow{r}{\mathbf{f}}
\arrow[equal]{d}{}
&Y
\arrow{r}{\mathbf{g}}
\arrow{d}{\mathbf{f}'} 
&
Z
\arrow{r}{\mathbf{h}}
\arrow{d}{\mathbf{d}}
&
\Sigma X \arrow[equal]{d}{}
\\
X
\arrow{r}[swap]{\mathbf{f}'\circ\mathbf{f}}
&
Y' 
\arrow{r}[swap]{\mathbf{g}^{\prime\prime}}
\arrow{d}{\mathbf{g}'} 
&Z^{\prime\prime}
\arrow{r}[swap]{\mathbf{h}^{\prime\prime}}
\arrow{d}{\mathbf{z}}
&\Sigma X
\arrow{d}{\Sigma(\mathbf{f}'\circ\mathbf{f})}
\\
&Z'
\arrow[equal]{r}
\arrow{d}{\mathbf{h}'}
&Z'
\arrow{r}[swap]{(\Sigma\mathbf{f}')\circ\mathbf{h}'}
\arrow{d}{(\Sigma\mathbf{g})\circ\mathbf{h}'}
&\Sigma Y'
\\
&\Sigma Y
\arrow[swap]{r}{\Sigma\mathbf{g}}
&\Sigma Z
\end{tikzcd}
\]
is commutative in $\CH/[\CW]$ and 
$Z\xrightarrow{\mathbf{d}}Z^{\prime\prime}\xrightarrow{\mathbf{z}}Z'\xrightarrow{(\Sigma\mathbf{g})\circ\mathbf{h}'}\Sigma Z$
belongs to $\vartriangleright$.
\end{lemma}
\begin{proof}
We may assume that 
$X\xrightarrow{\mathbf{f}}Y\xrightarrow{\mathbf{g}}Z\xrightarrow{\mathbf{h}}\Sigma X$ and 
$Y\xrightarrow{\mathbf{f}'}Y'\xrightarrow{\mathbf{g}'}Z'\xrightarrow{\mathbf{h}'}\Sigma Y$
are standard right triangles
associated to $f\in\CH(X,Y)$ and $f'\in\CH(X',Y')$, respectively. We use the same notation introduced in Definition~\ref{def:standard_right_triangle} for $f$, and denote the corresponding objects and morphisms associated to $f'$ by replacing $f$ with $f'$ in the notation, while also taking into account the fact that $Y=X'$.
In particular, we have morphisms of distinguished triangles
\[
\begin{tikzcd}[row sep=0.8cm, column sep=1.2cm]
X \arrow{r}{\begin{bmatrix}f\\ a_X\end{bmatrix}}\arrow[d, equal]
&Y\oplus W^X \arrow{d}{{[}0\ \id{]}} \arrow{r}{{[}g_f\ w_f{]}}
&M_f\arrow{r}{k_f} \arrow{d}{h_f}
&X{[}1{]} \arrow[d, equal]
\\
X\arrow{r}[swap]{a_X}
&W^X\arrow{r}[swap]{b_X}
&X \langle1\rangle \arrow{r}[swap]{c_X}
&X{[}1{]}
\end{tikzcd}\]
and
\[
\begin{tikzcd}[row sep=0.8cm, column sep=1.2cm]
Y \arrow{r}{\begin{bmatrix}f'\\ a_Y\end{bmatrix}}\arrow[d, equal]
&Y'\oplus W^Y \arrow{d}{{[}0\ \id{]}} \arrow{r}{{[}g_{f'}\ w_{f'}{]}}
&M_{f'}\arrow{r}{k_{f'}} \arrow{d}{h_{f'}}
&Y{[}1{]} \arrow[d, equal]
\\
Y\arrow{r}[swap]{a_Y}
&W^Y\arrow{r}[swap]{b_Y}
&Y \langle1\rangle \arrow{r}[swap]{c_Y}
&Y{[}1{]}
\end{tikzcd}
\]
in $\CT$.
Then,
\[
Y\oplus W^X\xrightarrow{{\scriptsize\begin{bmatrix}f'&0\\ a_Y&0\\0&\id\end{bmatrix}}}Y'\oplus W^Y\oplus W^X\xrightarrow{[g_{f'}\ w_{f'}\ 0]}M_{f'}\xrightarrow{k'} (Y\oplus W^X)[1]
\]
is a distinguished triangle in $\CT$, where $k'$ denotes the composite of $\begin{bmatrix}k_{f'}\\0\end{bmatrix}$ with the canonical isomorphism $Y[1]\oplus W^X[1]\xrightarrow{\cong}(Y\oplus W^X)[1]$.

On the other hand, we have a standard right triangle
\color{blue}
\[
X\xrightarrow{\overline{f^{\prime\prime}}}Y'\xrightarrow{\overline{l_{M_{f^{\prime\prime}}}}\darkblueedit{\circ}\overline{g_{f^{\prime\prime}}}}(M_{f^{\prime\prime}})_+\xrightarrow{\overline{(h_{f^{\prime\prime}}})_+}\Sigma X
\]
\normalcolor
associated to $f^{\prime\prime}=f'\circ f\in\CH(X,Y')$, where
\[
\begin{tikzcd}[row sep=0.8cm, column sep=1.2cm]
X \arrow{r}{\begin{bmatrix}f^{\prime\prime}\\ a_X\end{bmatrix}}\arrow[d, equal]
&Y'\oplus W^X \arrow{d}{{[}0\ \id{]}} \arrow{r}{{[}g_{f^{\prime\prime}}\ w_{f^{\prime\prime}}{]}}
&M_{f^{\prime\prime}}\arrow{r}{k_{f^{\prime\prime}}} \arrow{d}{h_{f^{\prime\prime}}}
&X{[}1{]} \arrow[d, equal]
\\
X\arrow{r}[swap]{a_X}
&W^X\arrow{r}[swap]{b_X}
&X \langle1\rangle \arrow{r}[swap]{c_X}
&X{[}1{]}
\end{tikzcd}
\]
is a morphism of distinguished triangles in $\CT$.
Then, 
\[
X
\xrightarrow{{\scriptsize\begin{bmatrix}f^{\prime\prime}\\a_X\\0\end{bmatrix}}}
Y'\oplus W^X\oplus W^Y
\xrightarrow{{\scriptsize\begin{bmatrix}g_{f^{\prime\prime}}&w_{f^{\prime\prime}}&0\\ 0&0&\id\end{bmatrix}}}
M_{f^{\prime\prime}}\oplus W^Y
\xrightarrow{[k_{f^{\prime\prime}}\ 0]}
X[1]
\]
is also a distinguished triangle.
We note that there exists a morphism $w\in\CT(W^X,W^Y)$ that satisfies $w\circ a_X=a_Y\circ f$ in $\CT$.
By composing an isomorphism
\[
\begin{bmatrix}
\id&0&0\\
0&w&\id\\
0&\id&0
\end{bmatrix}\in\CT(Y'\oplus W^X\oplus W^Y,Y'\oplus W^Y\oplus W^X)
\]
to the above distinguished triangle, we obtain a distinguished triangle
\[
X
\xrightarrow{{\scriptsize\begin{bmatrix}f^{\prime\prime}\\a_Y\circ f\\a_X\end{bmatrix}}}
Y'\oplus W^Y\oplus W^X
\xrightarrow{{\scriptsize\begin{bmatrix}g_{f^{\prime\prime}}&0&w_{f^{\prime\prime}}\\ 0&\id&-w\end{bmatrix}}}
M_{f^{\prime\prime}}\oplus W^Y
\xrightarrow{[k_{f^{\prime\prime}}\ 0]}
X[1]
\]
in $\CT$. By the octahedron axiom in $\CT$, there exist
morphisms $\begin{bmatrix}d\\d'\end{bmatrix}\in\CT(M_f,M_{f^{\prime\prime}}\oplus W^Y)$ and $[z\ z']\in\CT(M_{f^{\prime\prime}}\oplus W^Y,M_{f'})$ such that
\begin{equation}\label{octa_XYWMX}
\begin{tikzcd}[row sep=1.2cm, column sep=1.4cm]
{}
X
\arrow{r}{{\scriptsize\begin{bmatrix}f\\a_X\end{bmatrix}}}
\arrow[equal]{d}{}
&Y\oplus W^X
\arrow{r}{[g_f\ w_f]}
\arrow{d}{{\scriptsize\begin{bmatrix}f'\ \,0\\a_Y\ 0\\0\ \ \ \id\end{bmatrix}}} 
&
M_f
\arrow{r}{k_f}
\arrow{d}{{\scriptsize\begin{bmatrix}d\\d'\end{bmatrix}}}
&
X{[}1{]} \arrow[equal]{d}
\\
X
\arrow{r}[swap]{{\scriptsize\begin{bmatrix}f'\circ f\\a_Y\circ f\\ a_X\end{bmatrix}}}
&
Y'\oplus W^Y\oplus W^X 
\arrow{r}[swap]{{\scriptsize\begin{bmatrix}g_{f^{\prime\prime}}\ 0\ \ w_{f^{\prime\prime}}\\ 0\ \ \ \ \id\ -w\end{bmatrix}}}
\arrow{d}[swap]{{[}g_{f'}\ w_{f'}\ 0{]}} 
&M_{f^{\prime\prime}}\oplus W^Y
\arrow{r}[swap]{{[}k_{f^{\prime\prime}}\ 0{]}}
\arrow{d}{[z\ z']}
&X{[}1{]}
\arrow{d}{{\scriptsize\begin{bmatrix}f\\a_X\end{bmatrix}}{[}1{]}}
\\
&M_{f'}
\arrow[equal]{r}
\arrow{d}[swap]{k'}
&M_{f'}
\arrow{r}[swap]{k'}
\arrow{d}{({[}g_f\ w_f{]}{[}1{]})\circ k'=g_f{[}1{]}\circ k_{f'}}
&(Y\oplus W^X){[}1{]}
\\
&(Y\oplus W^X){[}1{]}
\arrow[swap]{r}{{[}g_f\ w_f{]}{[}1{]}}
&M_f{[}1{]}
\end{tikzcd}
\end{equation}
is commutative in $\CT$,
and \begin{equation}\label{dt_MWMM}
M_f\xrightarrow{{\scriptsize\begin{bmatrix}d\\d'\end{bmatrix}}}M_{f^{\prime\prime}}\oplus W^Y\xrightarrow{[z\ z']}M_{f'}\xrightarrow{g_f[1]\circ k_{f'}}M_f[1]
\end{equation}
is a distinguished triangle.
In particular we have $k_{f^{\prime\prime}}\circ d=k_f$ and $d'\circ g_f=a_Y$ in $\CT$.

We check that the triangle $(\ref{dt_MWMM})$
satisfies the assumptions of Lemma~\ref{lem:right_triangle_replace}.
Note that there exists a morphism 
\[
\begin{tikzcd}[row sep=0.8cm, column sep=1.2cm]
Y \arrow{r}{a_Y}
\arrow{d}[swap]{g_f}
&
W^Y 
\arrow{r}{b_Y}
\arrow{d}{w'} 
&Y\langle1\rangle
\arrow{r}{c_Y}
\arrow{d}{g'}
&Y{[}1{]}
\arrow{d}{g_f{[}1{]}}
\\
M_f\arrow{r}[swap]{a_{M_f}}
&W^{M_f}\arrow{r}[swap]{b_{M_f}}
&M_f\langle1\rangle \arrow{r}[swap]{c_{M_f}}
&M_f{[}1{]}
\end{tikzcd}
\]
of distinguished triangles in $\CT$ that gives $\overline{g'}=\overline{g_f}\langle1\rangle$.
Since $Y\xrightarrow{g_f}M_f\xrightarrow{h_f}X\langle1\rangle\xrightarrow{s_f}Y[1]$ is a distinguished triangle and  
$(a_{M_f}-w'\circ d')\circ g_f=w'\circ a_Y-w'\circ a_Y=0$, 
there exists $a'\in\CT(X\langle1\rangle,W^{M_f})$ such that $a_{M_f}-w'\circ d'=a'\circ h_f$. If we put 
$a=[a'\circ h_{f^{\prime\prime}}\ \, w']\in\CT(M_{f^{\prime\prime}}\oplus W^Y,W^{M_f})$,
then it satisfies
\begin{eqnarray*}
a\circ \begin{bmatrix}d\\ d'\end{bmatrix}
&=&a'\circ h_{f^{\prime\prime}}\circ d+w'\circ d'
\ =\ a'\circ c_X\circ k_{f^{\prime\prime}}\circ d+w'\circ d'\\
&=&a'\circ c_X\circ k_f+w'\circ d'
\ =\ a'\circ h_f+w'\circ d'
\ =\ a_{M_f}.
\end{eqnarray*}
Thus, by Lemma~\ref{lem:right_triangle_replace}, it follows that
\[
(M_f)_+\xrightarrow{\overline{d}_+}(M_{f^{\prime\prime}})_+\xrightarrow{\overline{z}_+}
(M_{f'})_+\xrightarrow{\tau_{M_f}\circ \overline{v}_+}
\Sigma ((M_f)_+)
\]
belongs to $\vartriangleright$, where $\overline{v}_+\in(\CH/[\CW])((M_{f'})_+,(M_f\langle1\rangle)_+)$ is the unique morphism satisfying $H(c_{M_f})\circ H(v)=H(g_f[1]\circ k_{f'})$.
However, since $g'\circ h_{f'}$ satisfies
\[
H(c_{M_f})\circ H(g'\circ h_{f'})=H(g_f[1])\circ H(c_Y)\circ H(h_{f'})=H(g_f[1]\circ k_{f'}),
\]
we have
\begin{equation}\label{eq_ghv}
\overline{v}_+=\overline{(g'\circ h_{f'})}_+
=(\overline{g'})_+\circ \overline{(h_{f'})}_+
=(\overline{g_f}\langle1\rangle)_+\circ \overline{(h_{f'})}_+
\end{equation}
by the uniqueness.

It remains to show that the diagram
\[
\begin{tikzcd}[row sep=1.2cm, column sep=1.4cm]
{}
X
\arrow{r}{\overline{f}}
\arrow[equal]{d}
&Y
\arrow{r}{\overline{l_{M_f}}\circ\overline{g_f}}
\arrow{d}{\overline{f'}} 
&
(M_f)_+
\arrow{r}{\overline{(h_f)}_+}
\arrow{d}{\overline{d}_+}
&
\Sigma X \arrow[equal]{d}
\\
X
\arrow{r}[swap]{\overline{f'}\circ\overline{f}}
&
Y' 
\arrow{r}[swap]{\overline{l_{M_{f^{\prime\prime}}}}\circ\overline{g_{{f^{\prime\prime}}}}}
\arrow{d}[swap]{\overline{l_{M_{f'}}}\circ\overline{g_{f'}}} 
&(M_{f^{\prime\prime}})_+
\arrow{r}{\overline{(h_{f^{\prime\prime}})_+}}
\arrow{d}{\overline{z}_+}
&\Sigma X
\arrow{d}{\Sigma\overline{f}}
\\
&(M_{f'})_+
\arrow[equal]{r}
\arrow{d}[swap]{\overline{(h_{f'})}_+}
&(M_{f'})_+
\arrow{r}{\overline{(h_{f'})}_+}
\arrow{d}{\tau_{M_f}\circ \overline{v}_+}
&\Sigma Y
\\
&\Sigma Y
\arrow[swap]{r}{\Sigma (\overline{l_{M_f}}\circ\overline{g_f})}
&\Sigma ((M_f)_+)
\arrow[from=1-1,to=2-2, phantom, "\mathrm{(i)}"{pos=.5}]
\arrow[from=1-2,to=2-3, phantom, "\mathrm{(ii)}"{pos=.5}]
\arrow[from=1-3,to=2-4, phantom, "\mathrm{(iii)}"{pos=.5}]
\arrow[from=2-2,to=3-3, phantom, "\mathrm{(iv)}"{pos=.5}]
\arrow[from=2-3,to=3-4, phantom, "\mathrm{(v)}"{pos=.5}]
\arrow[from=3-2,to=4-3, phantom, "\mathrm{(vi)}"{pos=.5}]
\end{tikzcd}
\]
is commutative in $\CH/[\CW]$. 
Commutativity of the part marked as {\rm (i)} is obvious. 
Commutativity of the parts marked as \darkblueedit{{\rm (ii), (iv)}} follows immediately from the commutativity of the corresponding parts of $(\ref{octa_XYWMX})$ and the definition of $(-)_+$.
Commutativity of the parts marked as \darkblueedit{{\rm (iii), (v)}} also follows from the commutativity of the corresponding parts of $(\ref{octa_XYWMX})$, since $H(c_X)$ and $H(c_Y)$ are monomorphisms in $\CH/[\CW]$ by Lemma~\ref{lem:H(f)_monom}.
By $(\ref{eq_ghv})$, commutativity of the part marked as {\rm (vi)} follows since
\[
\begin{tikzcd}[row sep=1.2cm, column sep=1.8cm]
(M_{f'})_+ \arrow[r, "\overline{v}_+"]
\arrow{d}[swap]{\overline{(h_{f'})}_+}
&(M_f\langle1\rangle)_+
\arrow{d}{\tau_{M_f}=(\overline{(l_{M_f})}\langle1\rangle)_+}
\\
\Sigma Y \arrow[r, swap, "\Sigma(\overline{l_{M_f}}\circ\overline{g_f})"]
&\Sigma ((M_f)_+)
\arrow[from=2-1,to=1-2, "(\overline{g_f}\langle1\rangle)_+"{pos=.5}]
\end{tikzcd}
\]
is commutative in $\CH/[\CW]$.
\normalcolor
\end{proof}

\begin{proposition}\label{prop:right_triangulated}
$(\CH/[\CW],\Sigma,\vartriangleright)$ is a right triangulated category.
\end{proposition}
\begin{proof}
This follows from Lemmas~\ref{lem:RTR1}, \ref{lem:RTR2}, \ref{lem:RTR3} and \ref{lem:RTR4}, where
conditions {\rm (RT1)}, {\rm (RT2)}, {\rm (RT3)} and {\rm (RT4)}
are verified, respectively.
\end{proof}

To conclude this section, we prove that $(\Sigma,\Omega,\vartriangleright,\vartriangleleft)$ is a pretriangulated structure on $\CH/[\CW]$ (Proposition~\ref{prop:pretriangulated}). The following lemma, together with its dual, completes the proof.
\begin{lemma}\label{lem:PT}
Assume that $X\xrightarrow{\mathbf{f}}Y\xrightarrow{\mathbf{g}}Z\xrightarrow{\mathbf{h}}\Sigma X$
belongs to $\vartriangleright$, and 
\color{blue}
that 
\normalcolor
$\Omega Z'\xrightarrow{\mathbf{e}'}X'\xrightarrow{\mathbf{f}'}Y'\xrightarrow{\mathbf{g}'}Z'
$ 
\color{blue}
belongs 
\normalcolor
to $\vartriangleleft$.
If $\mathbf{y}\in(\CH/[\CW])(Z,Y')$ and $\mathbf{z}\in(\CH/[\CW])(\Sigma X,Z')$ satisfy
$\mathbf{g}'\circ \mathbf{y}=\mathbf{z}\circ \mathbf{h}$, then there exists $\mathbf{x}\in(\CH/[\CW])(Y,X')$
that makes 
\[
\begin{tikzcd}[row sep=0.8cm, column sep=1.0cm]
X\arrow[r, "\mathbf{f}"]
\arrow[d, swap, "\mathbf{z}^{\dagger}" ]
&Y\arrow[r, "\mathbf{g}"]
\arrow[d, "\mathbf{x}"]
&Z \arrow[r, "\mathbf{h}"]
\arrow[d, "\mathbf{y}"]
&\Sigma X
\arrow[d, "\mathbf{z}"]
\\
\Omega Z'\arrow[r, swap, "\mathbf{e}'"]
&X'\arrow[r, swap, "\mathbf{f}'"]
&Y' \arrow[r, swap, "\mathbf{g}'"]
&Z'
\end{tikzcd}
\]
commutative in $\CH/[\CW]$. Here, 
$\mathbf{z}^{\dagger}=(\Omega\mathbf{z})\circ\eta_X=\Theta_{X,Z'}(\mathbf{z})$, where $\Theta$ and $\eta$ are those %
obtained in Proposition~\ref{prop:adjoint_Sigma_Omega}.
\normalcolor

\end{lemma}
\begin{proof}
By replacing with isomorphic standard ones, it suffices to check the condition for the standard right triangle
\[
X\xrightarrow{\overline{f}}Y\xrightarrow{\overline{l_{M_f}}\circ \overline{g_f}}(M_f)_+\xrightarrow{\overline{(h_f)}_+}\Sigma X
\]
associated to $f\in\CH(X,Y)$, and the standard left triangle

\begin{equation}\label{st_left_tri_Z'}
\Omega Z'\xrightarrow{\overline{p}_-}K_{g'}\xrightarrow{\overline{q}\circ \overline{r_{K_{g'}}}}Y'\xrightarrow{\overline{g'}}Z'
\end{equation}
associated to $g'\in\CH(Y',Z')$.
By the construction dual to Definition~\ref{def:standard_right_triangle}, the sequence $(\ref{st_left_tri_Z'})$ is induced from 
the distinguished triangle \eqref{dtriXXWX} for $Z'$, which we denote by
\[
Z'[-1]\xrightarrow{c'}Z'\langle-1\rangle\xrightarrow{b'}W'\xrightarrow{a'}Z'
\]
briefly here, and a morphism
\[
\begin{tikzcd}[row sep=0.9cm, column sep=1.1cm]
Z'\langle-1\rangle
\arrow{r}{b'}
\arrow[d, swap, "p"]
&W'
\arrow{r}{a'}
\arrow{d}{{\begin{bmatrix}0\\ \id\end{bmatrix}}} 
&Z'\arrow{r}{t'} \arrow[d, equal]
&(Z'\langle-1\rangle){[}1{]} 
\arrow[d, "p{[}1{]}"]
\\
K_{g'}
\arrow{r}[swap]{{\begin{bmatrix}q\\ k\end{bmatrix}}}
&Y'\oplus W'
\arrow{r}[swap]{{{[}g'\ a'{]}}}
&Z' \arrow{r}[swap]{d}
&K_{g'}{[}1{]}
\end{tikzcd}
\]
of distinguished triangles in $\CT$, where we put $t'=-c'[1]$.

Suppose that we are given morphisms $y\in \CH(Z,Y')$ and $z\in \CH(\Sigma X,Z')$ that make the rightmost square of
\[
\begin{tikzcd}[row sep=0.8cm, column sep=1.1cm]
X
\arrow{r}{\overline{f}}
\arrow[d, swap, "\overline{z}^{\dagger}"]
&Y
\arrow{r}{\overline{l_{M_f}}\circ \overline{g_f}}
&(M_f)_+
\arrow{r}{\overline{(h_f)}_+}
\arrow[d, swap, "\overline{y}"]
&\Sigma X 
\arrow[d, "\overline{z}"]
\\
\Omega Z'
\arrow{r}[swap]{\overline{p}_-}
&(K_{g'})_-
\arrow{r}[swap]{\overline{q}\circ \overline{r_{K_{g'}}}}
&Y' 
\arrow{r}[swap]{\overline{g'}}
&Z'
\end{tikzcd}
\]
commutative in $\CH/[\CW]$. 
If we put $y'=y\circ l_{M_f}\in\CT(M_f,Y')$ and $z'=z\circ l_{X\langle1\rangle}$ for simplicity, then we have 
\color{blue}
$\overline{z'}\circ \overline{h_f}=\overline{g'}\circ\overline{y'}$
\normalcolor
in $\CT/[\CW]$. Thus $z'\circ h_f-g'\circ y'\in[\CW]$, and hence there exists $j\in\CT(M_f,W')$ such that
$z'\circ h_f-g'\circ y'=a'\circ j$. 
This means that the middle square of the diagram $(\ref{morphism_tri_YMf})$ below is commutative in $\CT$. 
Since the rows are distinguished triangles, we obtain $x'\in\CT(Y,K_{g'})$ that makes 
\begin{equation}\label{morphism_tri_YMf}
\begin{tikzcd}[row sep=1.0cm, column sep=1.0cm]
Y
\arrow{r}{g_f}
\arrow[d, swap, "x'"]
&M_f
\arrow{r}{h_f}
\arrow{d}{{\begin{bmatrix}y'\\ j\end{bmatrix}}}
&X\langle1\rangle
\arrow{r}{s_f}
\arrow[d, swap, "z'"]
&Y[1]
\arrow[d, "x'{[}1{]}"]
\\
K_{g'}
\arrow{r}[swap]{{\begin{bmatrix}q\\ k\end{bmatrix}}}
&Y'\oplus W'
\arrow{r}[swap]{{[}g'\ a'{]}}
&Z' 
\arrow{r}[swap]{d}
&K_{g'}{[}1{]}
\end{tikzcd}
\end{equation}
a morphism of distinguished triangles.
\normalcolor
Since $Y\in\CH$, there exists $x\in\CH(Y,(K_{g'})_-)$ that makes 
\[
\begin{tikzpicture}[>=stealth]
\node (1) at (0,0.8) {$Y$};
\node (2) at (1.1,-0.5) {$K_{g'}$};
\node (3) at (-1.1,-0.5) {$(K_{g'})_-$};
\draw[->] (1) -- node[right,font=\scriptsize] {$x'$} (2);
\draw[->] (1) -- node[left,font=\scriptsize] {$x$} (3);
\draw[->] (3) -- node[below,font=\scriptsize] {$r_{K_{g'}}$} (2);
\end{tikzpicture}
\]
commutative in $\CT$, by the dual of Proposition~\ref{prop:functor_+} {\rm (1)}.

It remains to show that $\overline{x}\in(\CH/[\CW])(Y,(K_{g'})_-)$ makes
\begin{equation}\label{diag_for_PTR}
\begin{tikzcd}[row sep=0.8cm, column sep=1.1cm]
X
\arrow{r}{\overline{f}}
\arrow[d, swap, "\overline{z}^{\dagger}"]
&Y
\arrow{r}{\overline{l_{M_f}}\circ \overline{g_f}}
\arrow[d, swap, "\overline{x}"]
&(M_f)_+
\arrow{r}{\overline{(h_f)}_+}
\arrow[d, swap, "\overline{y}"]
&\Sigma X 
\arrow[d, "\overline{z}"]
\\
\Omega Z'
\arrow{r}[swap]{\overline{p}_-}
&(K_{g'})_-
\arrow{r}[swap]{\overline{q}\circ \overline{r_{K_{g'}}}}
&Y' 
\arrow{r}[swap]{\overline{g'}}
&Z'
\end{tikzcd}
\end{equation}
commutative in $\CH/[\CW]$. 
The commutativity of the middle square is obvious by the construction.
Let us show the commutativity of the leftmost square. We recall that the morphism $\theta_{X,Z'}(\overline{z'})\in (\CT/[\CW])(X,Z'\langle-1\rangle)$, where $\theta$ is the isomorphism obtained in Lemma~\ref{lem:adjoint_+-}, is given by
$\theta_{X,Z'}(\overline{z'})=\overline{m}$
using a morphism $m\in\CT(X,Z'\langle-1\rangle)$ that makes
\[
\begin{tikzcd}[row sep=0.8cm, column sep=0.9cm]
X\langle1\rangle 
\arrow{r}{c_X}
\arrow{d}[swap]{z'}
&X[1] 
\arrow{d}{m{[}1{]}}
\\
Z' 
\arrow{r}[swap]{t'}
&(Z'\langle-1\rangle){[}1{]}
\end{tikzcd}
\]
commutative in $\CT$. Note that the following diagram is commutative except for the right-hand square.
\[
\begin{tikzpicture}[>=stealth]
\node (1) at (-2.3,0.8) {$X\langle1\rangle$};
\node (2) at (0,0.8) {$X{[}1{]}$};
\node (3) at (2.4,0.8) {$Y{[}1{]}$};
\node (4) at (-2.3,-0.8) {$Z'$};
\node (5) at (0,-0.8) {$(Z'\langle-1\rangle){[}1{]}$};
\node (6) at (2.4,-0.8) {$K_{g'}{[}1{]}$};
\draw[->] (1) -- node[below,font=\scriptsize] {$c_X$} (2);
\draw[->,bend left=25] (1) to node[above,font=\scriptsize] {$s_f$} (3);
\draw[->] (2) -- node[below,font=\scriptsize] {$f{[}1{]}$} (3);
\draw[->] (1) -- node[left,font=\scriptsize] {$z'$} (4);
\draw[->] (2) -- node[right,font=\scriptsize] {$m{[}1{]}$} (5);
\draw[->] (4) -- node[above,font=\scriptsize] {$t'$} (5);
\draw[->] (5) -- node[above,font=\scriptsize] {$p{[}1{]}$} (6);
\draw[->] (3) -- node[right,font=\scriptsize] {$x'{[}1{]}$} (6);
\draw[->,bend right=25] (4) to node[below,font=\scriptsize] {$d$} (6);
\end{tikzpicture}
\]
In particular we have
\[
(x'\circ f-p\circ m)[1]\circ c_X=0
\]
in $\CT$.
Since $X\xrightarrow{a_X}W^X\xrightarrow{b_X}X\langle1\rangle\xrightarrow{c_X}X[1]$ is a distinguished triangle, this equality implies that $x'\circ f-p\circ m$ factors through $W^X$, hence we obtain
$\overline{x'}\circ \overline{f}=\overline{p}\circ \overline{m}$
in $\CT/[\CW]$. This means that
\[
\overline{r_{K_{g'}}}\circ \overline{x}\circ \overline{f}
=\overline{x'}\circ \overline{f}
=\overline{p}\circ \overline{m}
=\overline{r_{K_{g'}}}\circ \overline{p}_-\circ \overline{z}^{\dagger}
\]
holds. By the dual of Proposition~\ref{prop:functor_+} {\rm (2)} we obtain $\overline{x}\circ \overline{f}=\overline{p}_-\circ \overline{z}^{\dagger}$, which shows the commutativity of the leftmost square of $(\ref{diag_for_PTR})$.
\end{proof}

\begin{proposition}\label{prop:pretriangulated}
$(\CH/[\CW],\Sigma,\Omega,\vartriangleright,\vartriangleleft)$ is 
an $n$-truncated 
\normalcolor
pretriangulated category.
\end{proposition}
\begin{proof}
By Proposition~\ref{prop:right_triangulated},
$(\CH/[\CW],\Sigma,\vartriangleright)$ is a right triangulated category.
In a similar way, we can show that $(\CH/[\CW],\Omega,\vartriangleleft)$ is a left triangulated category.
The remaining conditions {\rm (iv)} and {\rm (v)} of \cite[Definition~1.1]{BR07} follow from Lemma~\ref{lem:PT} and its dual. 
The fact that it is $n$-truncated follows from \cref{prop:functor_Sigma}. \normalcolor
\end{proof}

\subsection{The heart is abelian \texorpdfstring{$n$}{n}-truncated}
\label{subsection:nA_heart}
In this subsection, 
\color{blue}
we prove our main theorem, \cref{thm:quasitri}, which states that the heart equipped with the pretriangulated structure and the extriangulated structure obtained in the preceding results is an abelian $n$-truncated category.  
\normalcolor

\begin{lemma}\label{lemma:construction_of_dell}
Let $X,Z\in\CH$ be any pair of objects, and let $\delta\in\BE_{\CN}(Z,X)$ be any element.
Let %
$X\xrightarrow{a_X}W^X\xrightarrow{b_X}X\langle 1\rangle\xrightarrow{c_X}X[1]$
be the distinguished triangle in $\CT$ with $W^X\in\CW$, as in Definition~\ref{def:functor_<1>}.
Then, there exists a unique morphism $\wp_{\delta}\in(\CH/[\CW])(H(Z),H(X\langle1\rangle))$ that makes the diagram
\[
\begin{tikzpicture}[>=stealth]
\node (2) at (2,0.8) {$H(Z)$};
\node (4) at (-0.3,-0.8) {$H(X\langle1\rangle)$};
\node (5) at (2,-0.8) {$H(X[1])$};
\draw[->] (2) -- node[above,font=\scriptsize] {$\wp_{\delta}$} (4);
\draw[->] (4) -- node[below,font=\scriptsize] {$H(c_X)$} (5);
\draw[->] (2) -- node[right,font=\scriptsize] {$H(\delta)$} (5);
\end{tikzpicture}
\]
commutative in $\CH/[\CW]$.
\end{lemma}
\begin{proof}
Complete the morphism $-a_X[1]\circ\delta\in\CT(Z,W^X[1])$ into a distinguished triangle
\begin{equation}\label{triangle_WDZ}
W^X\to D_Z\xrightarrow{r}Z\xrightarrow{-a_X[1]\circ\delta}W^X[1]
\end{equation}
in $\CT$. Since $\delta\in\BE_{\CN}(Z,X)$, we have $-a_X[1]\circ\delta\in [\CN]$, and hence $r\in\CR$. By Proposition~\ref{prop:fully_faithful}, it follows that $H(r)$ is an isomorphism in $\CH/[\CW]$.

Since $(\ref{triangle_<1>})$ and $(\ref{triangle_WDZ})$ are distinguished triangles, \darkblueedit{there exists} $d\in\CT(D_Z,X\langle1\rangle)$ such that
\[
\begin{tikzcd}[row sep=0.7cm, column sep=0.7cm]
D_Z\arrow[r, "r"]
  \arrow[d, swap, "d"] %
&Z\arrow[d, "\delta"]
\\
X\langle1\rangle \arrow[r, swap, "c_X"]
&X[1]
\end{tikzcd}
\]
is commutative in $\CT$. If we put $\wp_{\delta}=H(d)\circ H(r)^{-1}$, then it satisfies $H(c_X)\circ \wp_{\delta}=H(\delta)$ in $\CH/[\CW]$.
Such $\wp_{\delta}$ is unique by Lemma~\ref{lem:H(f)_monom}. 
\end{proof}

\begin{definition}\label{def:construction_of_dell}
For any $X,Z\in\CH$ and any $\delta\in\BF(Z,X)$, define $\partial^+_{Z,X}(\delta)\in(\CH/[\CW])(Z,\Sigma X)$ 
to be the unique morphism that makes
\color{blue}
\[
\begin{tikzpicture}[>=stealth]
\node (0) at (-2.2,0.8) {$Z$};
\node (1) at (-0.1,0.8) {$Z_+$};
\node (2) at (2.2,0.8) {$H(Z)$};
\node (3) at (-2.2,-0.8) {$\Sigma X$};
\node (4) at (-0.1,-0.8) {$H(X\langle1\rangle)$};
\node (5) at (2.2,-0.8) {$H(X[1])$};
\draw[->] (0) -- node[above,font=\scriptsize] {$\overline{l_Z}$}node[below,font=\scriptsize] {$\cong$}  (1);
\draw[->] (2) -- node[above,font=\scriptsize] {$\overline{r_{Z_+}}$} node[below,font=\scriptsize] {$\cong$} (1);
\draw[->] (0) -- node[left,font=\scriptsize] {$\partial^+_{Z,X}(\delta)$} (3);
\draw[->] (2) -- node[right,font=\scriptsize] {$\wp_{\delta}$} (4);
\draw[->] (4) -- node[below,font=\scriptsize] {$\overline{r_{\Sigma X}}$} node[above,font=\scriptsize] {$\cong$} (3);
\draw[->] (4) -- node[below,font=\scriptsize] {$H(c_X)$} (5);
\draw[->] (2) -- node[right,font=\scriptsize] {$H(\delta)$} (5);
\end{tikzpicture}
\]
\normalcolor
commutative in $\CH/[\CW]$, where $\wp_{\delta}$ is the morphism obtained in Lemma~\ref{lemma:construction_of_dell}.
This assignment defines a map
$
\partial^+_{Z,X}\colon\BF(Z,X)\to(\CH/[\CW])(Z,\Sigma X).
$
\end{definition}

\begin{lemma}\label{lem:char_for_Sn-deflation}
Let $X\xto{f}Y\xto{g}Z\xto{h}X[1]$ be any distinguished triangle in $\CT$.
\begin{enumerate}[label=\textup{(\arabic*)}]
\item 
Assume that $Z\in\CT^-$. Then $g$ is an $\fs_\CN^R$-deflation in $\CT_\CN$ if and only if $h\in[\CU_{-n+1}^0]$.
\item 
Assume that $X\in\CT^+$. Then $f$ is an $\fs_\CN^L$-inflation in $\CT_\CN$ if and only if $h[-1]\in[\CV_0^{n-1}]$.
\end{enumerate}
\darkblueedit{In particular, for any $X\in\CT^+$ and $Z\in\CT^-$, we have an equality}
\[
\BE_\CN(Z,X) 
    =\Set{ h\in\BE(Z,X) | h\in[\CU_{-n+1}^0]\cap [\CV_1^{n}] }.
\]
\end{lemma}
\begin{proof}
Since (2) can be checked in a dual manner, we only show (1).

Suppose that $g$ is an $\fs_\CN^R$-deflation in $\CT_\CN$.
By the definition of an $n$-cotorsion pair, there exists a distinguished triangle $\wtil{U}\xto{}X[1]\xto{a}V[1]\xto{}\wtil{U}[1]$ with $\wtil{U}\in\CU_{-n+1}^0$ and $V\in\CV$.
By the dual of \cref{lem:merged} {\rm (1)}, we have $\BE^R_{\CN}(Z,V)=0$. Thus 
$a\circ h=0$, and hence 
\normalcolor
$h$ factors through $\wtil{U}\in\CU_{-n+1}^0$.

Conversely, suppose that $h\in[\CU_{-n+1}^0]$, namely, we have a factorization $h[-1]\colon Z[-1]\xto{a}\wtil{U}[-1]\xto{b}X$ with $\wtil{U}\in\CU_{-n+1}^0$.
We consider a morphism $X\xto{x}N$ to $N\in\add(\CU*\CV)$.
Then, by $\CT(\wtil{U}[-1],\CV)=0$, the composite $x\circ b$ factors through an object $U\in\CU$.
Since we see $U\in\CU[-1]*\CW$ by \cref{lem:inclusions_for_n-CT}, the condition $\CT(\wtil{U}[-1],\CW)=0$ shows that $x\circ b$ also factors through $\CU[-1]$.
This shows that $g$ is an $\fs_\CN^R$-deflation.
\end{proof}

\begin{proposition}\label{prop:construction_of_dell}
The following holds.
\begin{enumerate}
\item $\partial^+=\{\partial^+_{Z,X}\}_{Z,X\in\CH}$
forms a natural transformation $\partial^+\colon \BF\Rightarrow (\CH/[\CW])(-,\Sigma(-))$ between functors $(\CH/[\CW])^\op\times(\CH/[\CW])\to\Ab$.
\item For any $\ft$-triangle $X\xrightarrow{\mathbf{f}}Y\xrightarrow{\mathbf{g}}Z\overset{\delta}{\dashrightarrow}$ in $(\CH/[\CW],\BF,\ft)$, the sequence
\begin{equation}\label{sequence_to_belong_rtri}
X\xrightarrow{\mathbf{f}}Y\xrightarrow{\mathbf{g}}Z\xrightarrow{\partial^+_{Z,X}(\delta)}\Sigma X
\end{equation}
belongs to $\vartriangleright$.
\end{enumerate}

In a dual manner, we can define a natural transformation $\partial^-\colon \BF\Rightarrow (\CH/[\CW])(\Omega(-),-)$ such that the sequence
$\Omega Z\xrightarrow{\partial^-_{Z,X}(\delta)} X\xrightarrow{\mathbf{f}}Y\xrightarrow{\mathbf{g}}Z$
belongs to $\vartriangleleft$ for any $\ft$-triangle $X\xrightarrow{\mathbf{f}}Y\xrightarrow{\mathbf{g}}Z\overset{\delta}{\dashrightarrow}$ in $(\CH/[\CW],\BF,\ft)$.
Thus, the extriangulated structure $(\BF,\ft)$ 
\color{blue}
satisfies condition {\rm (a)} in \cref{def:compatible_ET_PT}.
\normalcolor
\end{proposition}
\begin{proof}
{\rm (1)} Let $\delta\in\BF(Z,X)$ be any element, with $X,Z\in\CH$. Let $x\in\CH(X,X')$ and $z\in\CH(Z',Z)$ be any pair of morphisms.
We remark that $\overline{x}\langle1\rangle\in\darkblueedit{(\CT/[\CW])}(X\langle1\rangle,X'\langle1\rangle)$ is given by $\overline{x}\langle1\rangle=\overline{x'}$, where  $x'\in\CT(X\langle1\rangle,X'\langle1\rangle)$ is a morphism such that $c_{X'}\circ x'=x[1]\circ c_X$ in $\CT$.
Then, we have a commutative diagram
\[
\begin{tikzpicture}[>=stealth]
\node (1) at (3.6,0.8) {$H(Z')$};
\node (2) at (1.4,0.8) {$H(Z)$};
\node (3) at (-1.4,-0.8) {$H(X\langle1\rangle)$};
\node (4) at (1.4,-0.8) {$H(X[1])$};
\node (5) at (-1.4,-2.4) {$H(X'\langle1\rangle)$};
\node (6) at (1.4,-2.4) {$H(X'[1])$};
\draw[->] (1) -- node[above,font=\scriptsize] {$H(z)$} (2);
\draw[->] (2) -- node[above,font=\scriptsize] {$\wp_{\delta}$} (3);
\draw[->] (2) -- node[left,font=\scriptsize] {$H(\delta)$} (4);
\draw[->] (3) -- node[below,font=\scriptsize] {$H(c_X)$} (4);
\draw[->] (3) -- node[left,font=\scriptsize] {$H(x')$} (5);
\draw[->] (4) -- node[left,font=\scriptsize] {$H(x[1])$} (6);
\draw[->] (5) -- node[below,font=\scriptsize] {$H(c_{X'})$} (6);
\draw[->] (1) to[bend left=10]
  node[right,font=\scriptsize] {$H(x[1]\circ\delta\circ z)=H(x_{\ast}z^{\ast}\delta)$}
  (6);
\end{tikzpicture}
\]
in $\CH/[\CW]$. Since $H(c_{X'})$ is monomorphic by Lemma~\ref{lem:H(f)_monom}, this induces
$\wp_{x_{\ast}z^{\ast}\delta}=H(x')\circ\wp_{\delta}\circ H(z)$.
This implies
\[
\partial^+_{Z',X'}(x_{\ast}z^{\ast}\delta)=(\Sigma\overline{x})\circ\partial^+_{Z,X}(\delta)\circ\overline{z},
\]
which shows the naturality of $\partial^+$.

{\rm (2)} By the definition of $(\BF,\ft)$, we may assume that the $\ft$-triangle appearing in the assumption is of the form
\[
X\xrightarrow{\overline{f}}Y\xrightarrow{\overline{g}}Z\overset{\delta}{\dashrightarrow},
\]
where $X\xrightarrow{f}Y\xrightarrow{g}Z\xrightarrow{\delta}X[1]$ is a distinguished triangle in $\CT$ with $\delta\in\BE_{\CN}(Z,X)$.

\color{blue}
Let $X\xrightarrow{a_X}W^X\xrightarrow{b_X}X\langle 1\rangle\xrightarrow{c_X}X[1]$ be the distinguished
triangle given in \cref{def:functor_<1>}, which satisfies $W^X\in\CW$.
By \cref{lem:char_for_Sn-deflation}, we have $\delta\in[\CU_{-n+1}^0]\cap[\CV_1^n]$. This implies $a_X\circ (\delta[-1])=0$, hence there exists $a\in\CT(Y,W^X)$ such that $a\circ f=a_X$.
\normalcolor
Applying Claim~\ref{claim:if_XYinH} to this $\fs$-triangle, we see that
\begin{equation}\label{rtri_from_claim}
X\xrightarrow{\overline{f}}Y\xrightarrow{\overline{l_Z\circ g}}Z_+\xrightarrow{\overline{v}_+}\Sigma X
\end{equation}
belongs to $\vartriangleright$, where
$\overline{v}_+\in(\CH/[\CW])(Z_+,X\langle1\rangle_+)$ is the unique morphism satisfying $H(c_X)\circ H(v)=H(\delta)$ in $\CH/[\CW]$. 
Note that we have $\wp_{\delta}=H(v)$.
Since $Z\in\CH$, it follows that $\overline{l_Z}\in(\CH/[\CW])(Z,Z_+)$ is an isomorphism. \darkblueedit{Moreover, since the right square and the outside of}
\color{blue}
\[
\begin{tikzpicture}[>=stealth]
\node (0) at (-2.2,0.8) {$Z$};
\node (1) at (-0.1,0.8) {$Z_+$};
\node (2) at (2.2,0.8) {$H(Z)$};
\node (3) at (-2.2,-0.8) {$\Sigma X$};
\node (4) at (-0.1,-0.8) {$H(X\langle1\rangle)$};
\draw[->] (0) -- node[above,font=\scriptsize] {$\overline{l_Z}$}node[below,font=\scriptsize] {$\cong$}  (1);
\draw[->] (2) -- node[above,font=\scriptsize] {$\overline{r_{Z_+}}$} node[below,font=\scriptsize] {$\cong$} (1);
\draw[->] (0) -- node[left,font=\scriptsize] {$\partial^+_{Z,X}(\delta)$} (3);
\draw[->] (1) -- node[left,font=\scriptsize] {$\overline{v}_+$} (3);
\draw[->] (2) -- node[right,font=\scriptsize] {$(\overline{v}_+)_-=H(v)=\wp_{\delta}$} (4);
\draw[->] (4) -- node[below,font=\scriptsize] {$\overline{r_{\Sigma X}}$} node[above,font=\scriptsize] {$\cong$} (3);
\end{tikzpicture}
\]
\normalcolor
are commutative in $\CH/[\CW]$, we also have $\partial^+_{Z,X}(\delta)=\overline{v}_+\circ\overline{l_Z}$. %
Thus $(\ref{sequence_to_belong_rtri})$ is isomorphic to $(\ref{rtri_from_claim})$, and hence also 
\color{blue}
belongs to 
\normalcolor
$\vartriangleright$.   
\end{proof}

The following gives a characterization of $j$-epimorphisms in $\CH/[\CW]$ for $1\le j\le n$.
\begin{proposition}\label{prop:equivalence_Sigma_epi}
Let $j$ be any integer with $1\le j\le n$. For any $f\in\CH(X,Y)$, the following are equivalent. Here, $M_f\in\CT^-$ is the object given in Definition~\ref{def:standard_right_triangle}. 
\begin{enumerate}
\item $\ovl{f}$ is a $j$-epimorphism.
\item $\Sigma^{j-1}((M_f)_+)\cong 0$ in $\CH/[\CW]$.
\item $M_f\langle j-1\rangle\in\CU$.
\item $M_f\in\CU_{-j+1}^0$.
\end{enumerate}

Moreover, when $j=n$, these conditions are also equivalent to the following.
\begin{enumerate}\setcounter{enumi}{4}
\item $g\in [\CU_{-n+1}^0]$ holds for a distinguished triangle $X\xrightarrow{f}Y\xrightarrow{g}Z\xrightarrow{h} X[1]$.
\end{enumerate}
\end{proposition}
\begin{proof}
$(1)\Leftrightarrow(2)$ This is nothing but the dual of \cref{lem:chara_n-mono_ker}.

$(2)\Leftrightarrow(3)$ 
By Proposition~\ref{prop:functor_Sigma}, we have $\Sigma^{j-1}((M_f)_+)\cong (M_f\langle j-1\rangle)_+$ in $\CH/[\CW]$.
By Proposition~\ref{prop:functor_+_property}, we have 
$(M_f\langle j-1\rangle)_+\cong 0$ in $\CH/[\CW]$ if and only if $M_f\langle j-1\rangle\in\CU$.

$(3)\Leftrightarrow(4)$  This is by Definition~\ref{def:functor_<1>}.

$(4)\Leftrightarrow(5)$ Assume $j=n$. Let $X\xrightarrow{f}Y\xrightarrow{g}Z\xrightarrow{h} X[1]$ be a distinguished triangle.
We continue to use the notation in Definition~\ref{def:standard_right_triangle}. By the octahedron axiom, we have a commutative diagram
\begin{equation}\label{octa_WM}
\begin{tikzcd}[row sep=0.9cm, column sep=0.7cm]
{}
&W^X\arrow{d}[swap]{\begin{bmatrix}0\\1\end{bmatrix}}\arrow[equal]{r}{}
&W^X \arrow{d}{}
\\
X \arrow{r}{\begin{bmatrix}f\\a_X\end{bmatrix}}\arrow[equal]{d}{}
&Y\oplus W^X \arrow{d}[swap]{{[}1\ 0{]}} \arrow{r}{{[}g_f\ w_f{]}}
&M_f \arrow{r}{k_f} \arrow{d}{z}
&X[1] \arrow[equal]{d}{}
\\
X\arrow{r}[swap]{f}
&Y \arrow{d}{0} \arrow{r}[swap]{g}
&Z\arrow{r}[swap]{h}\arrow{d}{}
&X[1]
\\
{}
&W^X[1] \arrow[equal]{r}
&W^X[1]
\end{tikzcd}
\end{equation}
in $\CT$, in which $W^X\to M_f\xrightarrow{z}Z\to W^X[1]$ is a distinguished triangle. Since $g=z\circ g_f$, we have $g\in[\CU_{-n+1}^0]$ whenever $M_f\in \CU_{-n+1}^0$.

Conversely, suppose that $g\in[\CU_{-n+1}^0]$ holds. It suffices to show $\CT(M_f,V[1])=0$ for any $V\in\CV$. 
Let $v\in\CT(M_f,V[1])$ be any morphism. 
In the above commutative diagram $(\ref{octa_WM})$, %
since $\CT(W^X,V[1])=0$, there exists $v'\in\CT(Z,V[1])$ such that $v=v'\circ z$.
By the assumption $g\in [\CU_{-n+1}^0]$, we have $v'\circ g=0$. Thus there exists $v^{\prime\prime}\in\CT(X[1],V[1])$ such that $v'=v^{\prime\prime}\circ h$.
Then we have
\[
v=v'\circ z=v^{\prime\prime}\circ h\circ z=v^{\prime\prime}\circ k_f=v^{\prime\prime}\circ c_X\circ h_f. 
\]
Since $X\langle1\rangle\in\CU_{-n+1}^0$, we have $\CT(X\langle1\rangle,V[1])=0$, hence $v^{\prime\prime}\circ c_X=0$. This shows $v=0$, as desired.
\end{proof}

\begin{corollary}\label{cor:implication_Sigma_epi}
Let $j$ be any integer with $1\le j\le n$. 
Let $X\xrightarrow{f}Y\xrightarrow{g}Z\xrightarrow{h} X[1]$ be a distinguished triangle in $\CT$ with $X,Y\in\CH$.
If $Z\in\CU_{-j+1}^0$, then $\overline{f}$ is a $j$-epimorphism.
\end{corollary}
\begin{proof}
We use the notation in Definition~\ref{def:standard_right_triangle}.
In the diagram \eqref{octa_WM}, since
\[
W^X\to M_f\xrightarrow{z}Z\to W^X[1]
\]
is a distinguished triangle, $Z\in\CU_{-j+1}^0$ implies $M_f\in\CW\ast\CU_{-j+1}^0=\CU_{-j+1}^0$. %
Thus \cref{prop:equivalence_Sigma_epi} shows that $\overline{f}$ is a $j$-epimorphism. 
\end{proof}

As a consequence of \cref{prop:equivalence_Sigma_epi}, we show that $\CH/[\CW]$ admits a suitable factorization in \cref{prop:factorization}.

\normalcolor

\begin{proposition}\label{prop:factorization}
Let $j$ be any integer with $1\le j\le n$.
For any morphism $\mathbf{f}\in(\CH/[\CW])(X,Y)$, there exist an \darkblueedit{$(n-j+1)$-monomorphism} $\mathbf{f}_2$ and a $j$-epimorphism $\mathbf{f}_1$ such that $\mathbf{f}=\mathbf{f}_2\circ \mathbf{f}_1$.
\end{proposition}
\begin{proof}
Take $f\in\CH(X,Y)$ that gives $\mathbf{f}=\overline{f}$.
Complete $f$ into a distinguished triangle $X\xrightarrow{f}Y\xrightarrow{g}Z\to X[1]$ in $\CT$.
Since $(\CU_{-j+1}^0,\CV_0^{n-j})$ is a cotorsion pair by \cref{prop:cotors_1-n}, there exists a distinguished triangle
$R\xrightarrow{b}L\xrightarrow{a}Z\to R[1]$
in $\CT$ with $L\in\CU_{-j+1}^0$ and $R\in\CV_0^{n-j}$.
By the octahedron axiom, we obtain a commutative diagram
\[
\begin{tikzcd}
{} & R\arrow{d}\arrow[equal]{r} & R\arrow{d}{b} 
\\
X\arrow{r}{f_1}\arrow[equal]{d} & Y'\arrow{d}[swap]{f_2}\arrow{r} & L\arrow{d}{a}
\\
X\arrow{r}[swap]{f} & Y\arrow{r}[swap]{g} & Z %
\end{tikzcd}
\]
in $\CT$, in which $X\xrightarrow{f_1}Y'\to L\to X[1]$ and $R\to Y'\xrightarrow{f_2}Y\to R[1]$ are distinguished triangles.

Since $L\in\CU_{-j+1}^0\subset\CU_{-n+1}^0$, it follows that 
$X\xrightarrow{f_1}Y'\to L\dashrightarrow$ is an $\fs^R_{\CN}$-triangle by \cref{lem:char_for_Sn-deflation}.
By the dual of \cref{lem:merged} {\rm (2)}, we obtain $Y'\in\CT^-$. 
Dually, since $R\in\CV_0^{n-j}\subset\CV_0^{n-1}$, it follows that 
$R\to Y'\xrightarrow{f_2}Y\dashrightarrow$ is an $\fs^L_{\CN}$-triangle, and hence $Y'\in\CT^+$. 
Thus we have $Y'\in\CH$. By \cref{cor:implication_Sigma_epi} and its dual, we see that $\overline{f_1}$ is a $j$-epimorphism and $\overline{f_2}$ is an $(n-j+1)$-monomorphism.
\end{proof}
\normalcolor

We next establish a characterization of %
$n$-epimorphisms in $\CH/[\CW]$ in \cref{prop:char_def}.

\begin{lemma}\label{lem:cocone_in_H2}
Let $X\xrightarrow{f}Y\xrightarrow{g}Z\xrightarrow{h}X[1]$ be a distinguished triangle in $\CT$ with $Y,Z\in\CH$. Let $\widetilde{V}\xrightarrow{v}W\xrightarrow{w}Z\xrightarrow{z} \widetilde{V}[1]$ be a distinguished triangle in $\CT$ with $W\in\CW$ and $\widetilde{V}\in\CV_0^{n-1}$, whose existence is guaranteed by the assumption $Z\in\CT^+$.
Then there exists an $\fs_\CN^L$-triangle \begin{equation}\label{seq:cocone_in_H2}
X'\to W\oplus Y\overset{\begin{bsmallmatrix}
w & g
\end{bsmallmatrix}}{\lra}Z\overset{h'}{\dra}
\end{equation}
such that %
$X'\in\CT^+$.
\darkblueedit{Moreover, if} $h\in[\CU_{-n+1}^0]$, then
\color{blue}
$X'\in\CH$, and 
\normalcolor
\eqref{seq:cocone_in_H2} is an $\fs_\CN$-triangle in $\CH$.
\end{lemma}
\begin{proof}
Complete $h\circ w$ into a distinguished triangle
$X\to X'\to W\xrightarrow{h\circ w}X[1]$.
By the dual of \cite[Proposition~1.4.6]{Nee01}, we obtain a commutative diagram
\begin{equation*}
\begin{tikzcd}
&\wtil{V}\arrow[equal]{r}\arrow{d}&\wtil{V}\arrow{d}{v} &{}
\\
X\arrow{r}\arrow[equal]{d}&X'\arrow{r}\arrow{d}&W\arrow{d}{w}\arrow{r}{h\circ w} &X[1]\arrow[equal]{d}
\\
X\arrow{r}[swap]{f}&Y\arrow{r}{g}\arrow{d}{}&Z\arrow{d}{z}%
\arrow{r}{h} &X[1]\arrow{d}{d}
\\
&\wtil{V}[1]\arrow[equal]{r}&\wtil{V}[1]
\arrow{r}[swap]{e} &X'[1]
\end{tikzcd}
\end{equation*}
in $\CT$, in which $\widetilde{V}\to X'\to Y\to\widetilde{V}[1]$ and %
\eqref{seq:cocone_in_H2} are distinguished triangles, where we put $h'=d\circ h=e\circ z$.
By \cref{lem:merged} {\rm (2)}, \darkblueedit{we have} $X'\in\CT^+$.
Since $\wtil{V}\in\CV_0^{n-1}$, the sequence \eqref{seq:cocone_in_H2} is indeed an $\fs_\CN^L$-triangle by \cref{lem:char_for_Sn-deflation}~{\rm (2)}.

Moreover, if $h\in[\CU_{-n+1}^0]$, then we have $h'\in[\CU_{-n+1}^0]$, and hence \eqref{seq:cocone_in_H2} is an $\fs_\CN^R$-triangle by \cref{lem:char_for_Sn-deflation}~{\rm (1)}.
\color{blue}
For the object $M_g$ obtained by applying \cref{def:standard_right_triangle} to $g\in\CH(Y,Z)$, we have $M_g\in\CU_{-n+1}^0$ by \cref{prop:equivalence_Sigma_epi}.
Similarly to \eqref{octa_WM}, we have a distinguished triangle
\[ W^Y\to M_g\to X[1]\to Y[1]\]
with $W^Y\in\CW$.
By the octahedron axiom, we have a commutative diagram
\[
\begin{tikzcd}
{} & W^Y\arrow{d}\arrow[equal]{r} & W^Y\arrow{d} 
\\
X'\arrow{r}{}\arrow[equal]{d} & W'\arrow{d}[swap]{}\arrow{r} & M_g\arrow{d}{}
\\
X'\arrow{r} & W\arrow{r}[swap]{h\circ w} & X[1] %
\end{tikzcd}
\]
for some $W'\in\CT$, 
in which 
$W^Y\to W'\to W\to W^Y[1]$ and $X'\to W'\to M_g\to X'[1]$ are distinguished triangles.
Since $W^Y,W\in\CW$, the middle column splits, and $W'\in\CW$.
By \cref{cor:TUT} applied to $X'\to W'\to M_g\to X'[1]$, we obtain $X'\in\CT^-$.
Since $X'\in\CT^+$, it follows $X'\in\CH$.
\end{proof}

\normalcolor

\begin{proposition}\label{prop:char_def}
Let $\mathbf{g}\in(\CH/[\CW])(Y,Z)$ be any morphism.
The following are equivalent.
\begin{enumerate}
\item $\mathbf{g}$ is a $\ft$-deflation.
\item There exists a right triangle $X\xrightarrow{\mathbf{f}}Y\xrightarrow{\mathbf{g}}Z\xrightarrow{\mathbf{h}}\Sigma X$ that belongs to $\vartriangleright$.
\item %
$\mathbf{g}$ is an $n$-epimorphism.
\normalcolor
\end{enumerate}
\end{proposition}
\begin{proof}
$(1)\Rightarrow(2)$ This is immediate from Proposition~\ref{prop:construction_of_dell}.

$(2)\Rightarrow(1)$ It suffices to show for the standard right triangle
\[
X\xrightarrow{\overline{f}}Y\xrightarrow{\overline{l_{M_f}}\circ \overline{g_f}}(M_f)_+\xrightarrow{\overline{(h_f)}_+}\Sigma X
\]
associated to $f\in\CH(X,Y)$, with the notation given in Definition~\ref{def:standard_right_triangle}.
Complete $l_{M_f}\circ g_f$ into a distinguished triangle
\[
Y\xrightarrow{l_{M_f}\circ g_f}(M_f)_+\to P\to Y[1]
\]
in $\CT$, for some $P\in\CT$.
By the octahedron axiom, we obtain a commutative diagram
\[
\begin{tikzpicture}[>=stealth]
\node (1) at (-2.4,0) {$Y$};
\node (2) at (-0.42,0.88) {$M_f$};
\node (3) at (1.4,1.7) {$X\langle1\rangle$};
\node (4) at (0,0) {$(M_f)_+$};
\node (5) at (1.2,0) {$P$};
\node (6) at (0.98,-1.8) {$U'_{M_f}$};
\draw[->] (1) -- node[above,font=\scriptsize] {$g_f$} (2);
\draw[->] (1) -- node[below,font=\scriptsize] {$l_{M_f}\circ g_f$} (4);
\draw[->] (2) -- node[above,font=\scriptsize] {$h_f$} (3);
\draw[->] (2) -- node[right,font=\scriptsize] {$l_{M_f}$} (4);
\draw[->] (3) -- (5);
\draw[->] (4) -- (5);
\draw[->] (4) -- (6);
\draw[->] (5) -- (6);
\end{tikzpicture}
\]
in $\CT$, in which $X\langle1\rangle\to P\to U'_{M_f}\to (X\langle1\rangle)[1]$ is a distinguished triangle. Since $X\langle1\rangle\in\CU_{-n+1}^0$ and $U'_{M_f}\in\CU$, we have $P\in\CU_{-n+1}^0$.
If we put $g=l_{M_f}\circ g_f$ and apply \cref{lem:cocone_in_H2} to the distinguished triangle $P[-1]\to Y\xrightarrow{g}(M_f)_+\to P$, then we obtain an $\fs_{\CN}$-triangle
\[ X'\to W\oplus Y\xrightarrow{[w\ g]}(M_f)_+\dashrightarrow \]
for some $X'\in\CH$ and $W\in\CW$. Thus $\overline{g}$ is a $\ft$-deflation in $\CH/[\CW]$.

$(1)\Rightarrow(3)$ This follows from \cref{prop:construction_of_dell} and \cref{lem:defl_are_n-epi}. 

$(3)\Rightarrow(1)$ This follows from \cref{prop:equivalence_Sigma_epi} and \cref{lem:cocone_in_H2}.
\normalcolor
\end{proof}

\color{blue}

\begin{lemma}
\label{lem:partial_description}
Assume that $X,Z\in\CH$. 
Let $X\xrightarrow{a_X} W^X\xrightarrow{b_X} X\langle 1\rangle\xrightarrow{c_X} X[1]$ be the distinguished triangle given in \cref{def:functor_<1>}, and let $\delta\in \BF(Z,X)$. %
Then there exists a morphism $p\colon Z\to X\langle 1\rangle$ such that $c_X\circ p=\delta$ in $\CT$ and $\ovl{l_{X\langle 1\rangle}}\circ\ovl{p}=\partial^+(\delta)$ in $\CT/[\CW]$.
$$
\begin{tikzcd}
X\langle 1\rangle & & X[1] \\
&Z&
\Ar{2-2}{1-1}{"p"}
\Ar{2-2}{1-3}{"\delta"'}
\Ar{1-1}{1-3}{"c_X"}
\end{tikzcd}\quad\quad
\begin{tikzcd}
Z && \Sigma X \\
& X\langle 1\rangle
\Ar{1-1}{1-3}{"\partial^+(\delta)"}
\Ar{1-1}{2-2}{"\ovl{p}"'}
\Ar{2-2}{1-3}{"\ovl{l_{X\langle 1\rangle}}"'}
\end{tikzcd}
$$
\end{lemma}
\begin{proof}
Since $\delta\in \BF(Z,X)$, it factors through an object in $\CU_{-n+1}^0$ by \cref{lem:char_for_Sn-deflation}. Thus, the composite morphism $Z\xrightarrow{\delta}X[1]\xrightarrow{a_X[1]} W^X[1]$ is zero. Hence, there exists a morphism $p\colon Z\to X\langle 1\rangle$ such that $c_X\circ p=\delta$ in $\CT$. We next show that $\ovl{l_{X\langle 1\rangle}}\circ\ovl{p}=\partial^+(\delta)$ in $\CT/[\CW]$. By \cref{def:construction_of_dell}, it suffices to show that the following diagram commutes in $\CT/[\CW]$, since $H(c_X)$ is a monomorphism:
$$
\begin{tikzcd}
Z & Z_+ & H(Z) \\
\Sigma X & H(X\langle 1\rangle) & H(X[1])
\Ar{1-1}{1-2}{"\ovl{l_Z}"}
\Ar{1-3}{1-2}{"\ovl{r_{Z_+}}"',"\cong"}
\Ar{1-3}{2-3}{"H(\delta)"}
\Ar{2-2}{2-1}{"\ovl{r_{\Sigma X}}","\cong"'}
\Ar{2-2}{2-3}{"H(c_X)"'}
\Ar{1-1}{2-1}{"\ovl{l_{X\langle1\rangle}}\circ \ovl{p}"'}
\end{tikzcd}
$$
It follows from the following equations in $\CT/[\CW]$:
\begin{align*}
H(c_X)\circ \ovl{r_{\Sigma X}}^{-1}\circ \ovl{l_{X\langle 1\rangle}}\circ \ovl{p} 
&=H(c_X)\circ H(p) \circ \ovl{r_{Z_+}}^{-1}\circ \ovl{l_{Z}}\\
&=H(c_X\circ p) \circ \ovl{r_{Z_+}}^{-1}\circ \ovl{l_{Z}}\\ 
&=H(\delta) \circ \ovl{r_{Z_+}}^{-1}\circ \ovl{l_{Z}}\\ 
\end{align*}
\darkblueedit{The first equality follows from the definition of $H(p)$; the second equality follows from $c_X\circ p=\delta$ in $\CT$.} Note that $\ovl{r_{Z_+}}$ and $\ovl{r_{\Sigma X}}$ are isomorphisms in $\CT/[\CW]$ since $Z_+, \Sigma X\in \CT^-$ and the dual of \cref{prop:functor_+}.
\end{proof}

The following lemma is a dual version of \cref{lem:partial_description}.

\begin{lemma}
\label{lem:partial_description_dual}
Assume that $X,Z\in\CH$. 
Let $Z[-1]\xrightarrow{{}_Zc}Z\langle -1\rangle\xrightarrow{{}_Zb} {}^ZW\xrightarrow{{}_Za} Z$ be the distinguished triangle given in \cref{def:functor_<-1>}, and let $\delta\in \BF(Z,X)$.
Then there exists a morphism $q\colon Z\langle -1\rangle\to X$ such that $q\circ {}_Zc=\delta[-1]$ in $\CT$ and $\ovl{q}\circ\ovl{r_{Z\langle -1\rangle}}=\partial^-(\delta)$ in $\CT/[\CW]$.
\end{lemma}

\begin{lemma}
\label{lem:c_X_monomorphism}
$\ovl{c_X}$ is a monomorphism in $\CT/[\CW]$ for any $X\in\CH$.
\end{lemma}
\begin{proof}
Consider the distinguished triangle $X\xrightarrow{a_X} W^X\xrightarrow{b_X} X\langle 1\rangle\xrightarrow{c_X} X[1]$ in $\CT$ given in \cref{def:functor_<1>}. Let $f\colon M\to X\langle 1\rangle$ be a morphism in $\CT$ such that $\ovl{c_X}\circ\ovl{f}=0$ in $\CT/[\CW]$. Then $c_X\circ f=v\circ u$ for some $u\colon M\to W$ and $v\colon W\to X[1]$, where $W\in\CW$. Since $W\in\CW$, we have $a_X[1]\circ v=0$. Hence there exists a morphism $f'\colon W\to X\langle 1\rangle$ such that $c_X\circ f'=v$. It follows that $c_X\circ(f-f'\circ u)=0$, so $f-f'\circ u$ factors through $b_X$. Thus, $\ovl{f}=\ovl{f'\circ u}=0$ in $\CT/[\CW]$. Therefore, $\ovl{c_X}$ is a monomorphism in $\CT/[\CW]$.
\end{proof}

In what follows, as in \cref{prop:adjoint_Sigma_Omega}, $\eta$ and
$\varepsilon$ denote the unit and counit of the adjunction
$\Sigma\dashv\Omega$, respectively.
\begin{lemma}
\label{lem:equation_counit}
The following diagram is commutative in $\CT/[\CW]$ for any $Z\in \CH$:
$$
\begin{tikzcd}
(\Omega Z)\langle 1\rangle & (\Omega Z)[1] & (Z\langle -1\rangle)[1] \\
\Sigma\Omega Z & Z & (Z\langle -1\rangle)[1]
\Ar{1-1}{1-2}{"\ovl{c_{(\Omega Z)}}"}
\Ar{1-2}{1-3}{"\ovl{r_{Z\langle -1\rangle}[1]}"}
\Ar{2-1}{2-2}{"\varepsilon_Z"}
\Ar{2-2}{2-3}{"{-\ovl{{}_Zc[1]}}"}
\Ar{1-1}{2-1}{"\ovl{l_{(\Omega Z)\langle 1\rangle}}"'}
\Ar{1-3}{2-3}{equal}
\end{tikzcd}
$$
Here, morphisms $c_{\Omega Z}$ and ${}_Zc$ are those given in \cref{def:functor_<1>,def:functor_<-1>}.
\end{lemma}
\begin{proof}
This follows immediately by tracing the morphism corresponding to $\id_{\Omega Z}$ under the isomorphisms in the proof of \cref{prop:adjoint_Sigma_Omega}.
\end{proof}

\begin{lemma}
\label{lem:the_equation}
For every $\delta\in\BF(Z,X)$, one has
$
\Sigma\bigl(\partial^-(\delta)\bigr)
=
-\partial^+(\delta)\circ\varepsilon_Z
$
in $\CH/[\CW]$.
\end{lemma}
\begin{proof}
By \cref{prop:functor_+}(2), it is enough to prove the following equality
in $\CT/[\CW]$:
\[
\Sigma\bigl(\partial^-(\delta)\bigr)
\circ\ovl{l_{(\Omega Z)\langle1\rangle}}
=
-\partial^+(\delta)\circ\varepsilon_Z
\circ\ovl{l_{(\Omega Z)\langle1\rangle}}.
\]

In the following argument, the morphisms $p$ and $q$ are obtained from
$\delta$ by applying
\cref{lem:partial_description,lem:partial_description_dual}.
We first compute in $\CT/[\CW]$:
\begin{align*}
\ovl{c_X}
\circ
\bigl(
  \ovl q\circ\ovl{r_{Z\langle-1\rangle}}
\bigr)\langle1\rangle
&=
\ovl{q[1]}
\circ\ovl{r_{Z\langle-1\rangle}[1]}
\circ\ovl{c_{\Omega Z}}
&&
\text{by \cref{rem:c_is_natural}}
\\
&=
-\ovl{q[1]}
\circ\ovl{{}_Zc[1]}
\circ\varepsilon_Z
\circ\ovl{l_{(\Omega Z)\langle1\rangle}}
&&
\text{by \cref{lem:equation_counit}}
\\
&=
-\ovl\delta
\circ\varepsilon_Z
\circ\ovl{l_{(\Omega Z)\langle1\rangle}}
&&
\text{by \cref{lem:partial_description_dual}}
\\
&=
-\ovl{c_X}
\circ\ovl p
\circ\varepsilon_Z
\circ\ovl{l_{(\Omega Z)\langle1\rangle}}
&&
\text{by \cref{lem:partial_description}}.
\end{align*}

By \cref{lem:c_X_monomorphism}, the morphism $\ovl{c_X}$ is a
monomorphism in $\CT/[\CW]$. Hence
\[
\bigl(
  \ovl q\circ\ovl{r_{Z\langle-1\rangle}}
\bigr)\langle1\rangle
=
-\ovl p\circ\varepsilon_Z
\circ\ovl{l_{(\Omega Z)\langle1\rangle}}.
\]

We now compute in $\CT/[\CW]$:
\begin{align*}
\Sigma\bigl(\partial^-(\delta)\bigr)
\circ\ovl{l_{(\Omega Z)\langle1\rangle}}
&=
\Sigma\bigl(
  \ovl q\circ\ovl{r_{Z\langle-1\rangle}}
\bigr)
\circ\ovl{l_{(\Omega Z)\langle1\rangle}}
&&
\text{by \cref{lem:partial_description_dual}}
\\
&=
\Bigl(
  \bigl(
    \ovl q\circ\ovl{r_{Z\langle-1\rangle}}
  \bigr)\langle1\rangle
\Bigr)_+
\circ\ovl{l_{(\Omega Z)\langle1\rangle}}
&&
\text{by the definition of $\Sigma$}
\\
&=
\ovl{l_{X\langle1\rangle}}
\circ
\bigl(
  \ovl q\circ\ovl{r_{Z\langle-1\rangle}}
\bigr)\langle1\rangle
&&
\text{by naturality of the reflections}
\\
&=
-\ovl{l_{X\langle1\rangle}}
\circ\ovl p
\circ\varepsilon_Z
\circ\ovl{l_{(\Omega Z)\langle1\rangle}}
&&
\text{by the equality above}
\\
&=
-\partial^+(\delta)
\circ\varepsilon_Z
\circ\ovl{l_{(\Omega Z)\langle1\rangle}}
&&
\text{by \cref{lem:partial_description}}.
\end{align*}
This proves the desired equality.
\end{proof}

\normalcolor

By the argument so far, we obtain the following.
\begin{theorem}\label{thm:quasitri}
Let $(\CU,\CV)$ be any $n$-cotorsion pair. 
\color{blue}
The heart $\CH/[\CW]$ equipped with the pretriangulated structure $(\Sigma,\Omega,\vartriangleright,\vartriangleleft)$ and the extriangulated structure $(\BF,\ft)$ obtained so far is an abelian $n$-truncated category.
\normalcolor 
\end{theorem}%
\begin{proof}
$(\CH/[\CW],\Sigma,\Omega,\vartriangleright,\vartriangleleft)$ is an $n$-truncated pretriangulated category by \cref{prop:pretriangulated}. 
By \cref{prop:construction_of_dell} and {\color{blue}\cref{lem:the_equation}}, the extriangulated structure $(\BF,\ft)$ is compatible with the pretriangulated structure. 
By \cref{prop:char_def} and its dual, we see that $(\CH/[\CW],\Sigma,\Omega,\vartriangleright,\vartriangleleft)$ together with $(\BF,\ft)$ is an 
\color{blue}
abelian $n$-truncated 
\normalcolor
category.
\end{proof}
\normalcolor

\color{blue}
\begin{corollary}\label{cor:homotopy_category_abelian_truncated}
\begin{enumerate}
\item If $\mathcal{A}$ is a small abelian $(n,1)$-category in the sense of
\cite[Definition~6.2.4]{Ste23}, then its homotopy category $h\mathcal{A}$ is equivalent to an abelian $n$-truncated category.
\item If $\A$ is a small abelian $n$-truncated DG-category in the sense of
\cite[Definition~3.12]{Moc25}, then its homotopy category $H^0(\A)$ is equivalent to an abelian $n$-truncated category.
\end{enumerate}
\end{corollary}
\begin{proof}
By \cite[Theorem~6.3.2(1) and Remark~6.3.3]{Ste23}, there is a stable
$\infty$-category $\Db_\infty(\mathcal{A})$ endowed with a bounded
$t$-structure such that $\mathcal{A}$ is equivalent to its $n$-extended heart.
Thus, after shifting the $t$-structure if necessary to match our indexing
convention, there is an equivalence
\[
h\mathcal{A}\simeq
\bigl(h\Db_\infty(\mathcal{A})\bigr)^{\geq 0}
\cap
\bigl(h\Db_\infty(\mathcal{A})\bigr)^{\leq n-1}.
\]
By \cref{ex:t_structure_n_cotorsion_pair}, the category on the right-hand side
is the heart of the $n$-cotorsion pair
$\bigl(t^{\leq -1},t^{\geq n}\bigr)$.
Hence it is an abelian $n$-truncated category by \cref{thm:quasitri}.

For~{\rm (2)}, \cite[Theorem~5.4]{Plo26} gives a $t$-structure on the bounded
derived DG-category $\Db_{\dg}(\A)$ and a quasi-equivalence
\[
\A\simeq \Db_{\dg}(\A)^{(-n,0]}.
\]
Passing to homotopy categories and, if necessary, shifting the induced
$t$-structure to match our indexing convention, we obtain an equivalence
\[
H^0(\A)\simeq
\bigl(H^0(\Db_{\dg}(\A))\bigr)^{\geq 0}
\cap
\bigl(H^0(\Db_{\dg}(\A))\bigr)^{\leq n-1}.
\]
The same argument using \cref{ex:t_structure_n_cotorsion_pair,thm:quasitri}
shows that the category on the right-hand side is abelian $n$-truncated.
\end{proof}
\normalcolor

\subsubsection{A sufficient condition for the heart to have enough projectives}

By \cref{prop:cotors_1-n}, in particular $\CU_{-n}^{-1}\subset\CT$ is extension-closed.
In the following proposition, we regard $\CU_{-n}^{-1}$ as an extriangulated category endowed with the extriangulated structure induced by that on $\CT$. 
\begin{proposition}\label{prop:enough-proj}
Let $\CP\subset \CU_{-n}^{-1}$ denote the full subcategory of all projective objects in the extriangulated category $\CU_{-n}^{-1}$. The following holds.
\begin{enumerate}
\item For any $P\in\CP$, the object $H(P)\in\CH/[\CW]$ is projective in $\CH/[\CW]$.
\item If $\CU_{-n}^{-1}$ has enough projective objects, then the heart $\CH/[\CW]$ also has enough projective objects. Moreover, the full subcategory of its projective objects agrees with $\add H(\CP)$, where $H(\CP)\subset \CH/[\CW]$ denotes the essential image of $\CP$ by the functor $H$.
\end{enumerate}
\end{proposition}
\begin{proof}
{\rm (1)} Let $P\in\CP$ be any object. Since $\CP\subset\CU_{-n}^{-1}\subset\CT^-$, we have $H(P)\cong P_+$ in $\CH/[\CW]$.
Thus, it suffices to show %
that 
$\BE_{\CN}(P_+,X)=0$
holds for all $X\in\CH$. 
Let $h\in\BE_{\CN}(P_+,X)$ be any element. 
By the definition of $(\ )_+$, there is a distinguished triangle
\[
U'_P[-1]\xrightarrow{u^{\prime\prime}}P\xrightarrow{l_P}P_+\to U'_P
\]
with $U'_P\in\CU$ and $u^{\prime\prime}\in[\CU]$.
By $u^{\prime\prime}\in[\CU]$, %
it induces an $\fs_{\CN}^L$-triangle
$P\xrightarrow{l_P}P_+\to U'_P\dashrightarrow$.
This gives rise to an exact sequence
\begin{equation}\label{exact_p}
0\to\BE_{\CN}^L(P_+,X)\xrightarrow{(l_P)^{\ast}}\BE_{\CN}^L(P,X),
\end{equation}
since $\BE_{\CN}^L(U'_P,X)=0$ by \cref{lem:merged} {\rm (1)}. 
Also, since $X\in\CH\subset \CT^-$, there is a distinguished triangle 
$\widetilde U\xrightarrow{x}X\xrightarrow{w}W\to \widetilde{U}[1]$
with $\widetilde{U}\in\CU_{-n}^{-1}$ and $W\in\CW$.
By the dual of \cref{lem:merged} {\rm (1)}, we have $\BE_{\CN}(P_+,W)\subset\BE_{\CN}^R(P_+,W)=0$.
Since $w[1]\circ h=w_{\ast}h\in\BE_{\CN}(P_+,W)$, this implies $w[1]\circ h=0$. Thus there exists $h'\in\CT(P_+,\widetilde{U}[1])$ such that $h=x[1]\circ h'$.  
Since $P$ is projective in $\CU_{-n}^{-1}$, we have $\CT(P,\widetilde{U}[1])=\BE(P,\widetilde{U})=0$. This forces $h'\circ l_P=0$, which implies $(l_P)^{\ast}h=0$ in $\BE_{\CN}(P,X)\subset\BE_{\CN}^L(P,X)$. By the exactness of \eqref{exact_p}, we obtain $h=0$ as desired. 

{\rm (2)} Let $Z\in\CH$ be any object. It is enough to show that there exists a deflation $Q\to Z$ in $\CH/[\CW]$ from some object $Q\in H(\CP)$. Since $Z\in\CH\subset\CT^-$, there is a distinguished triangle 
$\widetilde U\xrightarrow{z}Z\to W\to \widetilde{U}[1]$
with $\widetilde{U}\in\CU_{-n}^{-1}$ and $W\in\CW$.
Since $\CU_{-n}^{-1}$ has enough projective objects by the assumption, there exists a distinguished triangle
\[
\widetilde{U}'\xrightarrow{p'} P\xrightarrow{p} \widetilde{U}\to \widetilde{U}'[1]
\]
\darkblueedit{in $\CT$ for some} $\widetilde{U}'\in\CU_{-n}^{-1}$ and $P\in\CP$. Apply $(\ )_+$ to obtain $P_+\in\CH$.
By definition, there is a distinguished triangle
$
U'_P[-1]\xrightarrow{u^{\prime\prime}}P\xrightarrow{l_P}P_+\to U'_P
$
in $\CT$ with $U'_P\in\CU$ and $u^{\prime\prime}\in[\CU]$.
By the octahedron axiom, we obtain a commutative diagram
\[
\begin{tikzpicture}[>=stealth]
\node (1) at (-0.98,1.8) {$\widetilde{U}'$};
\node (2) at (-1.19,0) {$Y$};
\node (3) at (-1.4,-1.7) {$W[-1]$};
\node (4) at (0,0) {$P$};
\node (5) at (0.45,-0.9) {$\widetilde{U}$};
\node (6) at (2.4,0) {$Z$};
\draw[->] (1) -- node[left,font=\scriptsize] {} (2);
\draw[->] (1) -- node[right,font=\scriptsize] {$p'$} (4);
\draw[->] (2) -- node[left,font=\scriptsize] {} (3);
\draw[->] (2) -- node[below,font=\scriptsize] {} (4);
\draw[->] (3) -- (5);
\draw[->] (4) -- node[left,font=\scriptsize] {$p$} (5);
\draw[->] (4) -- node[above,font=\scriptsize] {$z\circ p$} (6);
\draw[->] (5) -- node[below,font=\scriptsize] {$z$} (6);
\end{tikzpicture}
\]
for some $Y$ in $\CT$, in which $Y\to P\xrightarrow{z\circ p}Z\to Y[1]$ and $\widetilde{U}'\to Y\to W[-1]\to \widetilde{U}'[1]$ are distinguished triangles. 
In particular we have $Y\in\CU_{-n}^{-1}\ast\CW[-1]=\CU_{-n}^{-1}$.
 
Since $Z\in\CH$, there exists $g\in\CT(P_+,Z)$ such that $z\circ p=g\circ l_P$ by \cref{prop:functor_+} {\rm (1)}.
Complete $g$ into a distinguished triangle $X\to P_+\xrightarrow{g}Z\xrightarrow{h} X[1]$. Then, we obtain a commutative diagram
\[
\begin{tikzpicture}[>=stealth]
\node (11) at (6.02,1.8) {$U'_P[-1]$};
\node (12) at (5.81,0) {$Y$};
\node (13) at (5.6,-1.7) {$X$};
\node (14) at (7,0) {$P$};
\node (15) at (7.45,-0.9) {$P_+$};
\node (16) at (9.4,0) {$Z$};
\draw[->] (11) -- node[left,font=\scriptsize] {} (12);
\draw[->] (11) -- node[right,font=\scriptsize] {$u^{\prime\prime}$} (14);
\draw[->] (12) -- node[left,font=\scriptsize] {} (13);
\draw[->] (12) -- node[below,font=\scriptsize] {} (14);
\draw[->] (13) -- (15);
\draw[->] (14) -- node[left,font=\scriptsize] {$l_P$} (15);
\draw[->] (14) -- node[above,font=\scriptsize] {$z\circ p$} (16);
\draw[->] (15) -- node[below,font=\scriptsize] {$g$} (16);
\end{tikzpicture}
\]
in $\CT$, in which $U_P'[-1]\to Y\to X\to U_P'$ is a distinguished triangle.
Since $h$ factors through $Y[1]\in\CU_{-n+1}^0$, 
we obtain an $\fs_{\CN}$-triangle
$X' \to W'\oplus P_+\to Z\dashrightarrow$
in $\CH$ for some $X'\in\CH$ and $W'\in\CW$, by \cref{lem:cocone_in_H2}. Thus $Z$ admits a deflation from $Q=W'\oplus P_+\in H(\CP)$ in $\CH/[\CW]$.
\end{proof}

\begin{corollary}
Let $\CC\subset\CT$ be an $(n+1)$-cluster tilting subcategory. Then the extriangulated category $\CT/[\CC]$ has enough projectives, and the full subcategory of its projective objects agrees with \darkblueedit{$\add(\CC[-n])\subset\CT/[\CC]$}.
\end{corollary}
\begin{proof}
This immediately follows from \cref{prop:enough-proj} applied to the $n$-cotorsion pair $(\CC,\CC)$, since $\CC_{-n}^{-1}$ has enough projective objects, whose full subcategory of all projective objects is $\CC[-n]$, and since the functor $H$ is naturally isomorphic to the functor taking the ideal quotient $\CT\to\CT/[\CC]$.
\end{proof}

\normalcolor

\subsection{The heart via extriangulated quotients}\label{subsec:via_extri_quotient}

In this subsection, we assume that $\CT$ is skeletally small.
Let $\CT_{\CN}=(\CT,\BE_\CN,\fs_\CN)$ and $\SS_{\CN}$ be as in \cref{def:substructure_of_T,def:LRS}.
As a corollary of \cref{prop:fully_faithful}, we have the following.
\begin{corollary}\label{cor:equiv_by_H}
Let
\color{blue}
$Q\colon\CT\to\widetilde{\CT}$ be the localization functor with respect to $\SS_{\CN}$. %
Then, there exists a unique functor $\widetilde{H}\colon\widetilde{\CT}\to\CH/[\CW]$ such that $\widetilde{H}\circ Q=H$. Moreover, such $\widetilde{H}$ is an equivalence of categories.
\normalcolor
\end{corollary}
\begin{proof}
$\widetilde{H}$ is induced by the universality of the localization, and it is fully faithful by Proposition~\ref{prop:fully_faithful}.
Since $H$ is essentially surjective by Remark~\ref{rem:essentially_surjective_H}, so is
$\widetilde{H}$.
\end{proof}

We recall that $\CH/[\CW]$ is endowed with an extriangulated structure $(\BF,\ft)$ in \cref{cor:proj-inj_in_H}.
\color{blue}
In \cref{thm:extri_heart}, we show that $(\widetilde{\CT},\widetilde{\BE_{\CN}},\widetilde{\fs_{\CN}})$ and $(\CH/[\CW],\BF,\ft)$ \darkblueedit{are equivalent} as extriangulated categories.
\color{blue}
The following is from \cite{Oga24}.
\normalcolor
\begin{fact}\label{fact:Thm2.20}(\cite[Proposition~2.17, Theorem~2.20]{Oga24})
\color{blue}
Let $Q\colon\CT\to\widetilde{\CT}$ be the localization functor with respect to $\SS_{\CN}$.
Then, $\SS_{\CN}$ is saturated in the sense that $Q(f)$ is an isomorphism in $\widetilde{\CT}$ if and only if $f\in\SS_{\CN}$, for any morphism $f$ in $\CT$. 
Moreover, the image of $\SS_{\CN}$ in $\CT/[\CN]$
\normalcolor
satisfies conditions {\rm (MR1)},\,\ldots,\,{\rm (MR4)} introduced in \cite{NOS22}. As a consequence, 
\color{blue}
by \cite[Theorem~3.5]{NOS22}, the category $\widetilde{\CT}$ has a structure of an extriangulated category $(\widetilde{\CT},\widetilde{\BE_{\CN}},\widetilde{\fs_{\CN}})$. Moreover, the functor $Q$ gives an exact functor $(Q,\mu)\colon(\CT,\BE_{\CN},\fs_{\CN})\to(\widetilde{\CT},\widetilde{\BE_{\CN}},\widetilde{\fs_{\CN}})$, which satisfies the universality for exact functors $(F,\phi)\colon(\CT,\BE_{\CN},\fs_{\CN})\to(\CD,\BE',\fs')$ with $\CN\subset\Ker F$.
\color{blue}
\normalcolor
\end{fact}

\color{blue}
\begin{remark}
In the above, the notion of an \emph{exact functor} between extriangulated categories is nothing but that of an \emph{extriangulated functor} introduced in \cite[Definition~2.23]{B-TS21}. See \cite[Definition~2.11 and Remark~2.12]{NOS22} for details.
\end{remark}

\color{blue}
We note 
\normalcolor
that the functor taking the ideal quotient $\textcolor{blue}{\pi}\colon \CH\to\CH/[\CW]$ and the inclusion of the extension-closed subcategory $i\colon\CH\to\CT$ can be endowed with a natural structure of exact functors
\[
\textcolor{blue}{(\pi,\nu)}\colon(\CH,\BE_{\CN}|_{\CH},\fs_{\CN}|_{\CH})\to(\CH/[\CW],\BF,\ft)
\quad  \text{and} \quad
(i,\iota)\colon (\CH,\BE_{\CN}|_{\CH},\fs_{\CN}|_{\CH})\to(\CT,\BE_{\CN},\fs_{\CN}).
\]
As stated in \cite[Remark~3.35]{NOS22}, the exact functor $\textcolor{blue}{(\pi,\nu)}$ can be also regarded as a quotient of extriangulated functors. By its universality, we obtain an exact functor $(G,\psi)\colon (\CH/[\CW],\BF,\ft)\to(\widetilde{\CT},\widetilde{\BE_{\CN}},\widetilde{\fs_{\CN}})$ that makes the diagram
\[
\begin{tikzpicture}[>=stealth]
\node (1) at (-1.6,0.8) {$(\CH,\BE_{\CN}|_{\CH},\fs_{\CN}|_{\CH})$};
\node (2) at (1.6,0.8) {$(\CT,\BE_{\CN},\fs_{\CN})$};
\node (3) at (-1.6,-0.8) {$(\CH/[\CW],\BF,\ft)$};
\node (4) at (1.6,-0.8) {$(\widetilde{\CT},\widetilde{\BE_{\CN}},\widetilde{\fs_{\CN}})$};
\draw[->] (1) -- node[above,font=\scriptsize] {$(i,\iota)$} (2);
\draw[->] (1) -- node[left,font=\scriptsize] {$\textcolor{blue}{(\pi,\nu)}$} (3);
\draw[->] (2) -- node[right,font=\scriptsize] {$(Q,\mu)$} (4);
\draw[->] (3) -- node[below,font=\scriptsize] {$(G,\psi)$} (4);
\end{tikzpicture}
\]
commutative. 
Explicitly, the functor $G\colon\CH/[\CW]\to\widetilde{\CT}$ sends each object $X$ in $\CH/[\CW]$ to $X$ in $\widetilde{\CT}$, and each morphism $\mathbf{f}=\overline{f}$ in $\CH/[\CW]$ to $Q(f)=Q\circ i(f)$ in $\widetilde{\CT}$. The natural transformation $\psi\colon\BF\Rightarrow\widetilde{\BE_{\CN}}$ sends each $\delta\in\BF(Z,X)$ to 
\[
[Z\xleftarrow{\id}Z\overset{\textcolor{blue}{\widehat{\delta}}}{\dashrightarrow}X\xleftarrow{\id}X] \in \widetilde{\BE_{\CN}}(Z,X),
\]
for all $X,Z\in\CH/[\CW]$.

\color{blue}
To explain the notation above, we briefly recall the construction of the functor
$\widetilde{\BE_{\CN}}\colon\widetilde{\CT}^{\op}\times\widetilde{\CT}\to\Ab$
following the construction in \cite{NOS22}.
First, define an additive subfunctor
$\CK$ of $\BE_{\CN}\colon\CT^{\op}\times\CT\to\Ab$ by
\begin{eqnarray*}
\CK(C,A)
&=&
\{\delta\in\BE_{\CN}(C,A)\mid s_{\ast}\delta=0 \ \text{for some}\ s\in\SS_{\CN}(A,A')\}\\
&=&
\{\delta\in\BE_{\CN}(C,A)\mid t^{\ast}\delta=0 \ \text{for some}\ t\in\SS_{\CN}(C',C)\}
\end{eqnarray*}
for all $A,C\in\Ob(\CT)$ (\cite[Definition~3.14 and Proposition~3.15]{NOS22}),
where we put
\[
\SS_{\CN}(X,Y)
=
\{s\in\CT(X,Y)\mid s\in\SS_{\CN}\}
\]
for each $X,Y\in\CT$. Then we obtain a biadditive functor
\[ \overline{\BE_{\CN}}\colon(\CT/[\CN])^{\op}\times(\CT/[\CN])\to\Ab \]
defined by the following {\rm (i)}, {\rm (ii)} and {\rm (iii)}. See \cite[Definition~3.16]{NOS22} for details. To avoid confusion with the image $\overline{f}$ in $\CH/[\CW]$, in this article we shall henceforth write $\widehat{f}$ for the image of $f\in\CT(X,Y)$ in $(\CT/[\CN])(X,Y)$, and $\widehat{\delta}$ for the image of $\delta\in\BE_{\CN}(Z,X)$ in $\BE_{\CN}(Z,X)/\CK(Z,X)$.
\begin{enumerate}
\renewcommand{\labelenumi}{(\roman{enumi})}
\item $\overline{\BE_{\CN}}(C,A)=\BE_{\CN}(C,A)/\CK(C,A)$ for all $A,C\in\Ob(\CT/[\CN])=\Ob(\CT)$.
\item $\widehat{a}_{\ast}\widehat{\delta}=\widehat{a_{\ast}\delta}$ for any $\widehat{\delta}\in\overline{\BE_{\CN}}(C,A)$ and $\widehat{a}\in(\CT/[\CN])(A,A')$.
\item $\widehat{c}^{\ast}\widehat{\delta}=\widehat{c^{\ast}\delta}$ for any $\widehat{\delta}\in\overline{\BE_{\CN}}(C,A)$ and $\widehat{c}\in(\CT/[\CN])(C',C)$.
\end{enumerate}

\begin{remark}
Since $\SS_{\CN}$ is saturated, any isomorphism in $\CT/[\CN]$ belongs to the image of $\SS_{\CN}$ in $\CT/[\CN]$. In particular, $\overline{\SS_{\CN}}$ in \cite{NOS22} agrees with the image of $\SS_{\CN}$ in $\CT/[\CN]$. In fact, for any morphism $s$ in $\CT$, we have $\widehat{s}\in\overline{\SS_{\CN}}$ if and only if $s\in\SS_{\CN}$.
\end{remark}

\begin{lemma}\label{lem:equalF}
If $X,Z\in\CH$, then we have $\CK(Z,X)=0$. Thus, we may naturally identify $\overline{\BE_{\CN}}(Z,X)=\BE_{\CN}(Z,X)=\BF(Z,X)$ for all $X,Z\in\CH$.
\end{lemma}
\begin{proof}
Suppose that an element $\delta\in\BE_{\CN}(Z,X)$ belongs to $\CK(Z,X)$. By definition, there exist $A\in\CT$ and $a\in\SS_{\CN}(X,A)$ such that $a_{\ast}\delta=0$ in $\BE_{\CN}(Z,A)$.
Take $l_A\colon A\to A_+$ and $r_{(A_+)}\colon (A_+)_-\to A_+$ as in \cref{def:functor_+,def:functor_-}. 
By the dual of \cref{prop:functor_+} {\rm (1)}, there exists $a'\in\CT(X,(A_+)_-)$ such that $r_{(A_+)}\circ a'=l_A\circ a$.
By the $2$-out-of-$3$ property of $\SS_{\CN}$ shown in \cite[Corollary~2.18]{Oga24} and \cref{lem:fully_faithful1}, we obtain $a'\in\SS_{\CN}$.
Also, we have
\[
(r_{(A_+)})_{\ast}a'_{\ast}\delta=(l_A)_{\ast}a_{\ast}\delta=0
\]
in $\BE_{\CN}(Z,A_+)$.
Since $(r_{(A_+)})_{\ast}\colon\BE_{\CN}(Z,(A_+)_-)\to \BE_{\CN}(Z,A_+)$ is an isomorphism by \cref{prop:coreflection_tri1}, this implies $a'_{\ast}\delta=0$ in $\BE_{\CN}(Z,(A_+)_-)$.
From $a'\in\SS_{\CN}$, it follows that $H(a')$ is an isomorphism in $\CH/[\CW]$ by \cref{prop:fully_faithful}.
Since $Z,(A_+)_-\in\CH$, this means that $\overline{a'}$ is an isomorphism in $\CH/[\CW]$.
Thus we have
\[
\delta=(\overline{a'})^{-1}_{\ast}\overline{a'}_{\ast}\delta=0
\]
in $\BF(Z,X)$, and hence $\delta=0$ holds in $\BE_{\CN}(Z,X)$.
\end{proof}

For any pair of objects $A,C\in\Ob(\widetilde{\CT})=\Ob(\CT)$, we define $\widetilde{\BE_{\CN}}(C,A)$ to be the set of equivalence classes of triples
\[ (C\xleftarrow{\widehat{t}}C'\overset{\widehat{\delta}}{\dashrightarrow}A'\xleftarrow{\widehat{s}}A) \]
with $A',C'\in\Ob(\CT)$, $s,t\in\SS_{\CN}$ and $\delta\in\BE_{\CN}(C',A')$, where the equivalence relation is given as follows. See \cite[Proposition~3.19]{NOS22} for details.
\begin{itemize}
\item Triples
$(C\xleftarrow{\widehat{t_i}}C_i\overset{\widehat{\delta_i}}{\dashrightarrow}A_i\xleftarrow{\widehat{s_i}}A)$
with $A_i,C_i\in\Ob(\CT)$, $s_i,t_i\in\SS_{\CN}$ and $\delta_i\in\BE_{\CN}(C_i,A_i)$ for $i=1,2$, are equivalent if and only if there exists a triple
$(C\xleftarrow{\widehat{t}}C'\overset{\widehat{\delta}}{\dashrightarrow}A'\xleftarrow{\widehat{s}}A)$
with $A',C'\in\Ob(\CT)$, $s,t\in\SS_{\CN}$, $\delta\in\BE_{\CN}(C',A')$, and morphisms $c_i\in\SS_{\CN}(C',C_i)$, $a_i\in\SS_{\CN}(A_i,A')$ such that 
\[
\widehat{s}=\widehat{a_i}\circ\widehat{s_i},\quad
\widehat{t}=\widehat{t_i}\circ\widehat{c_i},\quad \text{and}\quad 
\widehat{\delta}=\widehat{a_i}_{\ast}\widehat{c_i}^{\ast}\widehat{\delta_i}
\]
for $i=1,2$.
This situation is illustrated by the following diagram.
\[
\begin{tikzpicture}[>=stealth]
\node (1) at (-1.2,0) {$C'$};
\node (2) at (-1.2,1.5) {$C_1$};
\node (3) at (-1.2,-1.5) {$C_2$};
\node (4) at (-3.4,0) {$C$};
\node (5) at (3.4,0) {$A$};
\node (6) at (1.2,1.5) {$A_1$};
\node (7) at (1.2,-1.5) {$A_2$};
\node (8) at (1.2,0) {$A'$};
\draw[->] (1) -- node[right,font=\scriptsize] {$\widehat{c_1}$} (2);
\draw[->] (1) -- node[right,font=\scriptsize] {$\widehat{c_2}$} (3);
\draw[->] (1) -- node[above,font=\scriptsize] {$\widehat{t}$} (4);
\draw[->] (2) -- node[above,font=\scriptsize] {$\widehat{t_1}$} (4);
\draw[->] (3) -- node[below,font=\scriptsize] {$\widehat{t_2}$} (4);
\draw[->, dashed] (2) -- node[above,font=\scriptsize] {$\widehat{\delta_1}$} (6);
\draw[->, dashed] (3) -- node[below,font=\scriptsize] {$\widehat{\delta_2}$} (7);
\draw[->, dashed] (1) -- node[above,font=\scriptsize] {$\widehat{\delta}$} (8);
\draw[->] (5) -- node[above,font=\scriptsize] {$\widehat{s_1}$} (6);
\draw[->] (5) -- node[below,font=\scriptsize] {$\widehat{s_2}$} (7);
\draw[->] (5) -- node[above,font=\scriptsize] {$\widehat{s}$} (8);
\draw[->] (6) -- node[left,font=\scriptsize] {$\widehat{a_1}$} (8);
\draw[->] (7) -- node[left,font=\scriptsize] {$\widehat{a_2}$} (8);
\end{tikzpicture}
\]
\end{itemize}
We denote the equivalence class of $(C\xleftarrow{\widehat{t}}C'\overset{\widehat{\delta}}{\dashrightarrow}A'\xleftarrow{\widehat{s}}A)$ by $[C\xleftarrow{\widehat{t}}C'\overset{\widehat{\delta}}{\dashrightarrow}A'\xleftarrow{\widehat{s}}A]$.
By definition, we have
\begin{equation}\label{ENT}
\widetilde{\BE_{\CN}}(C,A)=\Set{[C\xleftarrow{\widehat{t}}C'\overset{\widehat{\delta}}{\dashrightarrow}A'\xleftarrow{\widehat{s}}A]\,|\, \begin{array}{l}A',C'\in\Ob(\CT),\\ s,t\in\SS_{\CN},\\  \delta\in\BE_{\CN}(C',A')\end{array}}
\end{equation}

\begin{remark}
In \cite{NOS22}, we assumed that $\CT$ is small.
If $\CT$ is skeletally small with skeleton $\mathrm{sk}\CT$, then in \eqref{ENT} it suffices to consider $A',C'\in\mathrm{sk}\CT$. 
Hence $\widetilde{\BE_{\CN}}$ may be defined by considering only such representatives.
\end{remark}

\normalcolor

\begin{theorem}\label{thm:extri_heart}
The exact functor $(G,\psi)\colon (\CH/[\CW],\BF,\ft)\to(\widetilde{\CT},\widetilde{\BE_{\CN}},\widetilde{\fs_{\CN}})$ gives an equivalence of extriangulated categories.
\end{theorem}
\begin{proof}
By \cite[Proposition~2.13]{NOS22}, it suffices to show that $G$ is an equivalence of categories \darkblueedit{and that} $\psi$ is a natural isomorphism.
We have $\widetilde{H}\circ G\circ \textcolor{blue}{\pi}\cong \textcolor{blue}{\pi}$, and hence $\widetilde{H}\circ G\cong \id$. This means that $G$ is a quasi-inverse of the equivalence $\widetilde{H}$. In particular, the functor $G$ itself is an equivalence of categories.

\color{blue}
It suffices 
\normalcolor
to show that $\psi_{Z,X}\colon \BF(Z,X)\to \widetilde{\BE_{\CN}}(Z,X)$ \darkblueedit{is an isomorphism for any} $X,Z\in\CH/[\CW]$.
First, we show the surjectivity. 
\color{blue}
Let $\alpha=[Z\xleftarrow{\widehat{t}}C\overset{\widehat{\delta}}{\dashrightarrow}A\xleftarrow{\widehat{s}}X]\in\widetilde{\BE_{\CN}}(Z,X)$ be any element, with $s,t\in\SS_{\CN}$ and $\delta\in\BE_{\CN}(C,A)$.
Let $l_A\colon A\to A_+$ and $r_C\colon C_-\to C$ be as in \cref{def:functor_+,def:functor_-}, and put
$\delta'=(r_C)^{\ast}(l_A)_{\ast}\delta\in\BE_{\CN}(C_-,A_+)$.
Then, take $l_{(C_-)}\colon C_-\to (C_-)_+$ and $r_{(A_+)}\colon (A_+)_-\to A_+$ again as in \cref{def:functor_+,def:functor_-}. 
By \cref{prop:functor_+_property} and its dual, we have $(C_-)_+,(A_+)_-\in\CH$.
By \cref{prop:coreflection_tri1},
\[
(l_{(C_-)})^{\ast}\colon\BE_{\CN}((C_-)_+,A_+)\to\BE_{\CN}(C_-,A_+)
\quad\text{and}\quad
(r_{(A_+)})_{\ast}\colon\BE_{\CN}((C_-)_+,(A_+)_-)\to\BE_{\CN}((C_-)_+,A_+)
\]
are isomorphisms. Thus there exists $\delta^{\prime\prime}\in\BE_{\CN}((C_-)_+,(A_+)_-)$ 
such that $\delta'=(l_{(C_-)})^{\ast}(r_{(A_+)})_{\ast}\delta^{\prime\prime}$ in $\BE_{\CN}(C_-,A_+)$. By \cref{prop:functor_+} {\rm (1)} and its dual, there exist $t'\in\CT((C_-)_+,Z)$ and $s'\in\CT(X,(A_+)_-)$ such that $t'\circ l_{(C_-)}=t\circ r_C$ and $r_{(A_+)}\circ s'=l_A\circ s$.
By the $2$-out-of-$3$ property of $\SS_{\CN}$,
it follows that $s',t'\in\SS_{\CN}$.
Thus we obtain the following diagram.
\[
\begin{tikzpicture}[>=stealth]
\node (1) at (-1.2,0) {$C_-$};
\node (2) at (-1.2,1.5) {$C$};
\node (3) at (-1.2,-1.5) {$(C_-)_+$};
\node (4) at (-3.4,0) {$Z$};
\node (5) at (3.4,0) {$X$};
\node (6) at (1.2,1.5) {$A$};
\node (7) at (1.2,-1.5) {$(A_+)_-$};
\node (8) at (1.2,0) {$A_+$};
\draw[->] (1) -- node[right,font=\scriptsize] {$\widehat{r_C}$} (2);
\draw[->] (1) -- node[right,font=\scriptsize] {$\widehat{l_{(C_-)}}$} (3);
\draw[->] (1) -- node[above,font=\scriptsize] {$\widehat{t\circ r_C}$} (4);
\draw[->] (2) -- node[above,font=\scriptsize] {$\widehat{t}$} (4);
\draw[->] (3) -- node[below,font=\scriptsize] {$\widehat{t'}$} (4);
\draw[->, dashed] (2) -- node[above,font=\scriptsize] {$\widehat{\delta}$} (6);
\draw[->, dashed] (3) -- node[below,font=\scriptsize] {$\widehat{\delta^{\prime\prime}}$} (7);
\draw[->, dashed] (1) -- node[above,font=\scriptsize] {$\widehat{\delta'}$} (8);
\draw[->] (5) -- node[above,font=\scriptsize] {$\widehat{s}$} (6);
\draw[->] (5) -- node[below,font=\scriptsize] {$\widehat{s'}$} (7);
\draw[->] (5) -- node[above,font=\scriptsize] {$\widehat{l_A\circ s}$} (8);
\draw[->] (6) -- node[left,font=\scriptsize] {$\widehat{l_A}$} (8);
\draw[->] (7) -- node[left,font=\scriptsize] {$\widehat{r_{(A_+)}}$} (8);
\end{tikzpicture}
\]
\darkblueedit{Then we have}
\begin{equation}\label{eqalpha}
\alpha=[Z\xleftarrow{\widehat{t}}C\overset{\widehat{\delta}}{\dashrightarrow}A\xleftarrow{\widehat{s}}X]
=[Z\xleftarrow{\widehat{t'}}(C_-)_+\overset{\widehat{\delta^{\prime\prime}}}{\dashrightarrow}(A_+)_-\xleftarrow{\widehat{s'}}X]
\end{equation}
in $\widetilde{\BE_{\CN}}(Z,X)$.
By \cref{prop:fully_faithful}, we see that $H(s')$ and $H(t')$ are isomorphisms in $\CH/[\CW]$. Since 
$s'$ and $t'$ are morphisms in $\CH$,
this means that $\overline{s'}$ and $\overline{t'}$ are isomorphisms in $\CH/[\CW]$.
Take morphisms $a\in\CH((A_+)_-,X)$ and $b\in\CH(Z,(C_-)_+)$ that give $\overline{a}=(\overline{s'})^{-1}$ and $\overline{b}=(\overline{t'})^{-1}$ in $\CH/[\CW]$. Put $\rho=a_{\ast}b^{\ast}\delta^{\prime\prime}\in\BE_{\CN}(Z,X)$. Then it satisfies
\[
s'_{\ast}t^{\prime\ast}\rho=(s'\circ a)_{\ast}(b\circ t')^{\ast}\delta^{\prime\prime}=\id_{\ast}\id^{\ast}\delta^{\prime\prime}=\delta^{\prime\prime}
\]
in $\BE_{\CN}((C_-)_+,(A_+)_-)$. Indeed, this follows from 
\[
s'\circ a-\id,\ b\circ t'-\id\in[\CW]
\]
and the fact that $\CW$ consists of projective-injective objects in $(\CH,\BE_{\CN}|_{\CH},\fs_{\CN}|_{\CH})$. 
Thus we obtain
\[
[Z\xleftarrow{\widehat{t'}}(C_-)_+\overset{\widehat{\delta^{\prime\prime}}}{\dashrightarrow}(A_+)_-\xleftarrow{\widehat{s'}}X]=[Z\xleftarrow{\id}Z\overset{\widehat{\rho}}{\dashrightarrow}X\xleftarrow{\id}X]
=\psi_{Z,X}(\rho),
\]
hence we have $\alpha=\psi_{Z,X}(\rho)$ by \eqref{eqalpha}.
This shows the surjectivity of $\psi_{Z,X}$.

It remains to show the injectivity of $\psi_{Z,X}$. 
\color{blue}
Suppose that $\rho\in\BF(Z,X)$ satisfies
$\psi_{Z,X}(\rho)=0$ in $\widetilde{\BE_{\CN}}(Z,X)$.
This means that there exists a diagram
\[
\begin{tikzpicture}[>=stealth]
\node (1) at (-1.2,0) {$C$};
\node (2) at (-1.2,1.5) {$Z$};
\node (3) at (-1.2,-1.5) {$Z$};
\node (4) at (-3.4,0) {$Z$};
\node (5) at (3.4,0) {$X$};
\node (6) at (1.2,1.5) {$X$};
\node (7) at (1.2,-1.5) {$X$};
\node (8) at (1.2,0) {$A$};
\draw[->] (1) -- node[right,font=\scriptsize] {$\widehat{z_1}$} (2);
\draw[->] (1) -- node[right,font=\scriptsize] {$\widehat{z_2}$} (3);
\draw[->] (1) -- node[above,font=\scriptsize] {$\widehat{t}$} (4);
\draw[->] (2) -- node[above,font=\scriptsize] {$\id$} (4);
\draw[->] (3) -- node[below,font=\scriptsize] {$\id$} (4);
\draw[->, dashed] (2) -- node[above,font=\scriptsize] {$\widehat{\rho}$} (6);
\draw[->, dashed] (3) -- node[below,font=\scriptsize] {$0$} (7);
\draw[->, dashed] (1) -- node[above,font=\scriptsize] {$\widehat{\delta}$} (8);
\draw[->] (5) -- node[above,font=\scriptsize] {$\id$} (6);
\draw[->] (5) -- node[below,font=\scriptsize] {$\id$} (7);
\draw[->] (5) -- node[above,font=\scriptsize] {$\widehat{s}$} (8);
\draw[->] (6) -- node[left,font=\scriptsize] {$\widehat{x_1}$} (8);
\draw[->] (7) -- node[left,font=\scriptsize] {$\widehat{x_2}$} (8);
\end{tikzpicture}
\]
with $A,C\in\CT$, $s,t\in\SS_{\CN}$, $\delta\in\BE_{\CN}(C,A)$ and $x_1,x_2,z_1,z_2\in\SS_{\CN}$.
By the commutativity, we have $\widehat{x_1}=\widehat{x_2}=\widehat{s}$ and $\widehat{z_1}=\widehat{z_2}=\widehat{t}$ in $\CT/[\CN]$. Also, we have
\[ \widehat{\delta}=\widehat{x_2}_{\ast}\widehat{z_2}^{\ast}0=0 
\ \  \text{and hence}\quad
\widehat{x_1}_{\ast}\widehat{z_1}^{\ast}\widehat{\rho}=0
\]
in $\overline{\BE_{\CN}}(C,A)$. Since $x_1,z_1\in\SS_{\CN}$, this implies $\widehat{\rho}=0$ in $\overline{\BE_{\CN}}(Z,X)$ by the definition of $\CK(Z,X)$, as noted in \cite[Remark~3.17]{NOS22}.
By \cref{lem:equalF}, this means that $\rho=0$ holds in $\BE_{\CN}(Z,X)=\BF(Z,X)$ since $X,Z\in\CH$. Thus the injectivity of $\psi_{Z,X}$ is shown.

\normalcolor
\end{proof}
\normalcolor

\medskip
\noindent
{\bf Acknowledgement.}
N.M. %
is financially supported by JST SPRING, Grant Number JPMJSP2125.
H.N. is supported by JSPS KAKENHI Grant Numbers JP24K06645, JP24KK0250.
Y.O. is supported by JSPS KAKENHI (grant JP22K13893).

\bibliographystyle{mybstwithlabels}
\bibliography{references}

\end{document}